\input ./style/arxiv-general.cfg
\documentclass[aop,MSNbibl,dvips]{arximspdf}
\makeatletter
   \@ifpackageloaded{graphicx}{}{\usepackage{graphicx}}
\makeatother
\usepackage{mathrsfs}
\usepackage{url,breakurl}

\doi{10.1214/14-AOP919} %kopijuoti is PTS
\volume{43}
\issue{4}
\pubyear{2015}
\firstpage{1777}
\lastpage{1822}
\makeatletter
\newcommand{\rrVert}{\Vert}
\newcommand{\rrvert}{\vert}
\newcommand{\llVert}{\Vert}
\newcommand{\llvert}{\vert}
\newtheorem{teo}{Theorem}
\newproclaim{defn}[teo]{Definition}
\newtheorem{lem}[teo]{Lemma}
\newtheorem{prop}[teo]{Proposition}
\newtheorem{cor}[teo]{Corollary}
\newproclaim{rem}[teo]{Remark}
\makeatother

\begin{document}
\begin{frontmatter}

\title{On a functional contraction method}
\runtitle{On a functional contraction method}

\begin{aug}
\author[A]{\fnms{Ralph} \snm{Neininger}\corref{}\ead[label=e1]{neiningr@math.uni-frankfurt.de}}
\and
\author[A]{\fnms{Henning} \snm{Sulzbach}\ead[label=e2]{henning.sulzbach@gmail.com}}
\runauthor{R. Neininger and H. Sulzbach}
\affiliation{Goethe University Frankfurt}
%\dedicated{}
\address[A]{Institute for Mathematics\\
Goethe University Frankfurt\\
60054 Frankfurt am Main\\
Germany\\
\printead{e1}\\
\phantom{E-mail:\ }\printead*{e2}} %adresu isvedimo komanda gale!
%\address[B]{}
\end{aug}

% HISTORY:
\received{\smonth{4} \syear{2013}}
\revised{\smonth{2} \syear{2014}}
%\accepted{\smonth{} \syear{}}

% ABSTRACT
%
\begin{abstract}
Methods for proving functional limit laws are developed for sequences
of stochastic processes which allow a recursive distributional
decomposition either in time or space. Our approach is an extension of
the so-called contraction method to the space $\mathcal{C}[0,1]$ of continuous
functions endowed with uniform topology and the space $\mathcal
{D}[0,1]$ of c\`
{a}dl\`{a}g functions with the Skorokhod topology. The contraction
method originated from the probabilistic analysis of algorithms and
random trees where characteristics satisfy natural distributional
recurrences. It is based on stochastic fixed-point equations, where
probability metrics can be used to obtain contraction properties and
allow the application of Banach's fixed-point theorem. We develop the
use of the Zolotarev metrics on the spaces $\mathcal{C}[0,1]$ and
$\mathcal{D}[0,1]$ in this
context. Applications are given, in particular, a short proof of
Donsker's functional limit theorem is derived and recurrences arising
in the probabilistic analysis of algorithms are discussed.
\end{abstract}

% KEYWORDS

\end{frontmatter}

%s1 #&#
\section{Introduction} \label{secintro1}
The contraction method is an approach for proving convergence in
distribution for sequences of random variables which satisfy recurrence
relations in distribution. Such recurrence relations for a sequence
$(Y_n)_{n\ge0}$ are often of the form
%
%e1 #&#
%
\begin{equation}
\label{rec1} Y_n \stackrel{d} {=} \sum
_{r=1}^K A_r(n) Y_{I_r^{(n)}}^{(r)}
+ b(n),\qquad n \geq n_0,
\end{equation}
where $\stackrel{d}{=}$ denotes that the left-hand side and right-hand
side are identically distributed, and $(Y_j^{(r)})_{j\ge0}$ have the
same distribution as $(Y_n)_{n\ge0}$ for all $r=1,\ldots, K$, where
$K\ge1$ and $n_0\ge0$ are fixed integers. Moreover,
$I^{(n)}=(I_1^{(n)},\ldots,I_K^{(n)})$ is a vector of random integers
in $\{0,\ldots,n\}$. The\vspace*{1pt} basic independence assumption that fixes the
distribution of the right-hand side is that $(Y_j^{(1)})_{j\ge0},\ldots, (Y_j^{(K)})_{j\ge0}$ and $(A_1(n),\ldots,A_K(n),b(n),I^{(n)})$
are independent. Note, however, that dependencies between the
coefficients $A_r(n)$, $b(n)$ and the integers $I_r^{(n)}$ are allowed.

Recurrences of the form~(\ref{rec1}) come up in diverse fields, for
example, in the study of random trees, the probabilistic analysis of
recursive algorithms, in branching processes, in the context of random
fractals and in models from stochastic geometry where a recursive
decomposition can be found, as well as in information and coding
theory. For surveys of such occurrences, see \cite{RoRu01,NeRu04,Ne04}.
In some applications, one may need $K$ to depend on $n$ or the
case $K=\infty$, where generalizations of the results for our case of
fixed $K$ can be stated; cf. \cite{NeRu04}, Section~4.3, for such
extensions in the finite-dimensional case.

The sequence $(Y_n)_{n\ge0}$ satisfying~(\ref{rec1}) often is a
sequence of real random variables with real coefficients $A_r(n)$,
$b(n)$. However, the same recurrence appears also for sequences of
random vectors $(Y_n)_{n\ge0}$ in $\mathbb{R}^d$. Then the $A_r(n)$ are
random linear maps from $\mathbb{R}^d$ to $\mathbb{R}^d$ and $b(n)$
is a random
vector in $\mathbb{R}^d$. We will also review below work that considered
random sequences $(Y_n)_{n\ge0}$ into a separable Hilbert space
satisfying~(\ref{rec1}) where $A_r(n)$ become random linear operators
on the space and $b(n)$ a random vector in the Hilbert space. In the
present work, we develop a limit theory for such sequences in separable
Banach spaces, where our main applications are first to the space
$\mathcal{C}[0,1]$
endowed with the uniform topology. Secondly, although not a Banach
space, we will also be able to cover the space
$\mathcal{D}[0,1]$ equipped with the Skorokhod topology. Hence, we consider
sequences $(Y_n)_{n\ge0}$ of stochastic processes with state space
$\mathbb{R}$ and time parameter $t\in[0,1]$ with continuous, respectively,
c{\'a}dl{\'a}g paths and are interested in conditions that together with~(\ref{rec1}) allow to deduce functional limit theorems for rescaled
versions of $(Y_n)_{n\ge0}$.

For functions $f\in\mathcal{C}[0,1]$ or $f\in\mathcal{D}[0,1]$, we
denote the uniform norm by
\[
\|f\|_\infty:= \sup_{x\in[0,1]}\bigl|f(x)\bigr|.
\]
For functions $f,g \in\mathcal{D}[0,1]$, the Skorokhod distance
$d_{\mathrm{sk}}(f,g)$ is
used; see Section~\ref{subsecdsk}.

The rescaling of the process $(Y_n)_{n\ge0}$ can be done by centering
and normalization by the order of the standard deviation in case
moments of sufficient order are available. Subsequently, we assume that
the scaling has already been done and we denote the scaled process by
$(X_n)_{n\ge0}$. Note that affine scalings of the $Y_n$ implies that
the sequence $(X_n)_{n\ge0}$ also does satisfy a recurrence of type
(\ref{rec1}), where only the coefficients are changed:
%
%e2 #&#
%
\begin{equation}
\label{rec2} X_n \stackrel{d} {=} \sum
_{r=1}^K A_r^{(n)}
X_{I_r^{(n)}}^{(r)} + b^{(n)}, \qquad n \geq
n_0
\end{equation}
with conditions on identical distributions and independence similar to
recurrence~(\ref{rec1}). The coefficients $A_r^{(n)}$ and $b^{(n)}$ in
the modified recurrence~(\ref{rec2}) are typically directly computable
from the original coefficients $A_r(n)$, $b(n)$ and the \mbox{scaling} used;
see, for example, for the case of random vectors in $\mathbb{R}^d$,
\cite{NeRu04}, equation~(4).
Subsequently, we consider equations of type~(\ref{rec2}) together with
assumptions on the moments of $X_n$ which in applications have to be
obtained by an appropriate scaling.

For the asymptotic distributional analysis of sequences $(X_n)_{n\ge
0}$ satisfying~(\ref{rec2}), the so-called contraction method has
become a powerful tool. In the seminal paper~\cite{Ro91}, R\"osler
introduced this methodology for deriving a limit law for a special
instant of this equation that arises in the analysis of the complexity
of the Quicksort algorithm. In the framework of the contraction method,
first one derives limits of the coefficients $A_r^{(n)}$, $b^{(n)}$,
%
%e3 #&#
%
\begin{equation}
\label{coeconv} A_r^{(n)} \to A_r, \qquad
b^{(n)} \to b \qquad(n \to\infty)
\end{equation}
in an appropriate sense. If with $n\to\infty$, also the $I^{(n)}_r$
become large and it is plausible that the quantities $X_n$ converge,
say to a random variable $X$; then, by letting formally $n\to\infty$,
equation~(\ref{rec2}) turns into
%
%e4 #&#
%
\begin{equation}
\label{fix2} X \stackrel{d} {=} \sum_{r=1}^K
A_r X^{(r)} + b
\end{equation}
with $X^{(1)},\ldots,X^{(K)}$ distributed as $X$ and $X^{(1)},\ldots,X^{(K)}$, $(A_1,\ldots,A_k,b)$ independent. Hence, one can use the
distributional fixed-point equation~(\ref{fix2}) to characterize the
limit distribution ${\mathcal L}(X)$. The idea from R\"osler \cite
{Ro91} to formalize such an approach and to derive at least weak
convergence $X_n \to X$ consists of first using the right-hand side of
(\ref{fix2}) to define a map as follows: if $X_n$ are $B$-valued random
variables, denote by $\mathcal{M}(B)$ the space of all probability
measures on $B$ and
%
%e5 #&#
%e6 #&#
%
\begin{eqnarray}
\label{limitmap}
&\displaystyle T \dvtx \mathcal{M}(B) \to\mathcal{M}(B),&
\\
&\displaystyle T(\mu) = \mathcal{L} \Biggl( \sum_{r=1}^K
A_r Z^{(r)} + b \Biggr),&
\end{eqnarray}
where $(A_1, \ldots, A_K,b), Z^{(1)}, \ldots, Z^{(K)}$ are independent
and $Z^{(1)}, \ldots, Z^{(K)}$ have distribution $\mu$. Then a random
variable $X$ solves~(\ref{fix2}) if and only if its distribution
${\mathcal L}(X)$ is a fixed point of the map $T$. To obtain fixed
points of $T$ appropriate subspaces of $\mathcal{M}(B)$ are endowed
with a complete metric, such that the restriction of~$T$ becomes a
contraction. Then Banach's fixed-point theorem yields a (in the
subspace) unique fixed point of $T$ and one can as well use the metric
to also derive convergence
of ${\mathcal L}(X_n)$ to ${\mathcal L}(X)$ in this metric. If the
metric is also strong enough to imply weak convergence, one has
obtained the desired limit law $X_n \to X$.

This approach has been established and applied to a couple of examples
in R\"osler \cite{Ro91,Ro92} and Rachev and R\"uschendorf \cite
{RaRu95}. In the latter paper also the flexibility of the approach by
using various probability metrics has been demonstrated. Later on
general convergence theorems have been derived stating conditions under
which convergence of the coefficients of the form~(\ref{coeconv})
together with a contraction property of the map~(\ref{limitmap})
implies convergence in distribution $X_n \to X$. For random variables
in $\mathbb{R}$ with the minimal $\ell_2$ metric, see R\"osler \cite{Ro99},
and Neininger \cite{Ne01} for $\mathbb{R}^d$ with the same metric.
For a
more widely applicable framework for random variables in $\mathbb
{R}^d$, see
Neininger and R\"uschendorf \cite{NeRu04}, where in particular various
problems with normal limit laws could be solved which seem to be beyond
the scope of the minimal $\ell_p$ metric; see also \cite{NeRu04b}. An
extension of these theorems to continuous time, that is, to processes
$(X_t)_{t\ge0}$ satisfying recurrences similar to~(\ref{rec2}) was
given in Janson and Neininger \cite{JaNe08}.

For the case of random variables in a separable Hilbert space leading
to functional limit laws, general limit theorems for recurrences (\ref
{rec1}) have been developed in Drmota, Janson and Neininger \cite
{DrJaNe08}. The main application there was a functional limit law for
the profile of random trees which, via a certain encoding of the
profile, led to random variables in the Bergman space of square
integrable analytic functions on a domain in the complex plane.
In Eickmeyer and R\"uschendorf~\cite{EiRu07}, \mbox{general} limit theorems
for recurrences in $\mathcal{D}[0,1]$ under the $L_p$-topology were
developed. Note
that the uniform topology for $\mathcal{C}[0,1]$ and the Skorokhod
topology for $\mathcal{D}[0,1]
$ considered in the present paper are finer than the $L_p$-topology. In
$\mathcal{C}[0,1]$, the uniform topology provides more continuous
functionals such
as the supremum $f \mapsto\sup_{t\in[0,1]} f(t)$ or projections $f
\mapsto f(s_1,\ldots,s_k)$, for fixed $s_1,\ldots,s_k\in[0,1]$, to
which the continuous mapping theorem can be applied. In $\mathcal
{D}[0,1]$, these
functionals are also appropriate for the continuous mapping theorem if
the limit random variable has continuous sample paths.

Besides the minimal $\ell_p$ metrics the probability metrics that have
proved useful in most of the papers mentioned above is the family of
Zolotarev metrics $\zeta_s$ being reviewed and further developed here
in Section~\ref{zolosec}. All generalizations from $\mathbb{R}$ via
$\mathbb{R}
^d$ to separable Hilbert spaces are based on the fact that convergence
in $\zeta_s$ implies weak convergence; see Section~\ref{zolosec}.
However, for Banach spaces this is not true in general. Counterexamples
have been reported in Bentkus and Rachkauskas~\cite{BeRa84}, sketched
here in Section~\ref{zolobasic}. Also completeness of the $\zeta_s$
metrics on appropriate subspaces of ${\mathcal M}(B)$ is only known for
the case of separable Hilbert spaces; see~\cite{DrJaNe08}, Theorem 5.1.

Our study of the spaces $(\mathcal{C}[0,1], \| \cdot\|_\infty)$
and $(\mathcal{D}[0,1],
d_{\mathrm{sk}})$ is also based on
the Zolotarev metrics $\zeta_s$. Hence, we mainly have to deal with
implications that can be drawn from convergence
in the $\zeta_s$ metrics as well as with the lack of knowledge
about completeness of $\zeta_s$. In Section~\ref{subsecweak},
implications of convergence in the Zolotarev metric are discussed
together with additional conditions that enable to deduce in general
weak convergence from convergence in $\zeta_s$. A key ingredient here
is a technique developed in Barbour
\cite{Barbour90} in the context of Stein's method; see also Barbour and
Janson \cite{BaJa09}.
We also obtain criteria for the uniform integrability of $\{\|X_n\|
_\infty^s | n\ge0\}$
for $0\le s\le3$ in the presence of convergence in the Zolotarev
metric. This enables in applications as well to obtain moments
convergence of the $\sup$-functional.

In Section~\ref{nnnCM}, we give general convergence theorems in the
framework of the contraction method first for a general separable
Banach space and then apply and refine this to the space $(\mathcal
{C}[0,1], \|
\cdot\|_\infty)$ and develop a technique to also apply this to the
metric space $(\mathcal{D}[0,1], d_{\mathrm{sk}})$.
In particular, based on Janson and Kaijser \cite{jankai12}, we give a
criterion for the finiteness of the Zolotarev metric on appropriate
subspaces that can easily be checked in applications.

To compensate for the lack of knowledge about completeness of the
$\zeta
_s$ metrics, we need to
assume that the map $T$ in~(\ref{limitmap}) has a fixed point in an
appropriate subspace of ${\mathcal M}(\mathcal{C}[0,1])$ and
${\mathcal M}(\mathcal{D}[0,1])$,
respectively. In applications, one may verify this existence of a fixed
point either by guessing one successfully:
in the application of our framework to Donsker's functional limit
theorem in Section~\ref{secapp2}, the Wiener measure can easily be
guessed and be seen to be the fixed point of the map $T$ coming up there.
Alternatively, in general the existence of a fixed point may arise from
infinite iteration of the map $T$: applied to some probability measure,
such an iteration has a series representation for which one may be able
to show that it is the desired fixed point.
This path is being taken in an application of our framework outlined in
Section~\ref{secpm}.

In Section~\ref{secapp2}, we apply our functional contraction method
to derive a short proof of Donsker's functional limit theorem. This
does not require the full generality of our setting but illustrates how
self-similarities can easily been exploited with this approach. The
application in Section~\ref{secpm} is on the asymptotic study of
fundamental complexities in computer science. Here, the full generality
of our approach is needed to obtain a functional limit law. We
highlight and discuss the use of our conditions (C1)--(C5)
formulated in Section~\ref{nnnCM} on the recurrence~(\ref{rec2}) at
this example. Details on the verification of the conditions are
contained in Broutin, Neininger and Sulzbach \cite{BrNeSu11} where,
based on the functional limit law, also various long open standing
problems on the complexities in computer science are solved.

%%
%%
%%
%%
%%
%%

%s2 #&#
\section{The Zolotarev metric} \label{zolosec}

Let $(B, \| \cdot\|)$ be a real Banach space and $\mathcal{B}$ its Borel
$\sigma$-algebra. In Section~\ref{zolobasic}, we assume that the norm on $B$
induces a separable topology. We denote by $\mathcal{M}(B)$ the set of
all probability measures on $(B,\mathcal{B})$.
First, we introduce the Zolotarev metric $\zeta_s$ and collect some of
its basic properties, mainly covered in \cite{Zolo76,Zolo78}. In the
second subsection, we define our use of the Zolotarev metrics on the
metric space $(\mathcal{D}[0,1], d_{\mathrm{sk}})$. Although not a Banach
space, we will be
able to declare the Zolotarev metrics $\zeta_s$ on $(\mathcal
{D}[0,1], d_{\mathrm{sk}})$
using the notion of differentiability of functions $\mathcal
{D}[0,1]\to\mathbb{R}$ induced
by the supremum norm on $\mathcal{D}[0,1]$. We also comment in Remarks
\ref{remmeas} and~\ref{remsko} on delicate measurability issues for the
nonseparable Banach space ${(B, \| \cdot\|)=(\mathcal{D}[0,1], \|
\cdot\|
_\infty)}$ and the realm of our methodology when working with the
coarser (separable) topology on $\mathcal{D}[0,1]$ induced by the
Skorokhod metric.
In the third subsection,
conditions that allow to conclude from convergence in $\zeta_s$ to weak
convergence are studied for the case $(B, \| \cdot\|)=(\mathcal
{C}[0,1],\| \cdot
\|_\infty)$ as well as for the case $(\mathcal{D}[0,1], d_{\mathrm{sk}})$.
We also discuss
further implications from $\zeta_s$-convergence in these two spaces as
well as criteria for finiteness of $\zeta_s$. Additional material to
the content of this section can be found in the second author's
dissertation \cite{SuDiss}, Chapter~2.

%s2.1 #&#
\subsection{Definition and basic properties} \label{zolobasic}
For functions $f\dvtx  B \to\mathbb{R}$, which are Fr\'echet
differentiable, the
derivative of $f$ at a point $x$ is denoted by $Df(x)$. Note that
$Df(x)$ is an element of the space $L(B, \mathbb{R})$ of continuous linear
forms on $B$. We also consider higher order derivatives, where $D^m
f(x)$ denotes the $m$th derivative of $f$ at a point $x$. Thus, $D^m
f(x)$ is a continuous $m$-linear (or multilinear) form on $B$. The
space of continuous multilinear forms $g\dvtx  B^m \to\mathbb{R}$ is
equipped with
the norm
\[
\llVert g\rrVert= \sup_{ \llVert h_1\rrVert\leq1, \ldots, \llVert
h_m\rrVert\leq1} \bigl|g(h_1,
\ldots,h_m)\bigr|.
\]
For a comprehensive account on differentiability in Banach spaces, we
refer to Cartan \cite{cartan1971}.
Subsequently, $s >0$ is fixed and for $m:= \lceil s \rceil- 1$ and
$\alpha:= s-m$ we define
%
%e7 #&#
%
\begin{equation}
\label{deffs} \mathcal{F}_s = \bigl\{f\dvtx B \to\mathbb{R}\dvtx
\bigl\llVert D^m f(x) - D^m f(y) \bigr\rrVert\leq\llVert
x-y\rrVert ^{\alpha},\ \forall x,y \in B \bigr\}.
\end{equation}
For $\mu, \nu\in\mathcal{M}(B)$, the Zolotarev distance between
$\mu$
and $\nu$ is defined by
%
%e8 #&#
%
\begin{equation}
\label{defzo} \zeta_s(\mu, \nu) = \sup_{f \in\mathcal{F}_s} \bigl|
\mathbf{E} \bigl[f(X)-f(Y) \bigr] \bigr|,
\end{equation}
where $X$ and $Y$ are $B$-valued random variables with $\mathcal{L}(X)
= \mu$
and $\mathcal{L}(Y) = \nu$. Here, $\mathcal{L}(X)$ denotes the
distribution of the
random variable $X$.
The expression in~(\ref{defzo}) does not need to be finite or even well
defined. However, we have \mbox{$\zeta_s(\mu, \nu)<\infty$} if
%
%e9 #&#
%
\begin{equation}
\label{eqmom1} \int\llVert x\rrVert^s\, d \mu(x), \int\llVert x
\rrVert^s\, d \nu(x) < \infty
\end{equation}
and
%
%e10 #&#
%
\begin{equation}
\label{eqmom2} \int f(x, \ldots, x) \,d \mu(x) = \int f(x, \ldots, x) \,d \nu(x)
\end{equation}
for any bounded $k$-linear form $f$ on $B$ and any $1 \leq k \leq m$.
For random variables $X$, $Y$ in $B$, we use the abbreviation $\zeta
_s(X,Y):= \zeta_s(\mathcal{L}(X), \mathcal{L}(Y))$.
% \par
Finiteness of $\zeta_s(X,Y)$ in $\mathbb{R}^d$ fails to hold if $X$
and $Y$ do
not have the same mixed moments up to order $m$. The assumption on the
finite absolute moment of order $s$ can be relaxed slightly; see
Theorem 4 in \cite{Zolo77}.

We denote
\[
\mathcal{M}_s(B):= \biggl\{ \mu\in\mathcal{M}(B) \Big| \int\llVert x
\rrVert^s\, d \mu(x) < \infty \biggr\}
\]
and for all $\nu\in\mathcal{M}_s(B)$ denote
\[
\mathcal{M}_s(\nu):= \bigl\{ \mu\in\mathcal{M}_s(B) |
\mu \mbox{ and }\nu\mbox{ satisfy~(\ref{eqmom2})} \bigr\}.
\]
Then $\zeta_s$ is a metric on the space ${\mathcal M}_s(\nu)$ for any
$\nu\in\mathcal{M}_s(B)$; see \cite{Zolo79}, Remark~1, page~198.

A crucial property of $\zeta_s$ in the context of recursive
decompositions of stochastic processes is the following lemma; see
Theorem 3 in \cite{Zolo77}. A short proof is given for the reader's
convenience.

%le1 #&#
%
\begin{lem} \label{lemtrans}
Let $B'$ be a Banach space and $g\dvtx  B \to B'$ a linear and continuous
operator. Then we have
\[
\zeta_s \bigl(g(X),g(Y) \bigr) \leq\llVert g\rrVert^s
\zeta_s(X,Y), \qquad{\mathcal L}(X), {\mathcal L}(Y) \in{\mathcal
M}_s(\nu).
\]
Here, $\| g \|$ denotes the operator norm of $g$, that is, $\| g\| =
\sup_{x \in B, \|x \| \leq1} \| g(x) \|$.
\end{lem}

\begin{pf} Note that $g$ is also bounded. It suffices to show that
\[
\bigl\{ \llVert g\rrVert^{-s} f \circ g\dvtx f \in
\mathcal{F}'_s \bigr\}\subseteq\mathcal{F}_s,
\]
where $\mathcal{F}'_s$ is defined analogously to $\mathcal{F}_s$ in $B'$.
Let $f \in\mathcal{F}_s$ and $\eta:= \llVert g\rrVert^{-s} f \circ
g$. Then
$\eta$ is $m$-times continuously differentiable and we have $D^m\eta(x)
= \llVert g\rrVert^{-s} (D^m(f(g(x))) \circ g^{\otimes m}$ for $x \in
B$. Here, $g^{\otimes m}\dvtx  B^m \to(B')^m$ denotes the mapping
$g^{\otimes m}(h_1, \ldots,h_m) = (g(h_1), \ldots, g(h_m))$. This implies
\begin{eqnarray*}
\bigl\|D^m\eta(x) - D^m\eta(y)\bigr\| & = & \llVert g\rrVert
^{-s} \bigl\| \bigl(D^mf \bigl(g(x) \bigr) \bigr) \circ
g^{\otimes m} - \bigl(D^mf \bigl(g(y) \bigr) \bigr) \circ
g^{\otimes m} \bigr\|
\\
& \leq& \llVert g\rrVert^{-\alpha} \bigl\|g(x) - g(y)\bigr\|^{\alpha}
\\
& = & \llVert g\rrVert^{-\alpha} \bigl\|g(x-y)\bigr\|^{\alpha} \leq\|x-y\|
^{\alpha}.
\end{eqnarray*}
The assertion follows.
\end{pf}

%The linearity of the operators $g_u(x) = x + u$ for $u \in B$ and
%$h_c(x) = cx$ for $c \in\R\setminus\{0\}$ yields the following
%Lemma.
Another basic property is that $\zeta_s$ is $(s,+)$ ideal.

%le2 #&#
%
\begin{lem} \label{lemideal}
The metric $\zeta_s$ is ideal of order $s$ on ${\mathcal M}_s(\nu)$ for
any $\nu\in\mathcal{M}_s(B)$, that is, we have
\begin{eqnarray*}
\zeta_s(cX, cY) & = & |c|^s \zeta_s(X,Y),
\\
\zeta_s(X + Z, Y+ Z) & \leq& \zeta_s(X,Y)
\end{eqnarray*}
for any $c \in\mathbb{R}\setminus\{0\}$, ${\mathcal L}(X),
{\mathcal
L}(Y)\in{\mathcal M}_s(\nu)$ and random variables $Z$ in $B$, such
that $(X,Y)$ and $Z$ are independent.
\end{lem}

The lemma directly implies
%
%e11 #&#
%
\begin{equation}
\label{ungldo} \zeta_s(X_1 + X_2,
Y_1 + Y_2) \leq\zeta_s(X_1,
Y_1) + \zeta_s(X_2, Y_2)
\end{equation}
for ${\mathcal L}(X_1), {\mathcal L}(Y_1) \in{\mathcal M}_s(\nu_1)$
and ${\mathcal L}(X_2), {\mathcal L}(Y_2) \in{\mathcal M}_s(\nu_2)$
with arbitrary $\nu_1,\nu_2 \in{\mathcal M}_s(B)$ such that
$(X_1,Y_1)$ and $(X_2,Y_2)$ are independent.

We want to give a result similar to Lemma~\ref{lemtrans} where the
linear operator may also be random itself. We focus on the case that
$B'$ either equals $B$ or
$\mathbb{R}$ where an extension to $\mathbb{R}^d$ for $d > 1$ is
straightforward. Let
$B^*$ be the topological dual of $B$ and $\widehat{B}$ be the space of
all continuous linear maps from $B$ to $B$. Endowed with the operator norms
\[
\| f \|_{\mathrm{op}} = \sup_{x \in B, \|x \| \leq1} \bigl|f(x)\bigr|, \qquad\| f
\|_{\mathrm{op}} = \sup_{x \in B, \|x \| \leq1} \bigl\|f(x)\bigr\|,
\]
both spaces, $B^*$ and $\widehat{B}$, respectively, are Banach spaces.
However, these spaces are typically nonseparable, hence not suitable
for our purposes of measurability. Therefore, we will equip them with
smaller $\sigma$-algebras. Similar to the use of weak-* convergence,
let $\mathcal{B}^*$ be the $\sigma$-algebra on $B^*$ that is
generated by all continuous
(with respect to $\| \cdot\|_{\mathrm{op}}$) linear forms
$\varphi
$ on $B^*$ (i.e., elements of the bidual $B^{**}$) of the form $\varphi
(a) = a(x)$ for some $x \in B$. Note that the set of these continuous
linear forms coincides with the bidual $B^{**}$ if and only if $B$ is
reflexive, a property that is not satisfied in our applications. We
move on to $\widehat{B}$ and define $\widehat{\mathcal{B}}$ to be
the $\sigma
$-algebra generated by all continuous (with respect to $\| \cdot\|
_{\mathrm{op}}$) linear maps $\psi$ from $\widehat{B}$ to $B$ of the
form $\psi(a) = a(x)$ for some $x \in B$.
By Pettis' theorem, we have $\mathcal{B}= \sigma( \ell\in B^*)$.
Hence, if $S
\subseteq B^*$ with $\mathcal{B}= \sigma( \ell\in S)$, then
$\widehat{\mathcal{B}}$ is
also generated by the continuous
linear forms $\varrho$ on $\widehat{B}$ that can be written as
$\varrho
(a) = \ell(a(x))$ for $\ell\in S$ and $x \in B$.

Using the separability of $B$, it is now easy to see that the
norm-functionals $B^*\to\mathbb{R}$, $f \mapsto\|f \|_{\mathrm{op}}
$ and
$\widehat{B} \to\mathbb{R}$, $f \mapsto\|f \|_{\mathrm{op}}$ are
$\mathcal{B}^*$--$\mathcal{B}(\mathbb{R}
)$ measurable and \mbox{$\widehat{\mathcal{B}}$--$\mathcal{B}(\mathbb{R})$}
measurable, respectively.

%de3 #&#
%
\begin{defn} \label{deflin}
By a \emph{random continuous linear form on $B$}, we denote any random
variable with values in $(B^*, \mathcal{B}^*)$. Analogously,\vspace*{1pt} \emph{random
continuous linear operators
on $B$} are random variables with values in $(\widehat{B}, \widehat
{\mathcal{B}})$.
\end{defn}

Note that the definition of the $\sigma$-algebras $\mathcal{B}^*$ and
$\widehat
{\mathcal{B}}$ implies in particular that for any $a \in B^*$ or $a
\in\widehat
{B}, x \in B$, random continuous
linear form or operator $A$ and random variable $X$ in $B$, we have
that the compositions $a(X)$, $A(x)$ and $A(X)$ are again random
variables. The latter property
follows from measurability of the map $(a,x) \mapsto a(x)$ with respect
to $(\mathcal{B}^* \otimes\mathcal{B})$--$\mathcal{B}(\mathbb{R})$ and
$(\widehat{\mathcal{B}} \otimes\mathcal{B})$--$\mathcal{B}$,
respectively. In the case of the dual space, this follows as for any
$r\in\mathbb{R}$ we have
\begin{eqnarray*}
&& \bigl\{ (a,x)\in B^*\times B\dvtx a(x)< r \bigr\}
\\
&&\qquad= \bigcup_{k \geq1}\,
\bigcup_{m \geq1}\,
\bigcap_{n \geq m}\,
\bigcup_{i \geq1} \bigl\{a \in B^*\dvtx a(e_i)< r
-1/k \bigr\} \times\{x\in B\dvtx \|x - e_i\| < 1/n\},
\end{eqnarray*}
where $\{e_i | i\ge1\}$ denotes a countable dense subset of $B$;
the case $\widehat{B}$ being analogous.

The following lemma follows from Lemma~\ref{lemtrans} by conditioning.

%le4 #&#
%
\begin{lem} \label{lemzetalinrandom}
Let ${\mathcal L}(X), {\mathcal L}(Y) \in{\mathcal M}_s(\nu)$ for some
$\nu\in{\mathcal M}_s(B)$. Then, for any random linear continuous form
or operator $A$ with $\mathbf{E} [\|A\|_{\mathrm{op}}^s
] < \infty$ independent
of $X$ and $Y$, we have
\[
\zeta_s \bigl(A(X), A(Y) \bigr) \leq\mathbf{E} \bigl[\|A
\|_{\mathrm{op}}^s \bigr] \zeta_s(X,Y).
\]
\end{lem}

Zolotarev gave upper and lower bounds for $\zeta_s$, most of them being
valid if more structure on $B$ is assumed. Subsequently, only an upper
bound in terms of the minimal $\ell_p$ metric is needed.
For $p>0$ and $\mu, \nu\in\mathcal{M}_p(B)$, the minimal $\ell_p$
distance between $\mu$ and $\nu$ is defined by
\[
\ell_p(\mu, \nu) = \inf\mathbf{E} \bigl[\|X-Y\|^p
\bigr]^{(1/p)
\wedge1},
\]
where the infimum is taken over all common distributions ${\mathcal
L}(X,Y)$ with marginals $\mathcal{L}(X) = \mu$ and $\mathcal{L}(Y) =
\nu$. We
abbreviate $\ell_p(X,Y):=\ell_p({\mathcal L}(X),{\mathcal L}(Y))$.

The next lemma gives an upper bound of $\zeta_s$ in terms of $\ell_s$
where the first statement follows from the Kantorovich--Rubinstein
theorem and the second essentially coincides with Lemma 5.7 in \cite{DrJaNe08}.

%le5 #&#
%
\begin{lem} \label{lemls}
Let ${\mathcal L}(X),{\mathcal L}(Y) \in{\mathcal M}_s(\nu)$ for some
$\nu\in{\mathcal M}_s(B)$ with $B$ separable. If $s \leq1$ then
%
%e12 #&#
%
\begin{equation}
\label{zeta=ls} \zeta_s(X,Y) = \ell_s(X,Y).
\end{equation}
If $s > 1$ then
\[
\zeta_s(X,Y) \leq \bigl(\mathbf{E} \bigl[\| X\|^s
\bigr]^{1 -
1/s}+ \mathbf{E} \bigl[\| Y\|^s \bigr]^{1 - 1/s}
\bigr) \ell_s(X,Y).
\]
\end{lem}

If $X_n, X$ are real-valued random variables, $n\ge1$, then $\zeta
_s(X_n, X) \rightarrow0$ implies convergence of absolute moments of
order up to $s$ since there is a constant $C_s > 0$ such that the
function $x\mapsto C_s |x|^s$ is an element of $\mathcal{F}_s$, hence
$| \mathbf{E} [|X_n|^s - |X|^s ]| \leq C_s^{-1} \zeta_s(X_n,X)$.

We proceed with the fundamental question of how convergence in the
$\zeta_s$ distance relates to weak convergence on $B$.
By the first statement of the previous lemma, or more elementary,
by the proof of the Portmanteau lemma \cite{Billingsley1999}, Theorem~2.1(ii)--(iii), one obtains that for $0 < s \leq1$
convergence in the $\zeta_s$ metric implies weak convergence; see also
\cite{DrJaNe08}, page 300.

If $B$ is a separable Hilbert space, then for any $s > 0$ convergence
in the $\zeta_s$ metric implies weak convergence. This was first proved
by Gin\'{e} and
Le\'{o}n in \cite{GiLe1980}, see also Theorem 5.1 in \cite{DrJaNe08}.
In infinite-dimensional Banach spaces convergence in the $\zeta_s$
metric does not need to imply weak convergence: for any probability
distribution $\mu$ on $B=\mathcal{C}[0,1]$ with zero mean and $\int\|
x\|_\infty^s
\,d\mu(s) < \infty$ for some $s > 2$, that is pre-Gaussian, that is,
there exists a Gaussian
measure $\nu$ on $\mathcal{C}[0,1]$ with zero mean and the same
covariance as $\mu$,
one has $\zeta_s$-convergence of a rescaled sum of independent random
variables with distribution $\mu$ toward $\nu$; see inequality (48) in
\cite{Zolo76}.
However, pre-Gaussian probability distributions supported by a bounded
subset of $\mathcal{C}[0,1]$ that do not satisfy the central limit
theorem can be
found in \cite{DuSt}. For the central limit theorem in Banach spaces,
see \cite{LeTa}. %For various other examples, where $B$ can even be
%reflexive, see \cite{CoTa}.
Note that convergence with respect to $\zeta_s$ implies convergence of
the characteristic functions, hence $\zeta_s(X_n,X) \rightarrow0$
implies that $\mathcal{L}(X)$ is the only possible accumulation point of
$ (\mathcal{L}(X_n) )_{n \geq0}$ in the weak topology.

%s2.2 #&#
\subsection{The Zolotarev metric on $(\mathcal{D}[0,1],d_{\mathrm{sk}})$} \label{subsecdsk}
In this section, we discuss our use of the Zolotarev metric on the
metric space $(\mathcal{D}[0,1], d_{\mathrm{sk}})$ of
c\`{a}dl\`{a}g functions on $[0,1]$ endowed with the Skorokhod metric
defined by
\begin{eqnarray*}
&& d_{\mathrm{sk}}(f,g)
\\
&&\qquad = \inf \bigl\{ \varepsilon>0 | \max \bigl\{ \bigl|f(t) - g
\bigl(\tau(t) \bigr)\bigr|, \bigl|\tau(t) - t\bigr| \bigr\} < \varepsilon\mbox{ for all } t
\in[0,1]
\\
&&\phantom{\qquad = \inf \bigl\{} \mbox{for some monotonically increasing and bijective } \tau\dvtx [0,1]
\to[0,1] \bigr\}.
\end{eqnarray*}
The Borel $\sigma$-algebra of the induced topology is denoted by
$\mathcal{B}
_{\mathrm{sk}}$. For a general introduction to this space, see Billingsley \cite{Billingsley1999}, Chapter~3. In particular, $(\mathcal{D}[0,1],
d_{\mathrm{sk}})$ is a
Polish space, $\mathcal{B}_{\mathrm{sk}}$ coincides with the $\sigma$-algebra generated
by the finite-dimensional projections, the $\sigma$-algebra generated
by the open spheres (with respect to the uniform metric) and the
$\sigma
$-algebra generated by all norm-continuous linear forms on $\mathcal
{D}[0,1]$; see
\cite{Pestman94}, Theorem 3. Subsequently, norm on $\mathcal{D}[0,1]$
will always
refer to the uniform norm $\| \cdot\|_\infty$.
Moreover, the norm function $\mathcal{D}[0,1]\to\mathbb{R}$, $f
\mapsto\|f
\|_\infty$ is
$\mathcal{B}_{\mathrm{sk}}$--$\mathcal{B}(\mathbb{R})$ measurable. By Theorem 2,
respectively, Theorem 4, in
\cite{Pestman94}, any norm-continuous linear form on $\mathcal
{D}[0,1]$ is $\mathcal{B}
_{\mathrm{sk}}$--$\mathcal{B}(\mathbb{R})$ measurable and any norm-continuous linear
map from $\mathcal{D}[0,1]
$ to $\mathcal{D}[0,1]$ is $\mathcal{B}_{\mathrm{sk}}$--$\mathcal{B}_{\mathrm{sk}}$
measurable. Recently, Janson and
Kaijser \cite{jankai12}, Theorem 15.8, generalized the latter result and
proved that any norm-continuous $k$-linear form on $\mathcal{D}[0,1]$
is $(\mathcal{B}
_{\mathrm{sk}})^{\otimes k}$--$\mathcal{B}(\mathbb{R})$ measurable. We do,
however, not know
whether ${\mathcal F}_s$ defined in~(\ref{deffs}) based on the uniform
norm on $\mathcal{D}[0,1]$ is a subset of the $\mathcal
{B}_{\mathrm{sk}}$--$\mathcal{B}(\mathbb{R})$ measurable
functions. Hence, we denote the $\mathcal{B}_{\mathrm{sk}}$--$\mathcal
{B}(\mathbb{R})$
measurable functions
by ${\mathcal E}$ and define the Zolotarev metrics analogously to (\ref
{defzo}) by
\[
\zeta_s(\mu, \nu) = \sup_{f \in\mathcal{F}_s \cap{\mathcal E}} \bigl| \mathbf{E}
\bigl[f(X)-f(Y) \bigr] \bigr|,
\]
where $X$ and $Y$ are $(\mathcal{D}[0,1],d_{\mathrm{sk}})$-valued random
variables with $\mathcal{L}
(X) = \mu$ and $\mathcal{L}(Y) = \nu$.

We denote by $\mathcal{M}_s(\mathcal{D}[0,1])$ the set of probability
distributions
$\mu$ on $\mathcal{D}[0,1]$ with $\int\|x\|_\infty^s\, d \mu(x) <
\infty$ and for $\nu
\in\mathcal{M}_s(\mathcal{D}[0,1])$, we define $\mathcal{M}_s(\nu
)$ to be the
subset of measures $\mu$ from $\mathcal{M}_s(\mathcal{D}[0,1])$
satisfying~(\ref{eqmom2}).
Then $\zeta_s$ is a metric on $\mathcal{M}_s(\nu)$ for all $\nu\in
\mathcal{M}_s(\mathcal{D}[0,1])$, Lemmas~\ref{lemtrans} and \ref{lemideal},
inequality~(\ref{ungldo}), Lemma~\ref{lemls} where~(\ref{zeta=ls}) is
to be replaced by $\zeta_s(X,Y) \leq\ell_s(X,Y)$, and the implication
$\zeta_s(X_n, X) \rightarrow0 \Rightarrow X_n \rightarrow X$ in
distribution if $0 < s \leq1$ remain valid. %Although the linear form
%$f \to f(a)$ with $a \in(0,1)$ is not continuous with respect to
%Skorohod topology, it is norm-continuous, hence Lemma~\ref{lemtrans}
%can be applied and all statements of Proposition~\ref{propfdd} remain
%valid in the cadlag case.

The situation becomes more involved concerning random linear forms and
operators as defined in Definition~\ref{deflin} in the separable
Banach case.
Let $\mathcal{D}[0,1]^*$ and $\widehat{\mathcal{D}[0,1]}$ be the
dual space, respectively, the
space of norm-continuous endomorphisms on $\mathcal{D}[0,1]$ as in the
Banach case.
For reasons of measurability, we need to restrict to smaller subspaces.
Let $\mathcal{D}[0,1]^*_c \subseteq\mathcal{D}[0,1]^*$ be the subset
of functions that are
additionally continuous with respect to $d_{\mathrm{sk}}$. Analogously,
$\widehat
{\mathcal{D}[0,1]}_c \subseteq\widehat{\mathcal{D}[0,1]}$ are those
endomorphism which are
continuous regarded as maps from $(\mathcal{D}[0,1], d_{\mathrm{sk}})$ to
$(\mathcal{D}[0,1], d_{\mathrm{sk}})$.
We endow $\mathcal{D}[0,1]^*_c$ with the $\sigma$-algebra generated
by the function
$f \mapsto\| f \|_{\mathrm{op}}$ and all elements $\varphi$ of
$\mathcal{D}[0,1]
^{**}$ of the form $\varphi(a) = a(x)$ for some $x \in\mathcal
{D}[0,1]$. Also the
$\sigma$-algebra on $\widehat{\mathcal{D}[0,1]}_c$ is generated by
the function $f
\mapsto\| f \|_{\mathrm{op}}$ and the continuous linear maps $\psi\dvtx
\widehat{\mathcal{D}[0,1]} \to\mathcal{D}[0,1]$ of the form
$\varphi(a) = a(x)$ for some $x \in
\mathcal{D}[0,1]$.
Under these conditions, we have the same measurability results as in
the Banach case and Lemma~\ref{lemzetalinrandom} remains valid.

%re6 #&#
%
\begin{rem} \label{remmeas}
Note that we could as well develop the use of the Zolotarev metric
together with the contraction method for the Banach space $(\mathcal
{D}[0,1], \|
\cdot\|_\infty)$. This can be done analogously to the discussion of
Sections~\ref{subsecweak} and~\ref{nnnCM} and in fact would lead to a
proof of Donsker's theorem similar to the one given in Section~\ref
{donskerproof} when replacing the linear interpolation
$S^n=(S^n_t)_{t\in[0,1]}$ by a constant (c\`{a}dl\`{a}g) interpolation
of the random walk. However, the applicability of such a framework
seems to be limited due to measurability problems in the nonseparable
space $(\mathcal{D}[0,1], \| \cdot\|_\infty)$: for example, the
random function
$X$ defined by
\[
X_t = \mathbf{1}_{ \{t \geq U \}}, \qquad t \in[0,1]
\]
with $U$ being uniformly distributed on the unit interval is known to
be nonmeasurable with respect to the Borel-$\sigma$-algebra on
$(\mathcal{D}[0,1],
\| \cdot\|_\infty)$. However, we have applications of the functional
contraction method developed here in mind on processes with jumps at
random times. A typical example in the context of random trees is given
in Section~\ref{secpm}; see also \cite{BrNeSu11}. Hence, in order to
even have measurability of the processes considered it requires to work
with the coarser Skorokhod topology than the uniform topology and this
is our reason for using the Zolotarev metric on $(\mathcal{D}[0,1],
d_{\mathrm{sk}})$ instead
of $(\mathcal{D}[0,1], \| \cdot\|_\infty)$.
\end{rem}

%re7 #&#
%
\begin{rem}\label{remsko}
Although the methodology developed below covers sequences $(X_n)_{n\ge
0}$ of processes with jumps at random times these times will typically
need to be the same for all $n\ge n_0$. In particular, sequences of
processes with jumps at random times that require a (uniformly small)
deformation of the time scale to be aligned cannot be covered by this
methodology. The technical reason is that in condition~(C1) below
(see Section~\ref{nnnCM}) the convergence of the random continuous
endomorphisms $\|A^{(n)}_r - A_r\|_s$ is with respect to the operator
norm based on the uniform norm which in general does not allow a
deformation of the time scale.
\end{rem}

%s2.3 #&#
\subsection{Weak convergence on \texorpdfstring{$(\mathcal{C}[0,1],\|\cdot\|_\infty)$}{$(\mathcal{C}[0,1],||cdot||_{infty})$} and 
\texorpdfstring{$(\mathcal{D}[0,1],d_{\mathrm{sk}})$}{$(\mathcal{D}[0,1],d_{\mathrm{sk}})$}}\label{subsecwcd}
\label{subsecweak}
In this subsection, we only consider the spaces $(\mathcal{C}[0,1], \|
\cdot\|
_\infty)$ and $(\mathcal{D}[0,1], d_{\mathrm{sk}})$.

For random variables $X=(X(t))_{t\in[0,1]}$, $Y=(Y(t))_{t\in[0,1]}$ in
$(\mathcal{C}[0,1], \| \cdot\|_{\infty})$ with $\zeta_s(X,Y) <
\infty$ we have
%
%e13 #&#
%
\begin{equation}
\label{unglfdd} \zeta_s\bigl( \bigl(X(t_1), \ldots,
X(t_k) \bigr), \bigl(Y(t_1), \ldots, Y(t_k)
\bigr)\bigr) \leq k^{s/2} \zeta_s(X,Y)
\end{equation}
for all $0 \leq t_1 \leq\cdots\leq t_k \leq1$. This follows from
Lemma~\ref{lemtrans} using the continuous and linear function $g\dvtx
\mathcal{C}[0,1]
\to\mathbb{R}^k, g(f) = (f(t_1),\ldots, f(t_k))$ and observing that
$\|g \| =
\sqrt{k}$. The bound $\zeta_s((X(t_1), \ldots, X(t_k)), (Y(t_1),\ldots, Y(t_k))) \leq\zeta_s(X,Y)
$ can be obtained if $\mathbb{R}^k$ is endowed with the $\max$-norm
instead of
the Euclidean norm. However, no use of this is made here. Hence, we
obtain for random variables $X_n$, $X$ in $(\mathcal{C}[0,1], \| \cdot
\|_{\infty
})$, $n\ge1$, the implication
\[
\zeta_s(X_n, X) \rightarrow0 \quad\Rightarrow\quad
X_n \stackrel{\mathrm{f.d.d.}} {\longrightarrow} X.
\]
Here, $\stackrel{\mathrm{f.d.d.}}{\longrightarrow}$ denotes weak
convergence of all finite-dimensional marginals of the processes.
Additionally, if $Z$ is a random variable in $[0,1]$, independent of
$(X_n)$ and $X$, then applying Lemma~\ref{lemzetalinrandom}
with the random continuous linear form $A$ defined by $A(f) = f(Z)$
implies
%
%e14 #&#
%
\begin{equation}
\label{unifein} \zeta_s \bigl(X_n(Z), X(Z) \bigr) \leq
\mathbf{E} \bigl[Z^s \bigr] \zeta_s(X_n, X).
\end{equation}

In the c\`{a}dl\`{a}g case, that is, $X=(X(t))_{t\in[0,1]}$,
$Y=(Y(t))_{t\in[0,1]}$ being random variables in $(\mathcal{D}[0,1], d_{\mathrm{sk}})$
inequality~(\ref{unglfdd}) remains true by Lemma~\ref{lemtrans}.
(The fact that $g$ is not continuous with respect to the product
Skorokhod topology does not cause problems since measurability is
sufficient here.) Next, in general, the operator $A$ is no element of
$\mathcal{D}[0,1]^*_c$. Hence, we cannot apply Lemma~\ref{lemzetalinrandom} to
deduce~(\ref{unifein}). Nevertheless, by Theorem 2 in \cite{Zolo77},
the convergence of the characteristic functions of $X_n(t)$ is uniform
in $t$, hence we also have convergence in distribution of $X_n(Z)$ to $X(Z)$.
The\vspace*{-1pt} same argument works for the moments of $X_n(Z)$.
We summarize these properties in the following proposition, where
$\stackrel{d}{\longrightarrow}$ denotes convergence in distribution.

%pr8 #&#
%
\begin{prop} \label{propfdd} For random variables $X_n$, $X$ in
$(\mathcal{C}[0,1],
\| \cdot\|_{\infty})$ or  $(\mathcal{D}[0,1], d_{\mathrm{sk}})$, $n\ge1$, with
$\zeta_s(X_n,X) \rightarrow0$ for $n\to\infty$ we have
\[
X_n \stackrel{\mathrm{f.d.d.}} {\longrightarrow} X.
\]
$\mathcal{L}(X)$ is the only possible accumulation point of $(\mathcal
{L}(X_n))_{n\ge
1}$ in the weak topology.
For all $t \in[0,1]$ we have
\[
X_n(t) \stackrel{d} {\longrightarrow} X(t), \qquad\mathbf{E}
\bigl[\bigl|X_n(t)\bigr|^s \bigr] \to\mathbf{E}
\bigl[\bigl|X(t)\bigr|^s \bigr].
\]
For any random variable $Z$ in $[0,1]$ being independent of $(X_n)$ and
$X$, we have
\[
\mathbf{E} \bigl[\bigl|X_n(Z)\bigr|^s \bigr]\to\mathbf{E}
\bigl[\bigl|X(Z)\bigr|^s \bigr], \qquad X_n(Z) \stackrel{d} {
\longrightarrow} X(Z).
\]
\end{prop}

To conclude from convergence in the $\zeta_s$ metric to weak
convergence on $(\mathcal{C}[0,1], \| \cdot\|_\infty)$ or $(\mathcal
{D}[0,1], d_{\mathrm{sk}})$, further
assumptions are needed. Let, for $r>0$,
%
%e15 #&#
%
\begin{eqnarray}\label{defcrn}
{\mathcal C}_r[0,1]&:=& \bigl\{ f\in\mathcal{C}[0,1] | \exists
0=t_1<t_2<\cdots<t_\ell=1,\ \forall i=1,\ldots, \ell\dvtx
\nonumber\\[-8pt]\\[-8pt]
&&\hspace*{103pt} |t_i-t_{i-1}|\ge r, f|_{[t_{i-1},t_i]}\mbox{ is linear} \bigr\} \nonumber
\end{eqnarray}
denote the set of all continuous functions for which there is a
decomposition of $[0,1]$ into intervals of length at least $r$ such
that the function is piecewise linear on those intervals.
Analogously, we define
%
%e16 #&#
%
\begin{eqnarray} \label{defdrn}
{\mathcal D}_r[0,1]&:=& \bigl\{ f\in\mathcal{D}[0,1] | \exists
0=t_1<t_2<\cdots<t_\ell=1,\ \forall i=1,\ldots,
\ell\dvtx
\nonumber\\[-8pt]\\[-8pt]
&&\hspace*{21pt} |t_i-t_{i-1}|\ge r, f|_{[t_{i-1},t_i)} \mbox{ is
constant, continuous in }1 \bigr\}.\nonumber
\end{eqnarray}

%th9 #&#
%
\begin{teo} \label{teozet}
Let $X_n$ be random variables in $\mathcal{C}_{r_n}[0,1]$, $n\ge0$,
and $X$ a random
variable in $\mathcal{C}[0,1]$. Assume that for $0 < s \leq3$ with
$s=m+\alpha$ as
in~(\ref{deffs})
%
%e17 #&#
%
\begin{equation}
\label{eqzetas2} \zeta_s(X_n, X) = o \biggl(
\log^{-m} \biggl(\frac{1}{r_n} \biggr) \biggr).
\end{equation}
Then $X_n \rightarrow X$ in distribution. The assertion remains valid
if $\mathcal{C}[0,1], \mathcal{C}_{r_n}[0,1]$ are replaced by
$\mathcal{D}[0,1]$, $\mathcal{D}_{r_n}[0,1]$ endowed with the Skorokhod
topology and $X$ has continuous sample
paths.
\end{teo}

As discussed above, $\zeta_s$ convergence does not imply weak
convergence in the spaces $\mathcal{C}[0,1]$ and $\mathcal{D}[0,1]$
without any further
assumption such as~(\ref{eqzetas2}).
In the counterexample from \cite{DuSt}, the sequence $S_n / \sqrt{n}$
there converges to a Gaussian limit with respect
to $\zeta_s$ for $2 < s \leq3$ where the rate of convergence is upper
bounded by the order $n^{1 - s/2}$; see \cite{Zolo76} or \cite{SuDiss}.
Moreover, the sequence is piecewise linear but the sequence $r_n$ can
only be chosen of the order $(cn)^{-2n}$ for some $c > 0$. Hence,
(\ref{eqzetas2}) is not satisfied.

In applications such as our proof of Donsker's functional limit law in
Section~\ref{donskerproof} or the application of the present
methodology to a problem from the probabilistic analysis of algorithms
in \cite{BrNeSu11}, the rate of convergence will typically be of
polynomial order which is fairly sufficient.

We postpone the proof of the theorem to the end of this section and
state two variants, where the first one, Corollary~\ref{Cor4}, contains
a slight relaxation of the assumptions that is useful in applications
such as in the analysis of the complexity of partial match queries in
quadtrees; see Section~\ref{apppm} or \cite{BrNeSu11}.
The second one will be needed in the case $s > 2$; see Section~\ref{secapp2}.

%co10 #&#
%
\begin{cor} \label{Cor4}
Let $X_n, X$ be $\mathcal{C}[0,1]$ valued random variables, $n \geq
0$, and $0 < s
\leq3$ with $s=m+\alpha$ as in~(\ref{deffs}). Suppose $X_n = Y_n +
h_n$ with $Y_n$ being $\mathcal{C}[0,1]$ valued random variables and
$h_n \in\mathcal{C}[0,1]$,
$n \ge0$, such that $\|h_n -h \|_\infty\to0$ for a
$h \in\mathcal{C}[0,1]$ and
%
%e18 #&#
%
\begin{equation}
\label{eqbdge} \mathbf{P} \bigl(Y_n \notin\mathcal{C}_{r_n}[0,1]
\bigr) \to0.
\end{equation}
If
\[
\zeta_s(X_n, X) = o \biggl(\log^{-m} \biggl(
\frac{1}{r_n} \biggr) \biggr),
\]
then
\[
X_n \stackrel{d} {\longrightarrow} X.
\]
The statement remains true if $\mathcal{C}[0,1]$ and $\mathcal
{C}_{r_n}[0,1]$ are replaced by $\mathcal{D}[0,1]$
and $\mathcal{D}_{r_n}[0,1]$ endowed with the Skorokhod topology,
respectively, $X$ has
continuous sample paths and $h$ remains continuous.
\end{cor}

%co11 #&#
%
\begin{cor} \label{Cor1}
Let $X_n, Y_n, X$ be $\mathcal{C}[0,1]$ valued random variables, $n
\geq0$, and $0
< s \leq3$ with $s=m+\alpha$ as in~(\ref{deffs}). Suppose $X_n \in
\mathcal{C}_{r_n}[0,1]
$ for all $n$ and $Y_n \rightarrow X$ in distribution. If
\[
\zeta_s(X_n, Y_n) = o \biggl(
\log^{-m} \biggl(\frac{1}{r_n} \biggr) \biggr),
\]
then
\[
X_n \stackrel{d} {\longrightarrow} X.
\]
The statement remains true if $\mathcal{C}[0,1]$ and $\mathcal
{C}_{r_n}[0,1]$ are replaced by $\mathcal{D}[0,1]$
and $\mathcal{D}_{r_n}[0,1]$ endowed with the Skorokhod topology,
respectively, and $X$ has
continuous sample paths.
\end{cor}

In $\mathcal{C}[0,1]$ (or $\mathcal{D}[0,1]$, if the limit $X$ has
continuous paths), convergence
in distribution implies distributional convergence of the supremum norm
$\|X_n\|_\infty$ by the continuous mapping theorem. In applications,
one is also interested
in convergence of moments of the supremum. For random variables $X$ in
$\mathcal{C}[0,1]$ or $\mathcal{D}[0,1]$, we denote by
\[
\|X\|_s:= \bigl(\mathbf{E} \bigl[\|X\|_\infty^s
\bigr] \bigr)^{(1/s)\wedge1}
\]
the $L_s$-norm of the supremum norm.
% For technical reasons, we have to restrict ourselves to integer $s
%\in\{1,2, 3\}$ in the following theorem.\footnote{Ich sehe nicht,
%wie Janson das zu $0 < s \leq3$ verbessern kann, an irgendeiner
%Stelle muss ich immer die $L_p$ Norm mit negativem Exponenten gegen
%die Supremumsnorm abschaetzen. Das geht aber nicht.} Note that (
%\ref{deffs}) then implies $m=s-1$.

%
%th12 #&#
%
\begin{teo} \label{teosup}
Let $X_n, X$ be $\mathcal{C}[0,1]$ valued random variables and $0 < s
\leq3$ with
$\|X_n \|_s, \|X\|_s < \infty$ for all $n \geq0$. Suppose one of the
following conditions is satisfied:
\begin{longlist}[(2)]
\item[(1)] $X_n \in\mathcal{C}_{r_n}[0,1]$ for all $n$ and
%
%e19 #&#
%
\begin{equation}
\zeta_s(X_n, X) = o \biggl(\log^{-m} \biggl(
\frac{1}{r_n} \biggr) \biggr).
\end{equation}
\item[(2)]\label{item2} $X_n = Y_n + h_n$ with $Y_n$ being $\mathcal
{C}[0,1]$ valued
random variables and $h_n \in\mathcal{C}[0,1]$, $n \ge0$, such that
$\|h_n -h \|
_\infty\to0$ for a
$h \in\mathcal{C}[0,1]$,
%
%e20 #&#
%
\begin{equation}
\label{eqbdge2} \mathbf{E} \bigl[\|X_n\|_\infty^s
\mathbf{1}_{\{ Y_n \notin
\mathcal{C}_{r_n}[0,1]\}} \bigr] \rightarrow0
\end{equation}
and
\[
\zeta_s(X_n, X) = o \biggl(\log^{-m} \biggl(
\frac{1}{r_n} \biggr) \biggr).
\]
\item[(3)]$(Y_n)_{n \geq0}$ is a sequence of $\mathcal{C}[0,1]$ valued
random variables
with $Y_n \leq Z$ almost surely for a $\mathcal{C}[0,1]$ valued random
variable $Z$
with $\|Z\|_s < \infty$, $X_n \in\mathcal{C}_{r_n}[0,1]$ for all $n$ and
\[
\zeta_s(X_n, Y_n) = o \biggl(
\log^{-m} \biggl(\frac{1}{r_n} \biggr) \biggr).
\]
\end{longlist}
Then $\{\|X_n\|_\infty^s | n\ge0\}$ is uniformly integrable. All
statements remain true if
$\mathcal{C}[0,1], \mathcal{C}_{r_n}[0,1]$ are replaced by $\mathcal
{D}[0,1], \mathcal{D}_{r_n}[0,1]$ and $h$ in item~(2) remains continuous.
\end{teo}

It is of interest whether the metric space $({\mathcal M}_s(\nu),\zeta
_s)$ is complete. This is true for $0 < s \leq1$. Also, in the case
that $B$ is a separable Hilbert space, this holds true; see Theorem 5.1
in \cite{DrJaNe08}. Nevertheless, the problem remains open in the
general case, in particular in the cases $\mathcal{C}[0,1]$ and
$\mathcal{D}[0,1]$ with $s>1$. We
can only state the following proposition.

%pr13 #&#
%
\begin{prop} \label{propcomp}
Let $B = (\mathcal{C}[0,1], \|\cdot\|_\infty)$ or $B = (\mathcal
{D}[0,1], d_{\mathrm{sk}})$, $s > 0$ and
$\nu\in\mathcal{M}_s(B)$. Furthermore, let $(\mu_n)_{n\ge0}$ be a
sequence of probability measures from $\mathcal{M}_s(\nu)$ which is a
Cauchy sequence with respect to the $\zeta_s$ metric. Then there exists
a probability measure $\mu$ on $\mathbb{R}^{[0,1]}$ such that, as
$n\to
\infty$,
%
%e21 #&#
%
\begin{equation}
\label{convprop} \mu_n \stackrel{\mathrm{f.d.d.}} {\longrightarrow} \mu.
\end{equation}
\end{prop}

\begin{pf}
Let $\mathcal{L}(X_n) = \mu_n$ for all $n\ge0$.
According to~(\ref{unglfdd}), $(X_n(t_1), \ldots,\break X_n(t_k))_{n\ge0}$
is a Cauchy sequence and hence it exists a random variable $Y_{t_1,
\ldots, t_k}$ in $\mathbb{R}^k$ with
\[
\bigl(X_n(t_1), \ldots, X_n(t_k)
\bigr) \stackrel{d} {\longrightarrow} Y_{t_1,
\ldots, t_k}\qquad(n\to\infty).
\]
The set of distributions of $Y_{t_1, \ldots, t_k}$ for $0 \leq t_1 <
\cdots< t_k\le1$ and $k \in\mathbb N$ is consistent so there exists a
process $Y$ on the product space $\mathbb{R}^{[0,1]}$ whose distribution
satisfies~(\ref{convprop}).
\end{pf}

%re14 #&#
%
\begin{rem} If the distribution $\mu$ found in Proposition~\ref{propcomp} has a version with continuous paths then condition~(\ref{eqmom2}) for $\mu_n$ and $\mu$ is satisfied.
\end{rem}

We now present proofs of the theorems and corollaries of the present
sections. Theorem~\ref{teozet} essentially follows directly from
Theorem 2 in \cite{Barbour90}; see also \cite{BaJa09}. Nevertheless, we
present a version of the proof given there so that we can deduce the
variants and implications given in our other statements. A basic tool
are Theorems~2.2, 2.3~and~2.4 in Billingsley \cite{Billingsley1999}.
%%The following lemma is a special case of Theorem 2.4 there.%It is
%obvious that the assumption of separability of $B$ can be replaced by
%the fact that the limit $\mu$ is
% supported on a separable subset of $B$.}
%the following Lemma whose proof is at the core of Skorokhod's
%representation theorem, see section 6 in \cite{Billingsley1999}.
%However, we will present it in a slightly different way and start with
%two basic Lemma which give a useful criterion of tightness in complete
%metric spaces.
%The first one is part of Lemma 4.3 in \cite{DrJaNe08}, we include the
%proof for the sake of completeness.
%\begin{lem}
% Let $(S, d)$ be a complete metric space. Then a sequence of
%probability measures $(\mu_n)$ on $S$ is tight if for every $
%\varepsilon> 0, \gamma> 0$ there exists a finite set $F$ such that
%for any $n$ $$\mu_n(F^{\gamma}) \geq1- \varepsilon.$$
%Here $F^{\gamma} = \{x \in E\dvtx  d(x,y) < \gamma\mbox{ for some y }
%\in F \}$.
%\end{lem}
%\begin{proof}
% Let $F_n$ be a finite set with $\mu_n \left( \left(F_m^{1/m}\right)^c
%\right) < \varepsilon2^{-m}$ for all $n$, and let $K = \overline{
%\cap_{m \geq1} F_m^{1/m}}$. Then $K$ is closed and totally bounded
%therefore compact. Futhermore
%\[
% \mu_n(K^c) \leq\sum_{m \geq1} \mu_n \left( \left(F_m^{1/m}\right)^c
%\right) < \varepsilon
%\]
%for all $n$ which shows tightness.
%\end{proof}
%The next Lemma immediately follows from the previous one using the
%inclusion-exclusion formula.s

%le15 #&#
%
\begin{lem} \label{lemhelp}
Let $(\mu_n)_{n \geq0}, \mu$ be probability measures on a separable
metric space $(S,d)$. For $r > 0, x \in S$ let $B_r(x) = \{ y \in S\dvtx
d(x,y) < r\}$.
If for any $x_1, \ldots, x_k \in S, \gamma_1, \ldots, \gamma_k > 0$
with $\mu(\partial B_{\gamma_i}(x_i)) = 0$ for $i=1, \ldots, k$ it holds
%
%e22 #&#
%
\begin{equation}
\label{conv22} \mu_n \biggl( \bigcap_{i \in I }
B_{\gamma_i}(x_i) \biggr) \rightarrow\mu \biggl( \bigcap
_{i \in I } B_{\gamma_i}(x_i) \biggr),
\end{equation}
where $I = \{1, \ldots, k \}$, then $\mu_n \rightarrow\mu$ weakly.

Let $(S,d) = (\mathcal{D}[0,1], d_{\mathrm{sk}})$. Then the assertion remains
true when the
balls $B_{\gamma_i}(x_i)$ are still defined with respect to the \emph
{uniform} distance and $\mu(\mathcal{C}[0,1]) = 1$.
\end{lem}

\begin{pf}
The first part of the lemma is a special case of Theorem 2.4 in \cite
{Billingsley1999}. To prove the assertion in the c{\`a}dl{\`a}g space,
we apply
Theorem 2.2 in \cite{Billingsley1999} upon choosing $\mathcal{A}_P$
there to be the set of
finite intersection of sets $A$ where $A$ is either a $\mu$-continuous
open sphere (in the uniform distance) whose center lies in $\mathcal
{C}[0,1]$ or a
measurable set with positive uniform distance from $\mathcal{C}[0,1]$.
Using~(\ref{conv22}) and the inclusion-exclusion formula, it is easy
to see that
$\mu_n(C) \to0$ for any measurable set $C$ with positive uniform
distance from $\mathcal{C}[0,1]$, in particular $\mu_n(A) \to\mu
(A)$ for any $A \in
\mathcal{A}_P$. Moreover, we can decompose any open set $O \in
\mathcal{D}[0,1]$ (in
the Skorokhod topology) into $O'$ and $O \setminus O'$ with
\[
O':= \bigcup_{x,\delta}
B^{\| \cdot\|}_{x}(\delta),
\]
where the union is over all $x \in O \cap\mathcal{C}'$ for a countable
set $ \mathcal{C}'$ that is dense in $\mathcal{C}[0,1]$ and $\delta
\in\mathbb{Q}^+$ such
that $B^{\| \cdot\|}_{x}(\delta) \subseteq O$ and $B^{\| \cdot\|
}_{x}(\delta)$ is $\mu$-continuous. We have $O \cap\mathcal
{C}[0,1]\subseteq O'$
since any ball in the metric $d_{\mathrm{sk}}$ with center in $\mathcal
{C}[0,1]$ contains a
concentric ball in the uniform distance. Hence,
\[
O \setminus O' = \bigcup_{\delta\in\mathbb{Q}^+} \bigl
\{x \in O \setminus O'\dvtx \|y-x\| > \delta\mbox{ for all } y
\in\mathcal{C}[0,1] \bigr\}.
\]
Thus, any open set $O$ is a countable union of sets in $\mathcal{A}_P$
which proves all conditions of Theorem 2.2 in \cite{Billingsley1999} to
be satisfied and the claim follows.
\end{pf}

A main difficulty in deducing weak convergence from convergence in
$\zeta_s$ compared to the Hilbert space case is the
nondifferentiability of the norm function $x \mapsto\| x \|_\infty$;
see \cite{dieudonne}, page 147. We will instead use the smoother
$L_p$-norm which approximates the
supremum norm in the sense that
%
%e23 #&#
%
\begin{equation}
\label{ellpconv} L_p(x) \rightarrow\| x \|_\infty
\end{equation}
for any fixed $x \in\mathcal{C}[0,1]$ as $p \rightarrow\infty$.

For the remaining part of this section, $p$, for fixed values or
tending to infinity, is always to be understood as an even integer with
$p \geq4$. We use the Bachmann--Landau big-$O$ notation.

%le16 #&#
%
\begin{lem} \label{lemdiff}
For $x, y \in\mathcal{C}[0,1]$ let
\[
L_p(x) = \biggl(\int_0^1
\bigl[x(t) \bigr]^p \,dt \biggr)^{1/p}, \qquad
\psi_{p,y}(x) = L_p \bigl( \bigl(1+ [x-y]^2
\bigr)^{1/2} \bigr).
\]
Then $L_p$ is smooth on $\mathcal{C}[0,1]\setminus\{\mathbf{0}\}$
where $\mathbf
{0}$ is the zero-function and $\psi_{p,y}$ is smooth on $\mathcal
{C}[0,1]$ for all
$y\in\mathcal{C}[0,1]$. Furthermore, for $k \in\{1,2,3 \}$, we have
\[
\bigl\| D^kL_p(x)\bigr\| = O \bigl(p^{k-1}
L_p^{1-k}(x) \bigr),
\]
uniformly for $p$ and $x \in\mathcal{C}[0,1]\setminus\{ \mathbf{0}
\}$.
Moreover, again for $k \in\{1,2,3 \}$,
%
%e24 #&#
%
\begin{equation}
\label{unglpsi} \bigl\| D^k \psi_{p,y}(x)\bigr\| = O
\bigl(p^{k-1} \bigr)
\end{equation}
uniformly for $p$ and $x,y \in\mathcal{C}[0,1]$. All assertions
remain valid when
$\mathcal{C}[0,1]$ is replaced by $\mathcal{D}[0,1]$, moreover both
functions $L_p$ and $\psi
_{p,y}$ are continuous with respect to the Skorokhod metric for all $p$
and $y \in\mathcal{D}[0,1]$.
\end{lem}

\begin{pf}
The smoothness properties are obvious. Differentiating $L_p$ by the
chain rule yields
\[
DL_p(x)[h] = \biggl(\int_0^1
\bigl[x(t) \bigr]^p \,dt \biggr)^{1/p -1} \int
_0^1 \bigl[x(t) \bigr]^{p-1} h(t) \,dt.
\]
For $h\in\mathcal{C}[0,1]$ with $\|h\| \leq1$ by Jensen's inequality and
$L_p(h)\le\|h\|$, we obtain that the right-hand side of the latter
display is
uniformly bounded by $1$. The bounds on the norms of the higher order
derivatives follow along the same lines. Using the same ideas, it is
easy to see that
\[
\bigl\| D^k \psi_{p,y}(x)\bigr\| = O \Biggl( \sum
_{j=1}^{k} p^{j-1} L_p^{1-j}
\bigl(\omega_y(x) \bigr) \Biggr),
\]
uniformly in $p$ and $x,y \in\mathcal{C}[0,1]$ where $\omega_y(x) =
(1+|x-y|^2)^{1/2}$. This gives~(\ref{unglpsi}).
\end{pf}

Note that the convergence in~(\ref{ellpconv}) holds pointwise; it is
easy to construct a sequence of continuous functions $(x_p)_{p \geq0}$ such
that $L_p(x_p) \rightarrow0$ and $\|x_p\|_\infty\rightarrow\infty$
as $p \to\infty$. Additionally to the obvious bound $L_p(x) \leq\|x\|
_\infty$, we will need the following
simple lemma which contains sort of a converse of this inequality.

%le17 #&#
%
\begin{lem} \label{ana1}
Let $\lambda$ denote the Lebesgue measure on the unit interval and let
$\gamma> 0$ and $0 < \vartheta< 1$.
\begin{longlist}
\item[(a)]
For all $f \in{\mathcal D}_r[0,1]$, we have
\[
\|f \|_\infty\geq\gamma\quad\Rightarrow\quad\lambda \bigl( \bigl\{t\dvtx \bigl|f(t)\bigr|
\geq (1-\vartheta) \gamma \bigr\} \bigr)\geq r.
\]
Moreover, for any $g \in\mathcal{C}[0,1]$, there exists a $\delta=
\delta(g,\gamma, \vartheta) > 0 $ such that
\[
\|f - g\|_\infty\geq\gamma\quad\Rightarrow\quad\lambda \bigl( \bigl\{t\dvtx \bigl|f(t)
- g(t)\bigr| \geq(1-\vartheta) \gamma \bigr\} \bigr) \geq\min(r, \delta).
\]
\item[(b)]
For all $f \in{\mathcal C}_r[0,1]$, we have
\[
\|f \|_\infty\geq\gamma\quad\Rightarrow\quad\lambda \bigl( \bigl\{t\dvtx \bigl|f(t)\bigr|
\geq (1-\vartheta) \gamma \bigr\} \bigr)\geq\frac{\vartheta}{2} r.
\]
Moreover, for $g \in\mathcal{C}[0,1]$, there exists a $\delta=
\delta(g,\gamma,\vartheta) > 0$
with
\[
\|f - g\|_\infty\geq\gamma\quad\Rightarrow\quad\lambda \bigl( \bigl\{t\dvtx \bigl|f(t)
- g(t)\bigr| \geq(1-\vartheta) \gamma \bigr\} \bigr) \geq\frac{\vartheta}{4} \min(r,
\delta). %
\]
\end{longlist}
\end{lem}

\begin{pf}
Ad (a): The first assertion is trivial. The second one follows by choosing
$\delta> 0$ small enough such that $|g(x) - g(y)| \leq\frac
{\vartheta
\gamma}{2}$ for all $|x-y| < \delta$.

Ad (b): For the first statement, assume $\|f\|_\infty\geq\gamma$ and
let $[e_0, e_1]$ be an interval where $f$ attains its maximum. A
geometric argument shows that
the quantity $\lambda(\{t \in[e_0,e_1]\dvtx  |f(t)| \geq
(1-\vartheta
) \gamma\} )$ is minimized when $f(e_0) = \gamma$ and $f(e_1) =
-(1-\vartheta)\gamma$. In this case, the quantity equals
$\vartheta r / (2(2-\vartheta))$ which implies the assertion since $0 <
\vartheta< 1$. Finally, the last statement follows from a combination
of the latter argument and by choosing $\delta> 0$ again
such that $|g(x) - g(y)| \leq\frac{\vartheta\gamma}{2}$ for all
$|x-y| < \delta$.
\end{pf}

%
%\begin{lem} \label{ana1}
%Let $f \in{\mathcal D}_r[0,1]$, $g \in\Co$ and denote by $\lambda(
%\cdot)$ the Lebesgue measure on the real line.
%Then for any $\gamma> 0$ and $0 < \vartheta< 1$ there exists
%$\delta= \delta(g,\gamma, \vartheta) > 0 $ such that
%$$ \|f - g\| \geq\gamma\Rightarrow\lambda\left(\{t\dvtx  |f(t) - g(t)|
%\geq(1-\vartheta) \gamma\}\right) \geq\frac{1}{2} \min\left(r,
%\delta\right).$$
% Let $f \in{\mathcal C}_r[0,1]$ and $g,\gamma, \vartheta$ as above.
%Then, there exists $\delta= \delta(g,\gamma,\vartheta) > 0$
%with
%$$ \|f - g\| \geq\gamma\Rightarrow\lambda\left(\{t\dvtx  |f(t) - g(t)|
%\geq(1-\vartheta) \gamma\}\right) \geq\frac{1}{6} \min\left(r,
%\delta\right).$$
%\end{lem}

We start with the proofs of Theorem~\ref{teozet} and its corollaries
in the continuous case.

\begin{pf*}{Proof of Theorem~\ref{teozet}}
For $r > 0, x \in\mathcal{C}[0,1]$ let $B_r(x) = \{ y \in\mathcal
{C}[0,1]\dvtx  \|y-x\|_\infty< r\}
$. According to Lemma~\ref{lemhelp}, we need to verify that
%
%e25 #&#
%
\begin{equation}
\label{ggnn} \mathbf{P} \biggl(X_n \in\bigcap
_{i \in I} B_{\gamma_i}(x_i) \biggr)
\rightarrow\mathbf{P} \biggl(X \in\bigcap_{i \in I}
B_{\gamma_i}(x_i) \biggr)
\end{equation}
for $I = \{1, \ldots, k\}$ and $x_1, \ldots, x_k \in S, \gamma_1,
\ldots, \gamma_k > 0$ such that $\mathbf{P} (X \in\break (\partial
B_{\gamma_i}(x_i) ) ) = 0$.
%\\ Note that it was sufficient to find functions $\nu_{p,i}\dvtx  \Co\to
%[0,1]$ such that
%$$ \nu_{p,i} \rightarrow\I{B_{\gamma}(x_i)}, \qquad(p \rightarrow
%\infty),$$
%\emph{uniformly} on $\Co$ and constants $\alpha_{p,i} > 0$ such that $
%\alpha_{p,i} \nu_{p,i} \in\F_s$. Then it follows
%\begin{eqnarray*}
% \lim_{n \to\infty} \Probc{X_n \in\bigcap_{i \in I} B_{\gamma}(x_i)}
%&= \lim_n \lim_p \Ec{\prod_{i\in I} \nu_{p,i}(X_n) } \\
%& = \lim_p \lim_n \Ec{\prod_{i \in I} \nu_{p,i}(X_n) } \\
%& = \lim_p \Ec{ \prod_{i \in I} \nu_{p,i}(X) } \\
%& = \Probc{X \in\bigcap_{i \in I} B_{\gamma}(x_i)}
%\end{eqnarray*}
%Even though we will be able to approximate $\I{B_{\gamma}(x_i)}$
%pointwise we cannot proceed along these lines due to the lack of
%uniformity.
The lack of uniformity in~(\ref{ellpconv}) leads us to find lower and
upper bounds on the desired quantity. We will establish
%
%e26 #&#
%
\begin{equation}
\label{limsup} \limsup_{n \to\infty} \mathbf{P} \biggl(X_n
\in\bigcap_{i \in I} B_{\gamma_i}(x_i)
\biggr) \leq\mathbf{P} \biggl(X \in\bigcap_{i \in I}
B_{\gamma_i}(x_i) \biggr)
\end{equation}
and
%
%e27 #&#
%
\begin{equation}
\label{liminf} \liminf_{n \to\infty} \mathbf{P} \biggl(X_n
\in\bigcap_{i \in I} B_{\gamma_i}(x_i)
\biggr) \ge\mathbf{P} \biggl(X \in\bigcap_{i
\in I}
B_{\gamma_i}(x_i) \biggr)
\end{equation}
separated from each other.
To this end, it is sufficient to construct functions $g_{i,n}, \tilde
{g}_{i,n}\dvtx  \mathcal{C}[0,1]\to[0,1]$ satisfying
%
%e28 #&#
%e29 #&#
%
\begin{eqnarray}
\tilde{g}_{i,n}(x) &\leq&\mathbf{1}_{  B_{\gamma_i}(x_i)  }(x) \leq
g_{i, n}(x) \qquad\mbox{for all } x \in\mathcal{C}_{r_n}[0,1],
\label{Ei1}
\\
g_{i, n}(x),\tilde{g}_{i,n}(x) &\rightarrow&
\mathbf{1}_{  B_{\gamma
_i}(x_i)  }(x) \qquad\mbox{for all } x \in\mathcal{C}[0,1]
\setminus\partial B_{\gamma_i}(x_i) \label{Ei2}
\end{eqnarray}
and such that $a_{n} \prod_{i \in I} g_{i,n}$, $\tilde{a}_{n} \prod_{i
\in I} \tilde{g}_{i,n} \in\mathcal{F}_s$ for appropriate constants $a_{n},
\tilde{a}_{n} > 0$ such that
$a_n^{-1} \zeta_s(X_n,X) \to0$ and $\tilde{a}_n^{-1} \zeta_s(X_n,X)
\to0$ as $n \to\infty$. This is sufficient since we then may conclude
%
%e30 #&#
%
\begin{eqnarray}\label{rechnung1}
\mathbf{P} \biggl(X_n \in\bigcap
_{i \in I} B_{\gamma_i}(x_i) \biggr) &\leq&
\mathbf{E} \biggl[\prod_{i \in I} g_{i,n}(X_n)
\biggr]
\nonumber\\[-8pt]\\[-8pt]
&\leq&\mathbf{E} \biggl[\prod_{i \in I}
g_{i,n}(X) \biggr] + a_n^{-1}
\zeta_s(X_n,X)\nonumber
\end{eqnarray}
and
%
%e31 #&#
%
\begin{eqnarray}\label{rechnung2}
\mathbf{P} \biggl(X_n \in\bigcap
_{i \in I} B_{\gamma_i}(x_i) \biggr) &\geq&
\mathbf{E} \biggl[\prod_{i \in I} \tilde{g}_{i,n}(X_n)
\biggr]
\nonumber\\[-8pt]\\[-8pt]
&\geq&\mathbf{E} \biggl[\prod_{i \in I} \tilde
{g}_{i,n}(X) \biggr] - \tilde{a}_n^{-1}
\zeta_s(X_n,X).\nonumber
\end{eqnarray}
%
%Now, if $a_n^{-1} \zeta_s(X_n,X) \to0$ for $n \rightarrow\infty$
%then \eqref{rechnung1} implies \eqref{limsup} and similarly
%\eqref{liminf} follows from
%\eqref{rechnung2}) if $\bar a_n^{-1} \zeta_s(X_n,X) \to0$ as $n
%\rightarrow\infty$.
While this is the basic idea subsequently, the construction is slightly
more involved.

We first give a motivation of how to construct the functions $g_{i,n}$:
according to~(\ref{Ei2}), asymptotically, the functions $g_{i,n}$ have
to separate points $x \in\mathcal{C}[0,1]$ which are in $B_{\gamma_i}(x_i)$
from those which are not. This is why we use the $L_p$ norm.
%%%%%%%Note, that for $x,y \in\Co$ with $\|x-y\|_\infty<\gamma$ we have
%In order to use~(\ref{eqzetas2})
%to bound the error terms the functions $g_{i,n}$ have to be
%differentiable. The norm function is not differentiable so we cannot
%use $||x-x_i||$ directly. Since the $L_p$ norm approximates the
%supremum norm for $p \rightarrow\infty$ we can measure differences in
%the suprema also in the integral. To be more precise, let $$||f||_p =
%\left(\int_0^1 |f|^p \right)^{1/p}$$ and observe that due to $||f||_p
%\rightarrow||f||$ for $p \rightarrow\infty$ we have
%%%%%%%\begin{eqnarray*}
%%%%%%% \limsup_{p \rightarrow\infty} L_p(x-y) < \gamma.
%%%%%%%\end{eqnarray*}
%%%%%%%Analogously, for $\|x-y\|_\infty>\gamma$ we have
%%%%%%\begin{eqnarray*}
%%%%%\liminf_{p \rightarrow\infty} L_p(x-y) > \gamma.
%%%%%\end{eqnarray*}
%In addition the map $x \to\int x(t) \,dt$ is linear and therefore easy
%do differentiate. \\
%Let $p > 0, s \in\Co$ and $\psi_{p, s}\dvtx  \Co\to\R$ be defined by
%\[
% \psi_{p, s}(x) = ||(1 + (x-s)^2)^{1/2}||_p
%\]
Consider $\psi_{p, x_i}$ as introduced in Lemma~\ref{lemdiff}.
If $x \in\overline{B_{\gamma_i}(x_i)}$, then $\psi_{p, x_i}(x) \leq
(1+\gamma_i^2)^{1/2}$ whereas if
$x \notin\overline{B_{\gamma_i}(x_i)}$ then $\liminf_{p \rightarrow
\infty} \psi_{p, x_i}(x) > (1+ \gamma_i^2)^{1/2}$.

Let $\varphi\dvtx  \mathbb{R}\to[0,1]$ be a three times continuously
differentiable function with $\varphi(u) = 1$ for $u \leq0$ and
$\varphi(u) = 0$ for $u \geq1$. For $\varrho\in\mathbb{R}$ and
$\eta>0$, we
denote $\varphi_{\varrho, \eta}\dvtx  \mathbb{R}^+ \to[0,1]$ by
$\varphi
_{\varrho,
\eta}(u) = \varphi((u - \varrho)/ \eta)$.

Let $g_i(x) = \varphi_{(1+\gamma_i^2)^{1/2}, \eta}(\psi_{p,x_i}(x))$.
Let $g_{i,n} = g_i$ with $\eta= \eta_n \downarrow0$ and $p = p_n
\uparrow\infty$. Then $g_{i,n}$ has the properties in~(\ref{Ei1}) and
(\ref{Ei2}).

We do not know how to construct functions $\tilde{g}_{i,n}$ with the
properties~(\ref{Ei1}) and~(\ref{Ei2}). Instead, we construct functions
$\bar g_{i,n}$ satisfying related conditions:
let $0 < \vartheta< 1$ and $x \in\mathcal{C}_{r_n}[0,1]$. %Since the
%family $(x_i)_{i
%\in I}$ is uniformly equicontinuous, by Lemma~\ref{ana1}
By Lemma~\ref{ana1}(b), we can find $\delta= \delta(\vartheta)$ (also
depending on $x_1, \ldots, x_k, \gamma_1,\ldots,\gamma_k$ which are
kept fixed) with
%
%e32 #&#
%
\begin{eqnarray}
\label{trick17}
&& \bigl\{ \|x - x_i \|_\infty\geq
\gamma_i \bigr\}\nonumber
\\
&&\qquad \subseteq \biggl\{ \lambda \bigl( \bigl\{ t\dvtx \bigl|x(t)
- x_i(t)\bigr| \geq\gamma_i(1-\vartheta) \bigr\} \bigr)
\geq \frac{\vartheta}{4}\min(r_n, \delta) \biggr\}
\nonumber\\[-8pt]\\[-8pt]
&&\qquad \subseteq \biggl\{ \psi_{p,x_i}(x) \geq \bigl(1 +
\gamma_i^2(1- \vartheta)^2
\bigr)^{1/2} \biggl(\frac{\vartheta}{4} \min(r_n, \delta)
\biggr) ^{1/p} \biggr\}\nonumber
\\
&&\qquad \subseteq \bigl\{\bar g_{i,n}(x) = 0 \bigr\}\nonumber
\end{eqnarray}
with $\bar g_{i,n}(x) = \varphi_{(1 + \gamma_i^2(1-\vartheta)^2)^{1/2}
(\vartheta\min(r_n, \delta)/4)^{1/p}- \eta, \eta}(\psi_{p,x_i}(x))$.
This gives~(\ref{Ei1}).
$\bar g_{i,n}$ does not fulfill (\ref{Ei2}), but we have
\[
\bar g_{i,n}(x) \rightarrow\mathbf{1}_{  B_{\gamma_i(1-\vartheta
)}(x_i)  }(x)
\]
for $x \in\mathcal{C}[0,1]\setminus\partial B_{\gamma
_i(1-\vartheta)}(x_i)$
and $ p = p_n \uparrow\infty, \eta= \eta_n \downarrow0$ such that
$r_n^{1/p_n} \rightarrow1$.
%This is easily seen since on the one hand for $x \in B_{x_i}^{
%\vartheta}$ there exists an $\varepsilon>0$ with $||x - x_i|| <
%\gamma_i(1- \vartheta) - \varepsilon$ such that for $n$ large enough
%\begin{eqnarray*}
% \psi_{p_n, x_i}(x) & < & (1 + (\gamma_i (1-\vartheta) - \varepsilon)
%^2)^{1/2} \\
%& < & (1 + \gamma_i^2 (1-\vartheta) ^2)^{1/2} \left(\frac{\vartheta
%r_n} {16}\right)^{1/p_n} - \eta_n.
%\end{eqnarray*}
This gives for every $0 < \vartheta< 1$ with $\mathbf{P} (X \in
\partial B_{\gamma_i(1-\vartheta)}(x_i) ) = 0$ for all $i \in I$
\[
\lim_{n \rightarrow\infty} \mathbf{E} \biggl[\prod
_{i \in I} \bar g_{i,n}(X) \biggr] = \mathbf{P} \biggl(X
\in\bigcap_{i \in I} B_{\gamma_i(1-\vartheta
)}(x_i)
\biggr).
\]
Assuming that $\bar a_n \prod_{i \in I} \bar g_{i,n} \in\mathcal
{F}_s$ and
letting $n$ tend to infinity~(\ref{rechnung2}) rewrites as
%
%e33 #&#
%
\begin{eqnarray}\label{eq1234}
&& \liminf_{n\to\infty} \mathbf{P} \biggl(X_n
\in\bigcap_{i \in I} B_{\gamma_i}(x_i)
\biggr)
\nonumber\\[-8pt]\\[-8pt]
&&\qquad \geq\mathbf{P} \biggl(X \in\bigcap_{i
\in I}B_{\gamma_i(1-\vartheta)}(x_i) \biggr)- \limsup_{n\to
\infty}\bar a_n^{-1} \zeta_s(X_n,X),\nonumber
\end{eqnarray}
where $\bar a_n$ may depend on $\vartheta$ and $\delta$.
Below, we will see that the error term on the right-hand side of (\ref
{eq1234}) vanishes as $n \to\infty$ uniformly in $\vartheta, \delta$.
So, choosing $\vartheta\downarrow0$ such that $\mathbf{P} (X
\in\partial B_{\gamma_i(1-\vartheta)}(x_i) ) = 0$ for all $i
\in I$ the assertion
\[
\liminf_{n\to\infty} \mathbf{P} \biggl(X_n \in\bigcap
_{i \in I} B_{\gamma_i}(x_i) \biggr)
\geq\mathbf{P} \biggl(X \in\bigcap_{i
\in I}
B_{\gamma_i}(x_i) \biggr) %
\]
follows.

It remains to show that the error terms vanish in the limit.
By Lemma~\ref{lemdiff} $g(x) = \varphi_{\varrho, \eta}(\psi_{p,y}(x))$
and using the mean value theorem,
%For the derivatives we can choose $y = \mathbf{0}$ without loss of
%generality.
%Again using Lemma~\ref{lemdiff}
we obtain for $m = 0, 1, 2$
\[
\bigl\| g^{(m)}(x+h) - g^{(m)}(x) \bigr\| \leq C_m
p^m \eta^{- (m+1)} \|h\| _\infty^\alpha
\]
for $p \geq4, \eta< 1$ and some constants $C_m > 0$.
It is easy to check that the same is valid for products of functions of
form $g$ with different constants, independent of the parameters.
It follows that both error terms in~(\ref{rechnung1}) and (\ref
{eq1234}) are bounded by $C'_m p_n^m\eta_n^{-(m+1)} \zeta_{s}(X_n, X)$
for all $n$,
uniformly in $\vartheta, \delta$, where $C_m'$ denotes a fixed constant
for each $m \in\{0,1,2\}$. By~(\ref{eqzetas2}), we can choose $p_n
\uparrow\infty$ and $\eta_n \downarrow0$ such that both $r_n^{1/p_n}
\rightarrow1$ and the error terms vanish in the limit.
%This proves the theorem in the continuous case. In the cadlag case,
%the proof runs along the same lines with the only difference that the
%first statement
%of Lemma~\ref{ana1} is used in \eqref{trick17}.
\end{pf*}

\begin{pf*}{Proof of Corollary~\ref{Cor4}}
Again, according to Lemma~\ref{lemhelp}, we only have to verify (\ref
{ggnn}), for which
we modify the proof of Theorem~\ref{teozet}: first note that the
assumption of piecewise linearity of $X_n$ and the convergence rate for
$\zeta_s(X_n,X)$ are not necessary for the upper bound
\[
\limsup_{n\to\infty} \mathbf{P} \biggl(X_n \in\bigcap
_{i \in I} B_{\gamma_i}(x_i) \biggr)
\leq\mathbf{P} \biggl(X \in\bigcap_{i
\in I}
B_{\gamma_i}(x_i) \biggr). %
\]
For the lower bound let $\varepsilon> 0$ and note that
\begin{eqnarray*}
\mathbf{P} \biggl(X_n \in\bigcap_{i \in I}
B_{\gamma_i}(x_i) \biggr) &\geq&\mathbf{P}
\biggl(X_n \in\bigcap_{i \in I}
B_{\gamma
_i}(x_i) \cap \bigl\{Y_n \in\mathcal
{C}_{r_n}[0,1] \bigr\} \biggr).
\end{eqnarray*}
We modify the functions $\bar g_{i,n}(x)$. Let $0 < \gamma_{K_i} <
\gamma_i$ such that
\[
\mathbf{P} \biggl(X \in\bigcap_{i \in I}
B_{\gamma_{K_i}}(x_i) \biggr) \geq\mathbf{P} \biggl(X \in\bigcap
_{i \in I} B_{\gamma
_i}(x_i) \biggr) -
\varepsilon
\]
and\vspace*{2pt} $\mathbf{P} (X \in\partial B_{\gamma_{K_i}}(x_i) )
=0$ for all $i$. Let
$0 < \vartheta< 1$ and $n_0$ be large enough such that
$\varrho_n = \|h_n-h\|_{\infty} < \min_i ( \gamma_{K_i}
(1-\vartheta)
\wedge\gamma-\gamma_{K_i})$ and
$\mathbf{P} (Y_n \notin\mathcal{C}_{r_n}[0,1] ) <
\varepsilon$ for all $n
\geq n_0$. %,
%since the functions $(x_i-h)_{i \in I}$ are uniformly equicontinuous,
By Lemma~\ref{ana1}(b),
there exists $\delta= \delta(\vartheta)$ such that for $y \in
\mathcal
{C}_{r_n}[0,1]$ with $x = y + h_n$ and $n \geq n_0$
%The family $x_i - h$ is uniformly equicontinuous and $y$ is piecewise
%linear on intervals of length at least $r_n$ and $we can find $n_0$
%such that for $n \geq n_0$ it holds
%and $x = y + h_n$ with $y \in\Crn$ and $|h_n - h| < \varrho_n$. Now,
%since
%
\begin{eqnarray*}
&& \bigl\{ \|x - x_i \|_\infty\geq\gamma_i \bigr\}
\\[-2pt]
&&\qquad \subseteq \bigl\{ \|y + h - x_i\|_\infty\geq
\gamma_{K_i} \bigr\}
\\[-2pt]
&&\qquad \subseteq \biggl\{ \lambda \bigl( \bigl\{t\dvtx \bigl|y(t) + h(t) -
x_i(t)\bigr| \geq\gamma_{K_i} (1-\vartheta) \bigr\} \bigr)
\geq\frac{\vartheta}{4}\min(r_n, \delta) \biggr\}
\\[-2pt]
&&\qquad  \subseteq \biggl\{ \lambda \bigl( \bigl\{t\dvtx \bigl|x(t) - x_i(t)\bigr|
\geq \gamma_{K_i} (1-\vartheta) - \varrho_n \bigr\} \bigr)
\geq \frac{\vartheta}{4}\min(r_n, \delta) \biggr\}
\\[-2pt]
&&\qquad \subseteq \biggl\{ \psi_{p,x_i}(x) \geq \bigl(1 + \bigl(
\gamma_{K_i} (1-\vartheta) - \varrho_n \bigr)^2
\bigr)^{1/2} \biggl( \frac{\vartheta}{4}\min(r_n, \delta)
\biggr) ^{1/p} \biggr\}
\\[-2pt]
&&\qquad \subseteq \bigl\{\bar g_{i,n}(x) = 0 \bigr\}
\end{eqnarray*}
with $\bar g_{i,n}(x) = \varphi_{(1 + (\gamma_{K_i} (1-\vartheta
)-\varrho_n)^2)^{1/2} (\vartheta\min(r_n, \delta)/4 )^{1/p}- \eta,
\eta
}(\psi_{p,x_i}(x))$.
Hence,
\begin{eqnarray*}
\mathbf{P} \biggl(X_n \in\bigcap_{i \in I}
B_{\gamma_i}(x_i) \biggr)
\geq \mathbf{E} \biggl[\prod
_{i \in I} \bar g_{i,n}(X_n)
\mathbf{1}_{\{Y_n \in\mathcal{C}_{r_n}[0,1]\}} \biggr]
&\geq& \mathbf{E} \biggl[\prod_{i \in I} \bar
g_{i,n}(X_n) \biggr] - \varepsilon
\end{eqnarray*}
for $n \geq n_0$. The upper bound of the error term $\bar a_n^{-1}
\zeta
_s(X_n,X)$ is a function of $p$ and $\eta$ so it is uniform in
$\varrho
_n, \vartheta, \delta$.
Following the same lines as in the proof of Theorem~\ref{teozet} gives\vspace*{-2pt}
\begin{eqnarray*}
\liminf_{n \rightarrow\infty} \mathbf{P} \biggl(X_n \in\bigcap
_{i\in I} B_{\gamma_i}(x_i) \biggr) &\geq&\mathbf{P} \biggl(X \in\bigcap_{i \in I}
B_{\gamma_{K_i}}(x_i) \biggr) - \varepsilon
\\[-3pt]
& \geq& \mathbf{P} \biggl(X \in\bigcap_{i \in I}
B_{\gamma_i}(x_i) \biggr) - 2 \varepsilon.
\end{eqnarray*}
Since $\varepsilon> 0$ was arbitrary, the result follows. %Again, in
%the cadlag case, we have to take the first part of Lemma~\ref{ana1}
%and the rest
%of the proof remains the same.
\end{pf*}

\begin{pf*}{Proof of Corollary~\ref{Cor1}}
In the setting of the proof of Theorem~\ref{teozet}, (\ref{rechnung1})~rewrites as
\begin{eqnarray*}
&& \mathbf{P} \biggl(X_n \in\bigcap_{i \in I}
B_{\gamma_i}(x_i) \biggr)
\\[-2pt]
&&\qquad \leq\mathbf{E} \biggl[\prod_{i \in I}
g_{i,n}(X_n) \biggr] \leq\mathbf{E} \biggl[\prod
_{i \in I} g_{i,n}(Y_n) \biggr] +
a_n^{-1} \zeta_s(X_n,Y_n)
\\[-2pt]
&&\qquad = \mathbf{E} \biggl[\prod_{i \in I}
g_{i,n}(Y_n) \biggr] - \mathbf{E} \biggl[\prod
_{i \in I} g_{i,n}(X) \biggr] + \mathbf{E} \biggl[\prod
_{i \in I} g_{i,n}(X) \biggr] +
a_n^{-1} \zeta_s(X_n,Y_n).
\end{eqnarray*}
We may choose $Y_n \rightarrow X$ almost surely. On the event $\{X \in
B_{\gamma_i}(x_i)\}$, we have $\lim_n g_{i,n}(Y_n) = \lim_n g_{i,n}(X)
= 1$ and on $\{X \notin\overline{B_{\gamma_i}(x_i)} \}$ we have
$\lim_n g_{i,n}(Y_n) = \lim_n g_{i,n}(X) = 0$. Since $\mathbf{P} (X
\in\partial B_{\gamma_i}(x_i) ) = 0$, it follows
\[
\prod_{i \in I} g_{i,n}(Y_n) -
\prod_{i \in I} g_{i,n}(X) \rightarrow0
\]
for $n \rightarrow\infty$ almost surely and dominated convergence yields
\[
\limsup_{n \to\infty} \mathbf{P} \biggl(X_n \in\bigcap
_{i \in I} B_{\gamma_i}(x_i) \biggr)
\leq\mathbf{P} \biggl(X \in\bigcap_{i
\in I}
B_{\gamma_i}(x_i) \biggr),
\]
just like in the proof of Theorem~\ref{teozet}. The lower bound
follows similarly.
\end{pf*}

We now head over to the case of c\`{a}dl\`{a}g functions. We only
discuss the approach in the proof of Theorem~\ref{teozet}.
Following exactly the same arguments as in the continuous case and
using the additional statements of Lemmas~\ref{lemdiff}~and~\ref{ana1}(a), it is easy to see that we also obtain~(\ref{ggnn}) if the
balls $B_{\gamma_i}(x_i)$ are defined with the uniform metric in
$\mathcal{D}[0,1]$.
Remember that we still have $x_i \in\mathcal{C}[0,1]$. Thus, Lemma
\ref{lemhelp}
yields the assertion.

The proof of Theorem~\ref{teosup} is close to the one of Lemma 5.3 in
\cite{DrJaNe08}. The $L_p$ approximation of the supremum norm
complicates the argument slightly. We only give all details in the
continuous case.

\begin{pf*}{Proof of  Theorem~\ref{teosup}}
Suppose $0 \leq s \leq3$ and that the first assumption of Theorem~\ref{teosup} is satisfied.
Let $\kappa\dvtx  \mathbb{R}^+_0 \to\mathbb{R}^+_0$ be a smooth,
monotonic function with
$\kappa(u) = 0$ for
$u \leq\frac{1} 2$ and $\kappa(u) = u^s$ for $u \geq1$. We may as well
assume that the interpolation for $\frac{1} 2\le u\le1$ is done
smoothly such that we have $\kappa(u) \le u^s$ for $\frac{1} 2\le u\le
1$, thus $\kappa(u) \le u^s$ for all $u\in\mathbb{R}^+_0$. Let
$f,f^{(p)}\dvtx \mathcal{C}[0,1]
\to\mathbb{R}$ be given by
\begin{eqnarray*}
f(x) &=& \kappa\bigl( \|x \|_\infty\bigr),
\\
f^{(p)}(x)& =& \kappa \bigl(L_p(x) \bigr).
\end{eqnarray*}
By Lemma~\ref{lemdiff}, the restrictions of $L_p$ and $f^{(p)}$ to
$\mathcal{C}[0,1]\setminus\{\mathbf{0}\}$ are smooth. Furthermore, all
derivatives of $f^{(p)}$ vanish for $\| x \|_\infty< 1/2$ which implies
that $f^{(p)}$ is smooth on $\mathcal{C}[0,1]$. Again, by Lemma~\ref{lemdiff} it is
easy to check that for any $k \in\{1, \ldots, m+1\}$,
\[
\bigl\|D^k f^{(p)}(x)\bigr\| = O \bigl(p^{k-1}
L_p^{s-k}(x) \bigr),
\]
uniformly in $p$ and $x \in\mathcal{C}[0,1]$. Let $x,y \in\mathcal
{C}[0,1]$ with $L_p(x),
L_p(y) \leq2 \|x-y\|_\infty$. Then
\begin{eqnarray*}
\bigl\|D^m f^{(p)}(x) - D^m f^{(p)}(y) \bigr\|
&\leq&\bigl\|D^m f^{(p)}(x)\bigr\| +\bigl\| D^m
f^{(p)}(y) \bigr\|
\\
&= &O \bigl(p^{m-1} \|x-y\|_\infty^\alpha
\bigr).
\end{eqnarray*}
Conversely let $2 \|x-y\|_\infty\leq L_p(x)$ [the case $2 \|x-y\|
_\infty\leq L_p(y)$ being analogous]. Then, by the mean value theorem,
there exists $z \in[x,y]:=\{\lambda x +(1-\lambda)y | \lambda\in
[0,1]\}$, such that
\begin{eqnarray*}
\bigl\|D^m f^{(p)}(x) - D^m f^{(p)}(y) \bigr\| &
=& \bigl\|D^{m+1} f^{(p)}(z)\bigr\| \cdot\| x-y \|_\infty
\\
& =& O \bigl(p^{m} L_p^{\alpha-1}(x) \bigr) \cdot\|x-y
\|_\infty
\\
&=&  O \bigl(p^m \| x-y\| _\infty^\alpha
\bigr).
\end{eqnarray*}
Hence, there is a constant $c > 0$ such that $c p^{-m}f^{(p)} \in
\mathcal{F}_s$ for all $p \geq4$. We define, for $r>0$,
\begin{eqnarray*}
f_r(x) &:=& c r^s f(x/r),
\\
f^{(p)}_r(x)&:=& c r^s f^{(p)} (x/r).
\end{eqnarray*}
Then $p^{-m}f_r^{(p)} \in\mathcal{F}_s$. Furthermore, $f_r(x)$ and
$f_r^{(p)}(x)$ are bounded by $c\|x\|^s$ for all $x\in\mathcal
{C}[0,1]$, uniformly
in $p$.
For any fixed $x$ we have $f_r(x) \rightarrow0$ and $\sup_{p \geq4}
f_r^{(p)}(x) \rightarrow0$ as $r \rightarrow\infty$. Hence, by
$\mathbf{E} [\| X\|^s ]<\infty$ and dominated
convergence this implies
%
%e34 #&#
%
\begin{equation}
\label{eq22} \mathbf{E} \Bigl[\sup_{p \geq4}f_r^{(p)}(X)
\Bigr] \rightarrow0 \qquad(r\to\infty).
\end{equation}
By the definition of $\zeta_s$, we have
\[
\mathbf{E} \bigl[f_r^{(p)}(X_n) \bigr] \leq
\mathbf{E} \bigl[f_r^{(p)}(X) \bigr] + p^m
\zeta_s(X_n,X).
\]
By the definition of $f_r$, for
$\|x \| > r $ we have $\|x\|^s = c^{-1} f_r(x)$. Hence,
%
%e35 #&#
%
\begin{eqnarray} \label{eq23}
&& \mathbf{E} \bigl[\|X_n\|_\infty^s
\mathbf{1}_{ \{ \|X_n\|_\infty
\geq2r \} } \bigr]\nonumber
\\
&&\qquad = c^{-1} \mathbf{E}
\bigl[f_r(X_n) \mathbf{1}_{ \{ \|X_n\|_\infty
\geq2r \} } \bigr]\nonumber
\\
&&\qquad \leq c^{-1} \mathbf{E} \bigl[f_r^{(p)}(X_n)
\bigr] + c^{-1} \bigl(\mathbf{E} \bigl[ \bigl(f_r(X_n)
- f_r^{(p)}(X_n) \bigr) \mathbf{1}_{ \{ \|X_n\|
_\infty\geq2r \} }
\bigr] \bigr)
\\
&&\qquad \leq c^{-1} \mathbf{E} \bigl[f_r^{(p)}(X)
\bigr] + c^{-1} p^m \zeta_s(X_n,X)
\nonumber
\\
&&\quad\qquad{} + c^{-1} \bigl(\mathbf{E} \bigl[ \bigl(f_r(X_n)
- f_r^{(p)}(X_n) \bigr) \mathbf{1}_{ \{ \|X_n\| _\infty\geq2r \} }
\bigr] \bigr).\nonumber
\end{eqnarray}
Now, let $\varepsilon> 0$ be arbitrary. By~(\ref{eq22}), fix $r > 0$
such that $\mathbf{E} [f_r^{(p)}(X) ] < \varepsilon$ for
all $p \geq4$.
Additionally, by the given assumptions there exists a sequence $p_n
\uparrow\infty$ such that
\[
\frac{\log r_n}{p_n} \rightarrow0, \qquad p_n^m
\zeta_s(X_n,X) \rightarrow0 \qquad(n \rightarrow\infty).
\]
Therefore, let $N_0$ be large enough such that $p_n^m \zeta_s(X_n,X) <
\varepsilon$ for all $n \geq N_0$. It remains to bound the third
summand in~(\ref{eq23}).
Using Lemma~\ref{ana1}(a), %with $g = \mathbf{0}
piecewise linearity of $X_n$ implies that for all $0 < \vartheta< 1$,
\[
L_p(X_n) \geq\|X_n\|_\infty(1-
\vartheta) \biggl( \frac{\vartheta
r_n}{2} \biggr)^{1/p_n}.
\]
In particular, we have $L_p(X_n) \geq\frac{\|X_n\|_\infty}{2}$ for all
$n$ sufficiently large.
For those $n$ and $\| X_n \| > 2r $ we also have $f_r^{(p)}(X_n) =
cL_p^s(X_n)$. This yields
%
%e36 #&#
%e37 #&#
%
\begin{eqnarray}
&& \mathbf{E} \bigl[ \bigl(f_r(X_n) -
f_r^{(p)}(X_n) \bigr) \mathbf{1}_{ \{ \|X_n\|
_\infty\geq2r \} }
\bigr]
\nonumber\\[-8pt]\label{eq27}\\[-8pt]
&&\qquad  = c \mathbf{E} \bigl[ \bigl( \|X_n\|_\infty^s
- L^s_{p}(X_n) \bigr) \mathbf{1}_{ \{ \|X_n\|_\infty\geq2r \} }
\bigr] \nonumber
\\
&&\qquad \leq c \bigl(1 - 2^{-s} \bigr) \mathbf{E} \bigl[\|X_n
\|_\infty^s \mathbf{1}_{ \{ \|X_n\|_\infty\geq2r \} } \bigr] \label{eq24}
\end{eqnarray}
for all $n$ sufficiently large. Increasing $N_0$ if necessary,
inserting~(\ref{eq24}) into~(\ref{eq23}) and rearranging terms implies
\[
\mathbf{E} \bigl[\|X_n\|_\infty^s
\mathbf{1}_{ \{ \|X_n\|_\infty
\geq2r \} } \bigr] \leq2^{1+s} c^{-1}
\varepsilon
\]
for all $n \geq N_0$. Since $\varepsilon$ was arbitrary, the assertion
follows.

Now, suppose the second assumption is satisfied. Then we
have to modify the
last part of the proof. In~(\ref{eq27}), we can decompose
\[
L_p^s(X_n) = L_p^s(X_n)
\mathbf{1}_{\{ Y_n \in\mathcal{C}_{r_n}[0,1]\}
} + L_p^s(X_n)
\mathbf{1}_{\{ Y_n \notin\mathcal{C}_{r_n}[0,1]\}}.
\]
Using $L_p^s(X_n) \leq\|X_n\|_\infty^s$, the assumptions guarantee the
expectation of the second term to be small in the limit $n \to\infty$.
For the first
one, using similar arguments as above, given $\{ Y_n \in\mathcal
{C}_{r_n}[0,1]\}$, we find
\[
L_p(X_n) \geq\frac{\|X_n\|_\infty}{2} - 2
\varrho_n
\]
with $\varrho_n = \|h_n - h\|_\infty$ for all $n$ sufficiently large.
Proceeding as in the first part, we obtain the result.
Given the third
assumption, it only remains to bound $\mathbf{E} [f_r^{(p)}(Y_n)
]$ which
appears instead of $\mathbf{E} [f_r^{(p)}(X) ]$ by
$\mathbf{E} [f_r^{(p)}(Z) ]$ in~(\ref{eq23}).
\end{pf*}

%%
%%
%%
%%
%%
%%

%s3 #&#
\section{The contraction method} \label{nnnCM}
In this section, the contraction method is developed first for a
general separable Banach space $B$.
Then the framework is specialized to the cases $(\mathcal{C}[0,1], \|
\cdot\|_\infty)$
and $(\mathcal{D}[0,1], d_{\mathrm{sk}})$. For this section, $B$ will always
denote a
separable Banach space or $(\mathcal{D}[0,1], d_{\mathrm{sk}})$.

We recall the recursive equation~(\ref{rec2}). We have
%
%e38 #&#
%
\begin{equation}
\label{rec} X_n \stackrel{d} {=} \sum
_{r=1}^K A_r^{(n)}
X_{I_r^{(n)}}^{(r)} + b^{(n)}, \qquad n \geq
n_0,
\end{equation}
where $A_1^{(n)}, \ldots, A_K^{(n)}$ are random continuous linear
operators, $b^{(n)}$ is a $B$-valued random variable,
$(X_n^{(1)})_{n\ge0},\ldots,(X_n^{(K)})_{n\ge0}$ are distributed like\break
$(X_n)_{n\ge0}$, and $I^{(n)}=(I_1^{(n)}, \ldots, I_K^{(n)})$ is a
vector of random integers in $\{0,\ldots,n\}$. Moreover, $(A_1^{(n)},
\ldots, A_K^{(n)}, b^{(n)}, I^{(n)})$, $(X_n^{(1)})_{n\ge0}, \ldots,
(X_n^{(K)})_{n\ge0}$ are independent and $n_0\in\mathbb N$.

Recall that in order to be a random continuous linear operator, $A$ has
to take values in the set of continuous endomorphisms on $\mathcal{C}[0,1]$,
respectively, the set of norm-continuous endomorphisms that are
continuous with respect to $d_{\mathrm{sk}}$ on $\mathcal{D}[0,1]$ such that
$A(x)(t)$ is a
real-valued random variable for all $x \in\mathcal{C}[0,1]$,
respectively, $x\in\mathcal{D}[0,1]$
and $t\in[0,1]$. In $\mathcal{D}[0,1]$, we additionally have to
guarantee $\|A\|
_{\mathrm{op}}$ to be a real-valued random variable; see Section~\ref
{subsecdsk}.

We make assumptions about the moments and the asymptotic behavior of
the coefficients $A_1^{(n)}, \ldots, A_K^{(n)}$, $b^{(n)}$. For a
random continuous linear operator $A$, we write
\[
\|A\|_s:= \mathbf{E} \bigl[\|A\|_\mathrm{op}^s
\bigr]^{1\wedge(1/s)}.
\]
We consider the following conditions with an $s > 0$:
\begin{longlist}[(C2)]
\item[(C1)] We have $\|X_0\|_s,\ldots,\|X_{n_0-1}\|_s$, $\|
A_r^{(n)}\|_s, \|b^{(n)}\|_s<\infty$ for all $r=\break 1,\ldots,K$ and $n\ge
0$ and there exist random continuous linear operators $A_1, \ldots,
A_K$ on $B$ and a $B$-valued random variable $b$ such that, as $n\to
\infty$,
%
%e39 #&#
%
\begin{equation}
\label{eqrate1} \gamma(n):= \bigl\|b^{(n)} - b\bigr\|_s+\sum
_{r=1}^K \bigl(\bigl\|A_r^{(n)} -
A_r\bigr\| _s + \bigl\llVert{\mathbf1}_{\{I^{(n)}_r \le n_0\}}
A^{(n)}_r \bigr\rrVert_s \bigr) \to0
\end{equation}
and for all $\ell\in\mathbb N$,
%
%e40 #&#
%
\begin{equation}
\label{inicond} \mathbf{E} \bigl[{\mathbf1}_{\{I^{(n)}_r \in\{0,\ldots,\ell
\}\cup\{
n\}\}} \bigl\|
A^{(n)}_r \bigr\| _\mathrm{op}^s \bigr]
\to0.
\end{equation}
\end{longlist}

\begin{longlist}[(C2)]
\item[(C2)] We have
\[
L:= \sum_{r=1}^K \mathbf{E} \bigl[
\|A_r\|_\mathrm{op}^s \bigr] < 1.
\]
\end{longlist}
The limits of the coefficients determine the limiting operator $T$ from
(\ref{limitmap}):
%
%e41 #&#
%
\begin{eqnarray} \label{deftm}
T\dvtx \mathcal{M}(B) &\to&\mathcal{M}(B),
\nonumber\\[-8pt]\\[-8pt]
\mu&\mapsto&\mathcal{L} \Biggl( \sum_{r=1}^K
A_r Z^{(r)} + b \Biggr),\nonumber
\end{eqnarray}
where $(A_1, \ldots, A_K,b)$, $Z^{(1)}, \ldots, Z^{(K)}$ are
independent and $Z^{(1)}, \ldots, Z^{(K)}$ have distribution $\mu$.
\begin{longlist}[(C3)]
\item[(C3)] The map $T$ has a fixed point $\eta\in{\mathcal
M}_s(B)$, such that ${\mathcal L}(X_n) \in{\mathcal M}_s(\eta)$ for
all $n\ge n_0$.
\end{longlist}
The existence of a fixed point is not in general implied by contraction
properties of~$T$ with respect to a Zolotarev metric due to the lack of
knowledge of completeness of the metric on a the space $B$. However, we
can argue that there is at most one fixed point of $T$ in ${\mathcal
M}_s(\eta)$:

%le18 #&#
%
\begin{lem} \label{lemuniqueness}
Assume the sequence $(X_n)_{n\ge0}$ satisfies~(\ref{rec}). Under
conditions \textup{(C1)--(C3)}, we have $T({\mathcal M}_s(\eta))
\subseteq{\mathcal M}_s(\eta)$ and
\[
\zeta_s \bigl(T(\mu), T(\lambda) \bigr) \leq L \zeta_s(
\mu, \lambda)\qquad\mbox{for all } \mu,\lambda\in{\mathcal M}_s(
\eta).
\]
In particular, the restriction of $T$ to ${\mathcal M}_s(\eta)$ is a
contraction and has the unique fixed-point $\eta$.
\end{lem}

\begin{pf}
Let $\mu\in{\mathcal M}_s(\eta)$. Recall that we have $s=m+\alpha$
with $m\in\mathbb N_0$ and $\alpha\in(0,1]$. We introduce an
accompanying sequence
%
%e42 #&#
%
\begin{equation}
\label{defacc} \qquad Q_n:= \sum_{r=1}^K
A_r^{(n)} \bigl({\mathbf1}_{ \{ I^{(n)}_r<
n_0 \}}
X^{(r)}_{I^{(n)}_r} + {\mathbf1}_{ \{ I^{(n)}_r\ge
n_0 \}} Z^{(r)}
\bigr) + b^{(n)}, \qquad n\ge n_0,
\end{equation}
where $(A^{(n)}_1, \ldots, A^{(n)}_K, b^{(n)})$, $Z^{(1)}, \ldots,
Z^{(K)}$ are independent and $Z^{(1)}, \ldots, Z^{(K)}$ have
distribution $\mu$.

We first show that $\mathcal{L}(Q_n) \in{\mathcal M}_s(\eta)$ for
all $n\ge n_0$.
Condition (C1), conditioning on the coefficients and Minkowski's
inequality, implies
$\mathbf{E} [\|Q_n\|_\infty^s ]<\infty$ for all $n$.
For $s\le1$, we already obtain $\mathcal{L}(Q_n) \in{\mathcal
M}_s(\eta)$.

For $s>1$, we choose arbitrary
$1 \leq k \leq m$ and multilinear and
bounded $f\dvtx  B^k \to\mathbb{R}$. We have
\begin{eqnarray*}
\mathbf{E} \bigl[f(Z, \ldots, Z) \bigr] & =& \mathbf{E} \bigl[f(X_n,
\ldots, X_n) \bigr]
\\
& =& \mathbf{E} \Biggl[f \Biggl(\sum_{r=1}^K
A_r^{(n)} X_{I_r^{(n)}}^{(r)} +
b^{(n)}, \ldots, \sum_{r=1}^K
A_r^{(n)} X_{I_r^{(n)}}^{(r)} +
b^{(n)} \Biggr) \Biggr].
\end{eqnarray*}
To show $\mathcal{L}(Q_n) \in{\mathcal M}_s(\eta)$, we need to
verify that
the latter display is equal to
$\mathbf{E} [f(Q_n,\ldots, Q_n) ]$.
%\begin{eqnarray*}
%&= & \E{f(Q_n,\ldots, Q_n)} \\
%& = & \E{f \left( \sum_{r=1}^K A_r \circ Z^{(r)} + b, \ldots,
%\sum_{r=1}^K A_r \circ Z^{(r)} + b \right)} \\
%& = & \E{f(Z',\ldots, Z')},
%\end{eqnarray*}
Since $f$ is multilinear, both terms can be expanded as a sum and it
suffices to show that the corresponding summands are equal:
%
%e43 #&#
%
\begin{eqnarray}
\label{eq7} \mathbf{E} \bigl[f \bigl(C_{j_1}^{(n)},
\ldots, C_{j_k}^{(n)} \bigr) \bigr] &= & \mathbf{E} \bigl[f
\bigl(D_{j_1}^{(n)}, \ldots, D_{j_k}^{(n)}
\bigr) \bigr],
\end{eqnarray}
where $j_1,\ldots,j_k\in\{1,\ldots,K\}$ and for each $i\in\{
1,\ldots,k\}$ we either have
%
%e44 #&#
%
\begin{equation}
\label{eq7aa} \qquad C_{j_i}^{(n)}= A_{j_i}^{(n)}
X_{I_{j_i}^{(n)}}^{(j_i)}\quad\mbox{and}\quad D_{j_i}^{(n)}=
A_{j_i}^{(n)} \bigl({\mathbf1}_{ \{
I^{(n)}_{j_i}< n_0 \}}
X^{(j_i)}_{I^{(n)}_{j_i}} + {\mathbf1}_{
\{ I^{(n)}_{j_i}\ge n_0 \}} Z^{(j_i)}
\bigr)
\end{equation}
or
%
%e45 #&#
%
\begin{equation}
\label{eq7bb} C_{j_i}^{(n)}=b^{(n)} \quad\mbox{and}
\quad D_{j_i}^{(n)}=b^{(n)}.
\end{equation}
The equality in~(\ref{eq7}) is obvious for the case where we have
(\ref
{eq7bb}) for all $i=1,\ldots,k$. For the other cases, we have
(\ref{eq7aa}) for at least $1 \leq\ell\leq k$ arguments of $f$, say,
for simplicity of presentation, for the first $\ell$ with
$1 \leq\ell_1 < \cdots< \ell_d = \ell$
such that
$j_s = j_{\ell_i}$ for all $s = \ell_{i-1}+1, \ldots, \ell_{i}, i = 1,
\ldots, d$ and $j_{\ell_i}$ pairwise different for $i = 1, \ldots, d$
(by convention $\ell_0:=0$).
The claim in~(\ref{eq7}) reduces to
%
%e46 #&#
%
\begin{eqnarray}\label{eq71}
&& \mathbf{E} \bigl[f \bigl(C_{j_{\ell_1}}^{(n)},
\ldots, C_{j_{\ell
_1}}^{(n)}, C_{j_{\ell_2}}^{(n)},
\ldots, C_{j_{\ell_d}}^{(n)}, b^{(n)},\ldots,
b^{(n)} \bigr) \bigr]
\nonumber\\[-8pt]\\[-8pt]
&&\qquad = \mathbf{E} \bigl[f \bigl(D_{j_{\ell_1}}^{(n)}, \ldots,
D_{j_{\ell
_1}}^{(n)}, D_{j_{\ell_2}}^{(n)}, \ldots,
D_{j_{\ell
_d}}^{(n)},b^{(n)},\ldots, b^{(n)} \bigr)
\bigr].\nonumber
\end{eqnarray}
We will prove that, for each $p \in\{1, \ldots, d \}$,
%
%e47 #&#
%e48 #&#
%
\begin{eqnarray}
\label{eq72} \qquad && \mathbf{E} \bigl[f \bigl(C_{j_{\ell_1}}^{(n)},
\ldots, C_{j_{\ell
_{p-1}}}^{(n)}, C_{j_{\ell_p}}^{(n)},
\ldots, C_{j_{\ell_p}}^{(n)}, D_{j_{\ell_{p+1}}}^{(n)},
\ldots, D_{j_{\ell
_d}}^{(n)},b^{(n)},\ldots, b^{(n)}
\bigr) \bigr]
\\
&&\qquad = \mathbf{E} \bigl[f \bigl(C_{j_{\ell_1}}^{(n)}, \ldots,
C_{j_{\ell
_{p-1}}}^{(n)}, D_{j_{\ell_p}}^{(n)}, \ldots,
D_{j_{\ell_p}}^{(n)},
\nonumber\\[-8pt]\\[-8pt]
&&\hspace*{66pt} D_{j_{\ell_{p+1}}}^{(n)}, \ldots,D_{j_{\ell
_d}}^{(n)},b^{(n)},\ldots, b^{(n)} \bigr)\bigr],\nonumber
\end{eqnarray}
which in turn implies~(\ref{eq71}). Abbreviating $Y_i^{(r)} =
({\mathbf1}_{ \{ i< n_0 \}}
X^{(r)}_{i} + {\mathbf1}_{ \{ i \ge n_0 \}} Z^{(r)} )$ and
denoting by $\Upsilon$ the joint distribution of $ (A_{j_{\ell
_1}}^{(n)}, \ldots, A_{j_{\ell_d}}^{(n)}, I_{j_{\ell_1}}^{(n)},
\ldots,
I_{j_{\ell_d}}^{(n)}, b^{(n)} )$
we have
\begin{eqnarray*}
&& \mathbf{E} \bigl[f\bigl(C_{j_{\ell_1}}^{(n)}, \ldots,
C_{j_{\ell
_{i-1}}}^{(n)}, C_{j_{\ell_i}}^{(n)}, \ldots,
C_{j_{\ell_i}}^{(n)}, D_{j_{\ell_{i+1}}}^{(n)}, \ldots,
D_{j_{\ell
_d}}^{(n)},b^{(n)},\ldots, b^{(n)} \bigr)\bigr]
\\
&&\qquad = \int f (\alpha_1 x_1, \ldots, \alpha_{p-1}
x_{p-1}, \alpha_{p} x_p, \ldots,
\alpha_{p} x_p,
\\
&&\hspace*{94pt} \alpha_{p+1}x_{p+1},\ldots, \alpha_d x_d, b, \ldots, b )
\\
&&\hspace*{41pt}{}\times d \mathbb{P}_{X_{i_1}}(x_1) \cdots d \mathbb{P}_{X_{i_p}}(x_{p})
\,d \mathbb{P}_{Y_{i_{p+1}}}(x_{p+1}) \cdots d \mathbb{P}_{Y_{i_d}}(x_{d})
\\
&&\hspace*{41pt}{}\times
d\Upsilon(\alpha_1, \ldots, \alpha_d, i_1,
\ldots, i_d, b) \label{eq7dd}
\\
&&\qquad = \int\mathbf{E} \bigl[g(X_{i_p}, \ldots, X_{i_p}) \bigr]
\,d \mathbb{P}_{X_{i_1}} \cdots d \mathbb{P}_{X_{i_{p-1}}} \,d
\mathbb{P}_{Y_{i_{p+1}}} \cdots d \mathbb{P}_{Y_{i_d}} \,d\Upsilon,
%&\E{f \left(D_{j_1}^{(n)}, \ldots, D_{j_k}^{(n)} \right)} \nonumber\\
%&=\int\E{f\left(\alpha_1 \circ
%\left({\mathbf1}_{\{ m< n_0\}} X^{(j_1)}_m + {\mathbf1}_{\{m\ge n_0\}}
%Z^{(j_1)}\right), \beta_2,\ldots, \beta_k\right)} \,d\Upsilon(\alpha_1,
%\beta_2\ldots,\beta_k,m), \label{eq7ee}
\end{eqnarray*}
where, for all fixed $\alpha_1, \ldots, \alpha_d, i_1, \ldots, i_d, b,
x_1, \ldots, x_{p-1}, x_{p+1}, \ldots, x_d$, we use the bounded and
multilinear function ${g\dvtx  B^{\ell_p - \ell_{p-1}} \to\mathbb{R}}$,
\begin{eqnarray*}
&& g(y_1, \ldots, y_{\ell_p - \ell_{p-1}})
\\
&&\qquad:= f (\alpha_1 x_1, \ldots, \alpha_{p-1}
x_{p-1}, \alpha_p y_1, \ldots,
\alpha_{p} y_{\ell_p - \ell_{p-1}},
\\
&&\hspace*{111pt} \alpha_{p +1}
x_{p+1}, \ldots, \alpha_d x_d,b, \ldots, b).
\end{eqnarray*}
Since ${\mathcal L}(X_m), {\mathcal L}(Z) \in{\mathcal M}_s(\eta)$ for
all $m\ge n_0$ we can replace $X_{i_p}$ by $Y_{i_p}$.
This shows the equality~(\ref{eq72}), hence~(\ref{eq7}).
Altogether, we obtain ${\mathcal L}(Q_n) \in{\mathcal M}_s(\eta)$ for
all $n\ge n_0$.
%\begin{eqnarray} \label{eq7}
%\E{f \left(A_{j_1}^{(n)} \circ X_{I_{j_1}^{(n)}}^{(j_1)}, \ldots,
%A_{j_k}^{(n)} \circ X_{I_{j_k}^{(n)}}^{(j_k)} \right)} &= &
%\E{f \left(A_{j_1}^{(n)} \circ Z^{(j_1)}, \ldots, A_{j_k}^{(n)} \circ
%Z^{(j_k)} \right)}
%\end{eqnarray}
%\begin{eqnarray} \label{eq7}
%\E{f \left(A_{j_1}^{(n)} \circ X_{I_{j_1}^{(n)}}^{(j_1)}, \ldots,
%A_{j_k}^{(n)} \circ X_{I_{j_k}^{(n)}}^{(j_k)} \right)} &= &
%\E{f \left(A_{j_1}^{(n)} \circ Z^{(j_1)}, \ldots, A_{j_k}^{(n)} \circ
%Z^{(j_k)} \right)}
%\end{eqnarray}

Now, we show $T(\mu)\in{\mathcal M}_s(\eta)$. Let $W$ be a random
variable with distribution
$T(\mu)$.
By (C2), in particular, $\|A_r\|_s <\infty$ for
$r=1,\ldots,K$, by (C1) we have $\|b\|_s<\infty$. Thus, as for
$Q_n$, from Minkowski's inequality we obtain \mbox{$\mathbf{E} [\|W\|
_\infty^s ]<\infty
$}, hence $T(\mu)\in{\mathcal M}_s(\eta)$ for $s\le1$. For the case
$s>1$, we consider again arbitrary
$1 \leq k \leq m$ and multilinear and
bounded $f\dvtx  B^k \to\mathbb{R}$. It suffices to show $\mathbf{E}
[f(Q_n,\ldots,Q_n) ]=\mathbf{E} [f(W,\ldots,W)
]$ for some $n\ge n_0$. In fact, we will show that $\lim_{n\to\infty}
\mathbf{E} [f(Q_n,\ldots,Q_n) ]=\mathbf
{E} [f(W,\ldots,W) ]$.
For this, we expand
\[
\mathbf{E} \bigl[f(W,\ldots,W) \bigr]=\mathbf{E} \Biggl[f \Biggl( \sum
_{r=1}^K A_r Z^{(r)} + b,
\ldots, \sum_{r=1}^K A_r
Z^{(r)} + b \Biggr) \Biggr]
\]
into summands corresponding to~(\ref{eq7}) and have to show that
%
%e49 #&#
%
\begin{equation}
\label{neq8} \lim_{n\to\infty} \mathbf{E} \bigl[f
\bigl(D_{j_1}^{(n)}, \ldots, D_{j_k}^{(n)}
\bigr) \bigr]= \mathbf{E} \bigl[f(E_{j_1}, \ldots, E_{j_k})
\bigr],
\end{equation}
where $j_1,\ldots,j_k\in\{1,\ldots,K\}$. For each $i\in\{1,\ldots,k\}
$, we have in case~(\ref{eq7aa}) that $E_{j_i}= A_{j_i} Z^{(j_i)}$, in
case~(\ref{eq7bb}) that $E_{j_i}=b$.
We obtain, introducing a telescoping sum and using H\"older's inequality,
\begin{eqnarray*}
&& \bigl\llvert\mathbf{E} \bigl[f \bigl(D_{j_1}^{(n)},
\ldots, D_{j_k}^{(n)} \bigr) \bigr]- \mathbf{E}
\bigl[f(E_{j_1}, \ldots, E_{j_k}) \bigr] \bigr\rrvert
\\
&&\qquad= \Biggl\llvert\sum_{q=1}^k
\mathbf{E} \bigl[f \bigl(E_{j_1}, \ldots, E_{j_{q-1}},
D_{j_q}^{(n)}, \ldots, D_{j_k}^{(n)} \bigr)
\\
&&\hspace*{61pt}{} - f \bigl(E_{j_1}, \ldots, E_{j_{q}}, D_{j_{q+1}}^{(n)},
\ldots, D_{j_k}^{(n)} \bigr) \bigr] \Biggr\rrvert
\\
&&\qquad\le\sum_{q=1}^k \bigl\llvert
\mathbf{E} \bigl[f \bigl(E_{j_1}, \ldots, E_{j_{q-1}},
D_{j_q}^{(n)}-E_{j_q}, D_{j_{q+1}}^{(n)},
\ldots, D_{j_k}^{(n)} \bigr) \bigr] \bigr\rrvert
\\
&&\qquad\le\sum_{q=1}^k \|f\| \bigl\|
D_{j_q}^{(n)}- E_{j_q}\bigr\|_k \prod
_{v=1}^{q-1} \|E_{j_v}
\|_k \prod_{v=q+1}^k
\bigl\|D_{j_v}^{(n)}\bigr\|_k.
\end{eqnarray*}
Note that the $\|E_{j_v}\|_k$ and $\|D_{j_v}^{(n)}\|_k$ are all
uniformly bounded by independence, (C1), and $\|X_0\|_s, \ldots,\|
X_{n_0-1}\|_s$, $\|Z\|_s <\infty$. Hence, it suffices to show that $\|
D_{j_v}^{(n)}- E_{j_v}\|_k \to0$ for all $j_v$. In case~(\ref{eq7bb}),
this is $\|b^{(n)}-b\|_k \to0$ by condition~(C1). In case (\ref
{eq7bb}), we have, abbreviating $r=j_i$,
\begin{eqnarray*}
&& \bigl\llVert A_{r}^{(n)} \bigl({\mathbf1}_{ \{ I^{(n)}_{r}< n_0 \}}
X^{(r)}_{I^{(n)}_{r}} + {\mathbf1}_{ \{ I^{(n)}_{r}\ge n_0 \}} Z^{(r)}
\bigr)- A_{r} Z^{(r)} \bigr\rrVert_k
\\
&&\qquad \le \bigl\llVert \bigl(A_{r}^{(n)}- A_{r}
\bigr) Z^{(r)} \bigr\rrVert_k + \bigl\llVert
A_{r}^{(n)} \bigl({\mathbf1}_{ \{ I^{(n)}_{r}< n_0 \}}
\bigl(X^{(r)}_{I^{(n)}_{r}}- Z^{(r)} \bigr) \bigr) \bigr
\rrVert_k.
\end{eqnarray*}
The first summand of the latter display tends to zero by independence,
\mbox{$\|Z\|_s<\infty$} and condition~(C1). The second summand tends to
zero applying H\"older's inequality, condition~(C1), which implies
that $\|A_{r}^{(n)}\|_s$ in uniformly bounded, $\|X_0\|_s,\ldots,\|
X_{n_0-1}\|_s,\|Z\|_s<\infty$ and conditions (C1) and (C3).
Altogether we obtain $T(\mu) \in{\mathcal M}_s(\eta)$.

Let $\mu, \lambda\in{\mathcal M}_s(\eta)$.
Conditioning on the coefficients, using Lemma~\ref{lemtrans} and
(\ref
{ungldo}), it follows that
\[
\zeta_s \bigl(T(\mu), T(\lambda) \bigr) \leq \Biggl(\sum
_{r=1}^K \mathbf{E} \bigl[\|A_r\|
_\mathrm{op}^s \bigr] \Biggr) \zeta_s(\mu,
\lambda).
\]
Thus, by condition~(C2), the restriction of $T$ to ${\mathcal
M}_s(\eta)$ is a contraction with respect to $\zeta_s$.

Assume, $\mu$ was a fixed point of $T$ as well. Then the contraction
property implies
\[
\zeta_s(\mu, \eta)= \zeta_s \bigl(T(\mu), T(\eta) \bigr)
\le L \zeta_s(\mu, \eta),
\]
hence $\zeta_s(\mu, \eta)=0$. Since the $\zeta_s$-distance is a metric
on ${\mathcal M}_s(\eta)$ it follows $\mu= \eta$.
\end{pf}

We now turn to the problem of convergence of the sequence $(X_n)_{n\ge
0}$ to the fixed-point
$\eta$.

Aiming to proof $X_n \rightarrow X$ condition~(C1) is natural in
the context of contraction method. Condition (C2) is necessary if
working with $\zeta_s$ metrics. We will discuss this in detail for the
cases $\mathcal{C}[0,1]$ and $\mathcal{D}[0,1]$ below. The existence
of a solution of the
fixed-point equation in condition~(C3) is required since we miss
knowledge about completeness of the $\zeta_s$ metrics. If $\mu\in
\mathcal{M}_s(B)
$, then $(T^n(\mu))_{n \ge0}$ is a Cauchy sequence with respect to
$\zeta_s$, the proof being similar to the one of the previous lemma.
Then, for $B=\mathcal{C}[0,1]$ or $B=\mathcal{D}[0,1]$, by
Proposition~\ref{propcomp}, all
finite-dimensional\vspace*{1pt} marginals of $T^n(\mu)$ converge to the
corresponding marginals of some measure $\nu$ on $\mathbb
{R}^{[0,1]}$, the
natural candidate for a fixed-point of~(\ref{deftm}). In the
application discussed in Section~\ref{secpm}, the solution of the
fixed-point equation~(\ref{sulzbachlimit}) is constructed via a
sequence $(Z_n)_{n \ge0}$ of random variables that satisfy $\mathcal{L}(Z_n)
= T^n(\mu)$ and converge uniformly almost surely (cf. \cite{BrNeSu11}
for details). The starting point is the Dirac measure $\mu= \delta_f$
with a specific function $f \in\mathcal{C}[0,1]$.

The following proposition uses the ideas developed so far to infer
convergence of $X_n$ to $X$ in the $\zeta_s$ distance. The proof
extends a similar proof for the case $B=\mathbb{R}^d$; see \cite{NeRu04}, Theorem 4.1. We draw further implications from this proof; see
Corollary~\ref{cont2}.

%pr19 #&#
%
\begin{prop} \label{cont}
Let $(X_n)_{n\ge0}$ satisfy recurrence~(\ref{rec}) with conditions
\textup{\mbox{(C1)--(C3)}}. Then for the fixed-point $\eta={\mathcal L}(X)$ of $
T$ in~(\ref{deftm}) we have, as $n\to\infty$,
\[
\zeta_s(X_n, X) \rightarrow0.
\]
\end{prop}

\begin{pf}
We use the accompanying sequence defined in~(\ref{defacc}).
Throughout the proof, let $n \geq n_0$.
Again since the $\zeta_s$-distance is a metric, we have
%
%e50 #&#
%
\begin{equation}
\label{eqtriangle} \zeta_s(X_n, X) \leq\zeta_s(X_n,
Q_n) + \zeta_s(Q_n, X).
\end{equation}
First, we consider the second term. By (C1) and Minkowski's
inequality, absolute moments of order $s$ of the sequence $(Q_n)_{n
\geq n_0}$ are
bounded, hence using Lemma~\ref{lemls} it suffices to show
\[
\ell_s(Q_n, X) \rightarrow0.
\]
Using the same set of independent random variables $X^{(1)}, \ldots,
X^{(K)}$ for $Q_n$ and in the
recurrence of $X$, we obtain
\begin{eqnarray*}
\ell_s(Q_n, X) &\leq& \Biggl\llVert\sum
_{r=1}^K \bigl(A_r - {\mathbf
1}_{
\{ I^{(n)}_{r} \geq n_0 \}} A_r^{(n)} \bigr) X^{(r)}
\Biggr\rrVert_s + \Biggl\llVert\sum_{r=1}^K
{\mathbf1}_{ \{ I^{(n)}_{r} < n_0 \}} A_r^{(n)} X^{(r)}_{I_r^{(n)}}
\Biggr\rrVert_s
\\
&&{} +\bigl\|b^{(n)} - b\bigr\|_s
\\
& \leq& \sum_{r=1}^K \bigl( \bigl\llVert
A_r - A_r^{(n)} \bigr\rrVert_s +
\bigl\llVert{\mathbf1}_{ \{ I^{(n)}_{r} < n_0 \}} \bigl\|A_r^{(n)}
\bigr\|_{\mathrm
{op}} \bigr\rrVert_s \bigr)\llVert X \rrVert
_s +\bigl\|b^{(n)} - b\bigr\|_s
\\
&&{}+ \Biggl\llVert\sum_{r=1}^K {
\mathbf1}_{ \{ I^{(n)}_{r} < n_0 \}} A_r^{(n)} X^{(r)}_{I_r^{(n)}}
\Biggr\rrVert_s.
\end{eqnarray*}
By (C1) the first two summands tend to zero. Also, the third one
converges to zero using (C1) and
\[
\bigl\llVert{\mathbf1}_{ \{ I^{(n)}_{r} < n_0 \}} \bigl\|A_r^{(n)}\bigr\|
_{\mathrm{op}} X^{(r)}_{I_r^{(n)}} \bigr\rrVert_s \leq
\bigl\llVert{\mathbf1}_{
\{ I^{(n)}_{r} < n_0 \}} \bigl\|A_r^{(n)}
\bigr\|_{\mathrm{op}} \bigr\rrVert_s \Bigl\llVert\sup
_{j < n_0} \|X_j\| \Bigr\rrVert_s.
\]

Furthermore, conditioning on the coefficients and using that $\zeta_s$
is $(s,+)$ ideal and Lemma~\ref{lemtrans}, it is easy to see that
%
%e51 #&#
%e52 #&#
%
\begin{eqnarray}
\zeta_s(Q_n, X_n) & \leq& p_n
\zeta_s(X_n,X) + \mathbf{E} \Biggl[\sum
_{r=1}^K {\mathbf1}_{ \{ n_0 \leq
I^{(n)}_{r} \leq n-1 \} }
\bigl\|A_r^{(n)}\bigr\|_{\mathrm{op}}^s \zeta
_s (X_{I_r^{(n)}},X ) \Biggr]\hspace*{-35pt} \label{eqqx1}
\\
&\leq& p_n \zeta_s(X_n,X) + \Biggl( \sum
_{r=1}^K \mathbf{E} \bigl[
\bigl\|A_r^{(n)}\bigr\| _{\mathrm{op}} ^s \bigr]
\Biggr)\sup_{n_0 \leq i
\leq n-1} \zeta_s(X_i, X),
\label{eqqx2}
\end{eqnarray}
where
\[
p_n = \mathbf{E} \Biggl[\sum_{r=1}^K
{\mathbf1}_{ \{ I^{(n)}_{r} =
n \}} \bigl\| A_r^{(n)}
\bigr\|_{\mathrm{op}}^s \Biggr] \rightarrow0, \qquad n \rightarrow\infty.
\]
Combining~(\ref{eqtriangle}) and~(\ref{eqqx2}) implies
\[
\zeta_s(X_n, X) \leq\frac{1}{1-p_n} \Biggl[\sum
_{r=1}^K \mathbf{E} \bigl[\bigl\|
A_r^{(n)}\bigr\|_{\mathrm{op}}^s \bigr] \sup
_{n_0 \leq i \leq
n-1} \zeta_s(X_i, X) + o(1)
\Biggr].
\]
From this, it follows that $\zeta_s(X_n,X)$ is bounded. Let
\[
\bar\eta:= \sup_{n \geq n_0} \zeta_s(X_n,X),
\qquad\eta:= \limsup_{n
\rightarrow\infty} \zeta_s(X_n,X)
\]
and
$\varepsilon> 0$ arbitrary. Then there exists $\ell> 0$ with $\zeta
_s(X_n,X) \leq\eta+ \varepsilon$ for all $n \geq\ell$. Using
(\ref{eqtriangle}),~(\ref{eqqx1}) and
splitting $\{n_0 \leq I_r^{(n)} \leq n-1\}$ into $\{n_0 \leq I_r^{(n)}
\leq\ell\}$ and $\{\ell< I_r^{(n)} \leq n-1\}$, we obtain
\begin{eqnarray*}
\zeta_s(X_n,X) &\leq&\frac{\bar\eta}{1- p_n} \mathbf{E}
\Biggl[\sum_{r=1}^K {\mathbf
1}_{ \{ n_0 \leq I^{(n)}_{r} \leq\ell
\}} \bigl\|A_r^{(n)}\bigr\| _{\mathrm{op}}^s
\Biggr]
\\
&&{}  + \frac{\eta+ \varepsilon}{1-p_n} \mathbf{E} \Biggl[\sum_{r=1}^K
\bigl\| A_r^{(n)}\bigr\| _{\mathrm{op}}^s \Biggr] +
o(1),
\end{eqnarray*}
which, by (C1), finally implies
\[
\eta\leq\mathbf{E} \Biggl[\sum_{r=1}^K
\|A_r\|_{\mathrm{op}}^s \Biggr] (\eta+ \varepsilon).
\]
Since $\varepsilon>0$ is arbitrary and by condition~(C2), we
obtain $\eta= 0$.
\end{pf}

%re20 #&#
%
\begin{rem}
As pointed out in \cite{EiRu07} for a related convergence result, the
statements of Lemma~\ref{lemuniqueness} and Proposition~\ref{cont}
remain true if condition~(C1) is weakened by replacing
\[
\sum_{r=1}^K \bigl\llVert
A_r^{(n)} - A_r \bigr\rrVert_s
\rightarrow0
\]
by
\[
\sum_{r=1}^K \bigl\llVert
\bigl(A_r^{(n)} - A_r \bigr) f \bigr\rrVert
_s \rightarrow0, \qquad\bigl\|A_r^{(n)}
\bigr\|_s \to\|A_r\|_s
\]
for all $f \in\mathcal{C}[0,1]$ and uniform boundedness of $\|
A_r^{(n)}\|_s$
for all $n \geq0$ and all $r = 1, \ldots, K$. This follows from the
given independence structure and the dominated convergence theorem.
\end{rem}

To be able to apply the results of the previous section
to deduce weak convergence from convergence in $\zeta_s$ for the
special cases $\mathcal{C}[0,1]$ and $\mathcal{D}[0,1]$,
rates of convergence for $\zeta_s$ are required. We impose a further
assumption on the convergence rate of the coefficients to establish a
rate of convergence for the process that strengthens condition~(C2). We use the Bachmann--Landau big-$O$ notation for
sequences of numbers.
\begin{longlist}[(C4)]
\item[(C4)] The sequence $(\gamma(n))_{n\ge n_0}$ from condition~(C1) satisfies $\gamma(n)=\break O(R(n))$ as $n\to\infty$
for some positive sequence $R(n) \downarrow0$ such that
\[
L^{*} = \limsup_{n \rightarrow\infty} \mathbf{E} \Biggl[\sum
_{r=1}^K \bigl\|A_r^{(n)}
\bigr\|_{\mathrm{op}}^s \frac
{R(I_r^{(n)})}{R(n)} \Biggr] < 1.
\]
\end{longlist}

%
%co21 #&#
%
\begin{cor} \label{cont2}
Let $(X_n)_{n\ge0}$ satisfy recurrence~(\ref{rec}) with conditions
\textup{(C1)}, \textup{(C3)} and \textup{(C4)}. Then for the fixed-point $\eta
={\mathcal L}(X)$ of $ T$ in~(\ref{deftm}) we have, as $n\to\infty$,
\[
\zeta_s(X_n, X) =O \bigl(R(n) \bigr).
\]
\end{cor}

\begin{pf}
We consider the quantities introduced in the proof of Proposition~\ref{cont} again. By condition~(C4), we have $\zeta_s(Q_n, X) \leq C
R(n)$ for some $C > 0$ and all $n$.
Furthermore, we can choose $\gamma> 0$ and $n_1 > 0$ such that
\[
\mathbf{E} \Biggl[\sum_{r=1}^K
\bigl\|A_r^{(n)}\bigr\|_{\mathrm{op}}^s
\frac
{R(I_r^{(n)})}{R(n)} \Biggr] \leq1 - \gamma, \qquad p_n \leq
\frac
{\gamma}{2}
\]
for $n \geq n_1$. Obviously, for any $n_2 \geq n_1$, we can choose $K
\geq2C/\gamma$ such that $d(n):= \zeta_s(X_n,X) \leq K R(n)$ for all
$n < n_2$.
Using~(\ref{eqqx1}), this implies
\begin{eqnarray*}
d(n_2) &\leq&p_{n_2} d(n_2)+ \mathbf{E}
\Biggl[\sum_{r=1}^K {\mathbf
1}_{ \{ I^{(n_2)}_{r} \leq n_2-1 \}} \bigl\|A_r^{(n_2)}\bigr\| _{\mathrm{op}}^s
d \bigl(I_r^{(n_2)} \bigr) \Biggr] + C R(n_2)
\end{eqnarray*}
hence
\begin{eqnarray*}
d(n_2) &\leq& \frac{1}{1-p_{n_2}} \Biggl( \mathbf{E} \Biggl[\sum
_{r=1}^K \bigl\| A_r^{(n_2)}
\bigr\|_{\mathrm{op}}^s K R \bigl(I_r^{(n_2)}
\bigr) \Biggr] + C R(n_2) \Biggr)
\\
& = & \frac{1}{1-p_{n_2}} \Biggl(K R(n_2) \mathbf{E} \Biggl[\sum
_{r=1}^K \bigl\| A_r^{(n_2)}
\bigr\|_{\mathrm{op}}^s \frac
{R(I_r^{(n_2)})}{R(n_2)} \Biggr] + C
R(n_2) \Biggr)
\\
& \leq& \frac{1}{1-p_{n_2}} \bigl((1- \gamma)K + C \bigr) R(n_2) \leq
K R(n_2).
\end{eqnarray*}
Inductively, $d(n) \leq K R(n)$ for all $n$.
\end{pf}

We now consider the special cases $\mathcal{C}[0,1]$ and $\mathcal{D}[0,1]$.
Related to Corollary~\ref{Cor4}, we consider the following additional
assumption, where the notation ${\mathcal C}_r[0,1]$ defined in (\ref
{defcrn}) is used.
\begin{longlist}[(C5)]
\item[(C5)] \textit{Case} $(\mathcal{C}[0,1], \|\cdot\|_\infty)$: we
have $X_n = Y_n +
h_n$ for all $n\ge0$, where $\|h_n -h \|_\infty\to0$ with $h_n,h\in
\mathcal{C}[0,1]$, and there exists a positive sequence $(r_n)_{n\ge
0}$ such that
\[
\mathbf{P} \bigl(Y_n \notin\mathcal{C}_{r_n}[0,1] \bigr)
\to0.
\]

\textit{Case} $(\mathcal{D}[0,1], d_{\mathrm{sk}})$: we have $X_n = Y_n + h_n$ for all
$n\ge0$, where
$\|h_n -h \|_\infty\to0$ with $h_n \in\mathcal{D}[0,1], h\in
\mathcal{C}[0,1]$, and there
exists a positive sequence $(r_n)_{n\ge0}$ such that
\[
\mathbf{P} \bigl(Y_n \notin\mathcal{D}_{r_n}[0,1] \bigr)
\to0.
\]
\end{longlist}

We now state the main theorem of this section. It follows immediately
from Proposition~\ref{propfdd}, Corollary~\ref{Cor4}, Proposition
\ref{cont} and Corollary~\ref{cont2}.

%th22 #&#
%
\begin{teo} \label{teomain}
Let $(X_n)_{n\ge0}$ be a sequence of random variables in $(\mathcal
{C}[0,1],\break \|
\cdot\|_\infty)$ or $(\mathcal{D}[0,1], d_{\mathrm{sk}})$ satisfying
recurrence~(\ref{rec})
with conditions \textup{(C1)}, \textup{(C2)}, \textup{(C3)} being satisfied. Then, for
${\mathcal L}(X)=\eta$, we have for all $t \in[0,1]$
%
%e53 #&#
%
\begin{equation}
\label{folg1} X_n(t) \stackrel{d} {\longrightarrow} X(t), \qquad
\mathbf{E} \bigl[\bigl|X_n(t)\bigr|^s \bigr] \to\mathbf{E}
\bigl[\bigl|X(t)\bigr|^s \bigr].
\end{equation}
If $Z$ is distributed on $[0,1]$ and independent of $(X_n)$ and $X$ then
%
%e54 #&#
%
\begin{equation}
\label{folg2} X_n(Z) \stackrel{d} {\longrightarrow} X(Z), \qquad
\mathbf{E} \bigl[\bigl|X_n(Z)\bigr|^s \bigr] \to\mathbf{E}
\bigl[\bigl|X(Z)\bigr|^s \bigr].
\end{equation}
If moreover conditions \textup{(C4)} and \textup{(C5)} are satisfied, where
$R(n)$ in \textup{(C4)} and $r_n$ in~\textup{(C5)} can be chosen with
%
%e55 #&#
%
\begin{equation}
\label{condrate} R(n)=o \biggl(\frac{1}{\log^{m}(1/r_n)} \biggr), \qquad n\to\infty,
\end{equation}
then we have convergence in distribution:
\[
X_n \stackrel{d} {\longrightarrow} X.
\]
%
%For $(X_n)_{n \geq0}, X$ with values in $(\Do, d_{\mathrm{sk}})$,
%\eqref{folg1} and the moment convergence in \eqref{folg2} follow from {
%\bf C1}, (C2), (C3).
%If moreover, $X$ has continuous sample paths and (C4) and (C5)
%are satisfied with the relation $\eqref{condrate}$ then the
%distributional convergence in \eqref{folg2} and the functional limit
%law $X_n \rightarrow X$ in the cadlag space are valid.
%Furthermore, if $$\E[\|A_r^{(n)} - A_r\|^s] = O(n^{fd})$$ for all $r$
%and $$\E[\|b^{(n)} - b\|^s] = O(..)$$ then $$\zeta_s(Y_n, Y) = O(..)$$
\end{teo}

Finally, we give sufficient criteria to verify condition~(C3) for
the cases $\mathcal{C}[0,1]$ and $\mathcal{D}[0,1]$. First, consider
the general case where $\mathcal{L}
(Y) = \nu$ is a probability distribution on a separable Banach space
$(B, \| \cdot\|)$ with $\mathbf{E} [\|Y\|^s ] < \infty
$. If $B$ is a Hilbert
space, it is easy to see (and already indicated in \cite{Zolo76} for
$m=2$) that for a probability measure ${\mathcal L}(X)=\mu$ on $B$ to
be in ${\mathcal M}_s(\nu)$ the defining properties~(\ref{eqmom1}) and
(\ref{eqmom2}) are equivalent to $\mathbf{E} [\|X\|^s ]
< \infty$ and
\[
\mathbf{E} \bigl[\varphi_1(X)\cdots\varphi_k(X) \bigr] =
\mathbf{E} \bigl[\varphi_1(Y)\cdots\varphi_k(Y) \bigr]
\]
for all $0 < k \leq m$ and continuous linear forms $\varphi_1, \ldots,
\varphi_n$ on $B$. A generalization of this equivalence to Banach
spaces does not hold in general, a counterexample is constructed in
Janson and Kaijser \cite{jankai12}. However, with deeper arguments from
functional analysis, Janson and Kaijser \cite{jankai12} proved that
this equivalence does hold for separable Banach spaces having the
approximation property, such as $\mathcal{C}[0,1]$. The case $\mathcal
{D}[0,1]$ is also treated
in \cite{jankai12}. Combining~(\ref{eqmom1}),~(\ref{eqmom2}) and
Theorems 1.3~and~16.13 in \cite{jankai12} implies the following lemma.

%Let ${\mathcal L}(Y)=\nu$ be a probability distribution on $\Co$ with $
%\E{\|Y\|_\infty^s}<\infty$. Then for a probability measure ${\mathcal
%L}(X)=\mu$ on $\Co$ to be in ${\mathcal M}_s(\nu)$ we have the
%abstract defining properties in~(\ref{eqmom1}) and~(\ref{eqmom2}).
%%Note that the cases $0\le s\le3$ are of interest in our main result,
%Theorem~\ref{teomain}, and that $\mu\in{\mathcal M}_s(\nu)$ implies
%$\zeta_s(\mu,\nu)<\infty$.

%le23 #&#
%
\begin{lem} \label{matchmoments}
Let ${\mathcal L}(Y)={\mathcal L}((Y_t)_{t\in[0,1]})=\nu$ and
${\mathcal L}(X)={\mathcal L}((X_t)_{t\in[0,1]})=\mu$ be probability
measures on $\mathcal{C}[0,1]$. For $0< s \le1$ we have $\mu\in
{\mathcal M}_s(\nu
)$ if
%
%e56 #&#
%
\begin{equation}
\label{zeil1} \mathbf{E} \bigl[\|X\|_\infty^s \bigr],
\mathbf{E} \bigl[\|Y\| _\infty^s \bigr]<\infty.
\end{equation}
For $1< s \le2$ we obtain $\mu\in{\mathcal M}_s(\nu)$ if we have
condition~(\ref{zeil1}) and
%
%e57 #&#
%
\begin{equation}
\label{zeil2} \mathbf{E} [X_t ]=\mathbf{E} [Y_t ]\qquad\mbox{for all } 0\le t\le1.
\end{equation}
For $2< s \le3$ we obtain $\mu\in{\mathcal M}_s(\nu)$ if we have
conditions~(\ref{zeil1}),~(\ref{zeil2}) and
%
%e58 #&#
%
\begin{equation}
\label{zeil3} \operatorname{Cov}(X_t,X_u)=
\operatorname{Cov}(Y_t,Y_u)\qquad\mbox{for all } 0\le t,u\le1.
\end{equation}
The assertions remain true if $\mathcal{C}[0,1]$ is replaced by
$\mathcal{D}[0,1]$.
\end{lem}

%re24 #&#
%
\begin{rem}
Interpreting $\mathbf{E} [X ]$ as a Bochner integral in
the continuous case,
condition~(\ref{zeil2}) is equivalent to $\mathbf{E} [X
] = \mathbf{E} [Y ]$.
This is due to the fact that $\mathbf{E} [X ]$ is a
continuous function with $\mathbf{E} [X ] (t) = \mathbf
{E} [X(t) ]$ and $\varphi(\mathbf{E} [X ])
= \mathbf{E} [\varphi(X) ]$ for all
continuous linear forms $\varphi$ on $\mathcal{C}[0,1]$. Also the
higher moments can
be interpreted similarly as expectations of corresponding tensor
products; see \cite{DrJaNe08} or, for an elaborate account \cite{jankai12}.
\end{rem}

%re25 #&#
%
\begin{rem}
Note that condition~(\ref{zeil3}) typically cannot be achieved for a
sequence $(X_n)_{n\ge0}$ that arises as in~(\ref{rec2}) by an affine
scaling from a sequence $(Y_n)_{n\ge0}$ as in~(\ref{rec1}). This
fundamental problem for developing a functional contraction method on
the basis of the Zolotarev metrics $\zeta_s$ with $2<s\le3$ was
already mentioned in \cite{DrJaNe08}, Remark 6.2.
We describe a way to circumvent this problem in our application to
Donsker's invariance principle by a perturbation argument; see
Section~\ref{secapp2}.
\end{rem}

%s4 #&#
\section{Applications}
As applications, we first give as a toy example a short proof of
Donsker's invariance principle in Section~\ref{secapp2}. In
Section~\ref{secpm}, we discuss further examples from the
probabilistic analysis of algorithms on partial match queries which
requires the full generality of our abstract setting. This allows to
settle various long standing open questions about asymptotics of the
complexity of such queries.

%s4.1 #&#
\subsection{Donsker's invariance principle} \label{secapp2}
Let $(V_n)_{n\in\mathbb N}$ be a sequence of independent, identically
distributed real valued random variables with $\mathbf{E} [V_1
] = 0$, $\operatorname{Var} (V_1 )
= 1$ (for simplicity) and $\mathbf{E} [|V_1|^{2+\varepsilon}
] < \infty$ for some
$\varepsilon> 0$. We consider the properly scaled and linearized random
walk $S^n=(S^n_t)_{t\in[0,1]}$, $n\ge1$, defined by
\[
S_t^n = \frac{1}{\sqrt{n}} \Biggl( \sum
_{k = 1}^{\lfloor nt\rfloor
}V_k + \bigl(nt - \lfloor nt
\rfloor \bigr) V_{\lfloor nt\rfloor+1} \Biggr), \qquad t\in[0,1].
\]
With $W=(W_t)_{t\in[0,1]}$, a standard Brownian motion Donsker's
function limit law states the following.

%th26 #&#
%
\begin{teo}[(Donsker \cite{do51})]\label{th1}
We have $S^n \stackrel{d}{\longrightarrow} W$
as $n \to\infty$ in $(\mathcal{C}[0,1],\break \| \cdot\|_\infty)$.
\end{teo}

%s4.1.1 #&#
\subsubsection{A contraction proof} \label{donskerproof}
In this section, we apply the general methodology of Sections~\ref{zolosec} and~\ref{nnnCM} to give a short proof of Theorem~\ref{th1}.
For a recursive decomposition of $S^n$ and $W$, we
define operators for $\beta> 1$,
\begin{eqnarray*}
\varphi_{\beta}\dvtx \mathcal{C}[0,1]&\to&\mathcal{C}[0,1], \qquad
\varphi_{\beta} (f) (t) = {\mathbf1}_{\{ t \leq1 / \beta\}} f(\beta t) + {
\mathbf1}_{\{ t > 1 / \beta\}
}f(1),
\\
\psi_{\beta}\dvtx \mathcal{C}[0,1]&\to&\mathcal{C}[0,1], \qquad\psi
_{\beta}(f) (t) ={\mathbf1}_{\{ t
\le1 / \beta\}}f(0)+{\mathbf1}_{\{ t > 1 / \beta\}}
f \biggl(\frac
{\beta t
-1}{\beta-1} \biggr).
\end{eqnarray*}
Note that both $\varphi_{\beta} $ and $\psi_{\beta}$ are linear,
continuous and $\|\varphi_{\beta} (f)\|_\infty= \|\psi_{\beta}(f)\|
_\infty= \|f\|_\infty$ for all $f \in\mathcal{C}[0,1]$, hence we
have $\|\varphi
_\beta\|_\mathrm{op}=\|\psi_\beta\|_\mathrm{op}=1$.
By construction, we have
%
%e59 #&#
%
\begin{equation}
\label{eq3b} \qquad S^n \stackrel{d} {=} \sqrt{\frac{\lceil n /2\rceil}{n}}
\varphi_{{n}/{\lceil n/2\rceil}} \bigl(S^{\lceil n/2 \rceil} \bigr) + \sqrt{ \frac
{\lfloor n /2\rfloor}{n}}
\psi_{{n}/{\lceil n/2\rceil}} \bigl(\widehat{S} {}^{\lfloor n/2 \rfloor} \bigr), \qquad n\ge2,
\end{equation}
where $(S^1,\ldots,S^n)$ and $(\widehat{S}{}^1,\ldots,\widehat{S}{}^n)$ are
independent and $S^j$ and $\widehat{S}{}^j$ are identically
distributed for all $j\ge1$. Therefore, $(S^n)_{n\ge1}$ satisfies
recurrence~(\ref{rec}) choosing
\begin{eqnarray*}
K&=&2, \qquad I_1^{(n)} = \lceil n/2 \rceil, \qquad
I_2^{(n)} = \lfloor n/2 \rfloor, \qquad n_0=2,
\\
A_1^{(n)} &=& \sqrt{\frac{\lceil n /2\rceil}{n}}
\varphi_{{n}/{\lceil n/2\rceil}}, \qquad A_2^{(n)} = \sqrt{
\frac{\lfloor n
/2\rfloor}{n}} \psi_{{n}/{\lceil n/2\rceil}}, \qquad b^{(n)} = 0.
\end{eqnarray*}
In the following, let $\widehat{W}= (\widehat{W}_t)_{t\in[0,1]}$ be a
standard Brownian motion, independent of $W$. Properties of Brownian
motion imply
%
%e60 #&#
%
\begin{equation}
\label{eq4} W \stackrel{d} {=} \sqrt{\frac{1}{\beta}}
\varphi_{\beta}(W)+ \sqrt{\frac
{\beta-1}{\beta}} \psi_{\beta}(
\widehat{W})
\end{equation}
for any $\beta> 1$. Hence, the Wiener measure ${\mathcal L}(W)$ is a
fixed point of the operator $T$ in~(\ref{deftm}) with
%
%e61 #&#
%
\begin{equation}
\label{eqfixdo} K=2, \qquad A_1 = \sqrt{\frac{1}{\beta}}
\varphi_{\beta}, \qquad A_2 = \sqrt{\frac{\beta-1}{\beta}}
\psi_{\beta}, \qquad b = 0.
\end{equation}
For $\beta= 2$, the coefficients in~(\ref{eq3b}) converge to the ones
in~(\ref{eq4}), that is, as $n\to\infty$,
\[
\sqrt{\frac{\lceil n /2\rceil}{n}}\to\frac{1}{\sqrt{2}}, \qquad\sqrt{ \frac{\lfloor n /2\rfloor}{n}}
\to\frac{1}{\sqrt{2}},
\]
but the coefficients $A_1^{(n)}, A_2^{(n)}$ only converge to $A_1, A_2$
in the operator norm for $n$ even. Nevertheless, from the point of view
of the contraction method, this suggests weak convergence of $S^n$ to $W$.

Note that the operator $T$ associated with the fixed-point equation
(\ref{eq4}), that is, with the coefficients in~(\ref{eqfixdo}),
satisfies condition~(C2) only with $s>2$. In view of condition~(C3) and Lemma~\ref{matchmoments}, we need to match the mean and
covariance structure.
We have $\mathbf{E} [S_t^n ] = 0$ for all $0\le t\le1$
and a direct computation yields
%
%e62 #&#
%
\begin{eqnarray}
\label{eq2} \operatorname{Cov} \bigl(S_s^n,S_t^n
\bigr) = \cases{ s, &\quad for $\lfloor ns\rfloor< \lfloor nt\rfloor$,
\vspace*{3pt}\cr
\displaystyle\frac{1}{n} \bigl(\lfloor ns\rfloor+ \bigl(ns - \lfloor ns\rfloor
\bigr) \bigl(nt -\lfloor nt\rfloor \bigr)\bigr), &\quad for $\lfloor ns\rfloor= \lfloor nt
\rfloor$.}\hspace*{-30pt}
\end{eqnarray}
Hence, we do not have finite $\zeta_{2 + \varepsilon}$-distance
between $S^n$ and $W$ since they do not share their covariance
functions. To surmount this problem, we consider a linearized version
of the Brownian motion $W$.
For fixed $n \in\mathbb N$, we divide the unit interval into pieces of length
$1/n$ and interpolate $W$ linearly between the points $0,1/n,2/n,\ldots,(n-1)/n,1$. The interpolated process $W^n=(W_t^n)_{t\in[0,1]}$ is
given by
\[
W^n_t:= W_{{\lfloor nt\rfloor}/{n}} + \bigl( nt - \lfloor nt
\rfloor \bigr) (W_{{(\lfloor nt\rfloor+1)}/{n}}-W_{{\lfloor nt\rfloor}/{n}} ), \qquad t\in[0,1].
\]
We have $\mathbf{E} [W^n_t ] = 0$ and $W^n$ and $S^n$
have the same covariance
function~(\ref{eq2}) for all $n \in\mathbb N$. Furthermore, $W^n$ has the
same distributional recursive decomposition~(\ref{eq3b}) as $S^n$.

Note that the linearized Brownian motion does not differ much from the
original one:

%le27 #&#
%
\begin{lem} \label{LemmaB}
We have $\|W^n -W\|_\infty\to0$ as $n\to\infty$ almost surely.
\end{lem}

\begin{pf}
This directly follows from the uniform continuity of $W$. For
$\varepsilon> 0$, there exists a random $\delta> 0$ such that
$|W(t)-W(s)| < \varepsilon$ for any
$s,t \in[0,1]$ with $|t-s| < \delta$. The triangle inequality implies
$\| W^n - W \|_\infty< 2 \varepsilon$ for any $n > 1/\delta$.
\end{pf}

In view of Corollary~\ref{Cor1}, it suffices to prove that $S^n$ and
$W^n$ are close with respect to $\zeta_{2 + \varepsilon}$. The proof of
this runs along the same lines as the one for Proposition~\ref{cont},
respectively, Corollary~\ref{cont2}; in fact, it is much shorter due to
the simple form of the recurrence:

%pr28 #&#
%
\begin{prop} \label{prop2}
For any $\delta< \varepsilon/2$ we have $\zeta_{2+\varepsilon}(S^n,
W^n) = O(n^{-\delta})$ as \mbox{$n \rightarrow\infty$}.
\end{prop}

\begin{pf} We have
\begin{eqnarray*}
\zeta_{2+\varepsilon} \bigl(S^n, W^n \bigr) &= &
\zeta_{2+\varepsilon} \biggl( \sqrt{\frac{\lceil n /2\rceil}{n}} \varphi _{{n}/{\lceil n/2\rceil}}
\bigl(S^{\lceil n/2 \rceil
} \bigr) + \sqrt{\frac{\lfloor n /2\rfloor}{n}} \psi_{{n}/{\lceil
n/2\rceil
}}
\bigl(\overline{S} {}^{\lfloor n/2 \rfloor} \bigr),
\\
&& \hphantom{\zeta_{2+\varepsilon} \biggl(} \sqrt{\frac
{\lceil n
/2\rceil}{n}}
\varphi_{{n}/{\lceil n/2\rceil}} \bigl(W^{\lceil n/2
\rceil} \bigr) + \sqrt{\frac{\lfloor n /2\rfloor}{n}}
\psi_{{n}/{\lceil n/2\rceil}} \bigl(\overline{W} {}^{\lfloor n/2 \rfloor
} \bigr) \biggr)
\\
& \leq& \biggl( \frac{\lceil n /2\rceil}{n} \biggr)^{1+ \varepsilon/2} \zeta_{2+\varepsilon}
\bigl(S^{\lceil n/2 \rceil}, W^{\lceil n/2
\rceil} \bigr)
\\
&&{}+ \biggl( \frac{\lfloor n /2\rfloor}{n} \biggr)^{1+ \varepsilon/2} \zeta_{2+\varepsilon}
\bigl(S^{\lfloor n/2 \rfloor}, W^{\lfloor n/2
\rfloor} \bigr).
\end{eqnarray*}
We abbreviate
\[
d_n:= \zeta_{2+\varepsilon} \bigl(S^n, W^n
\bigr), \qquad a_n:= \biggl( \frac
{\lceil n /2\rceil}{n} \biggr)^{1+ \varepsilon/2},
\qquad b_n:= \biggl( \frac{\lfloor n /2 \rfloor}{n} \biggr)^{1+
\varepsilon/2}
\]
and note that we have $a_n + b_n \leq2^{-\varepsilon/2} +C'/ n$ for
some constant $C' > 0$ and all $n \in\mathbb N$. For arbitrary $\delta<
\varepsilon/2$, we prove the assertion by induction: fix $\delta<
\delta' < \varepsilon/2$ and choose $m_0 \in\mathbb N$ such that
$ \lfloor n/2 \rfloor^{-\delta} \leq(n / 2) ^{-\delta}
2^{\varepsilon
/2 - \delta'}$ and $1+ 2^{\varepsilon/2} C'/n \leq2^{\delta'-\delta}$
for all $n \geq m_0$. Furthermore, let $C > 0$ be large enough such
that $d_n \leq C n^{-\delta}$ for all $1\le n \leq m_0$. Then, for $n >
m_0$, assuming the claim to be verified for all smaller indices,
\begin{eqnarray*}
d_n &\leq&a_n d_{\lceil n/2 \rceil} + b_n
d_{\lfloor n/2 \rfloor}
\\
& \leq& C \bigl( a_n (n/2)^{-\delta} + b_n
(n/2)^{-\delta
}2^{\varepsilon/2 - \delta'} \bigr)
\\
& \leq& C n^{-\delta} 2^{\delta} 2^{\varepsilon/2 - \delta'} (a_n +
b_n)
\\
& \leq& C n^{-\delta}.
\end{eqnarray*}
The assertion follows.
\end{pf}

Now Donsker's theorem (Theorem~\ref{th1}) follows from Proposition
\ref{prop2},\break Lemma~\ref{LemmaB} and Corollary~\ref{Cor1}.

Note that our approach requires the assumption $\mathbf{E}
[|V_1|^{2 + \varepsilon} ] < \infty$ for some $\varepsilon
>0$, which in Donsker's
theorem can be weakened to $\mathbf{E} [V_1^2 ] < \infty
$.

By Theorem~\ref{teosup}, we directly obtain convergence of moments of
the supremum.

%co29 #&#
%
\begin{cor} \label{donmom}
Suppose $\mathbf{E} [|V_1|^{2+ \alpha} ] < \infty$ with
$0 < \alpha\leq1$.
Then $\|S^{n}\|_\infty^{2+\alpha}$ is uniformly integrable. Thus,
$\mathbf{E} [\| S^{n}\|_\infty^{\kappa} ]$ converges to
$\mathbf{E} [\|W\|_{\infty}^\kappa ]$ for any $0 <
\kappa\leq2 + \alpha$.
\end{cor}

%re30 #&#
%
\begin{rem}
Based on the recursion~(\ref{eq3b}), it is easy to show that\break $\mathbf
{E} [\| S^n\| _\infty^k ]$ is bounded uniformly in $n$
for integer valued $k \geq3$
if the increment $V_1$ has finite
absolute moment of order $k$. In this case, we have $\mathbf{E}
[\| S^n \| _\infty^\kappa ] \rightarrow\mathbf{E} [\|
W \|_\infty^{\kappa} ] $
for any real $0 < \kappa< k$.
\end{rem}

%s4.1.2 #&#
\subsubsection{Characterizing the Wiener measure by a fixed-point property} \label{subsecchar}
We reconsider the map $T$ corresponding to the fixed-point equation
(\ref{eq4}) for the case $\beta=2$:
%
%e63 #&#
%
\begin{equation}\label{tspec}
 T \dvtx \mathcal{M} \bigl(\mathcal{C}[0,1] \bigr) \to\mathcal{M} \bigl(
\mathcal{C}[0,1] \bigr), \qquad
T(\mu) = \mathcal{L} \biggl( \frac{1}{\sqrt{2}} \varphi_{2}(Z)+
\frac
{1}{\sqrt{2}} \psi_{2}({\overline{Z}}) \biggr),\hspace*{-35pt}
\end{equation}
where $Z$, $\overline{Z}$ are independent with distribution ${\mathcal
L}(Z)={\mathcal L}(\overline{Z})=\mu$. Our discussion above implies
that the Wiener measure $\mathcal{L}(W)$ is the unique fixed point of
$T$ restricted to $\mathcal{M}_{2 + \varepsilon}(\mathcal{L}(W))$
for any $\varepsilon
>0$. Note that $\mathcal{M}_{2 + \varepsilon}(\mathcal{L}(W))$ is
the space of the
distributions of all continuous stochastic processes $V=(V_t)_{t\in
[0,1]}$ with $\mathbf{E} [\|V\|_\infty^{2+\varepsilon}
]<\infty$, $\mathbf{E} [V_t ]=0$ and
$\operatorname{Cov}(V_t,V_u)=t \wedge u$ for all $0\le t,u\le1$. Note that
one easily verifies that $T(\mathcal{M}_{2 + \varepsilon}(\mathcal
{L}(W))) \subset\mathcal{M}
_{2 + \varepsilon}(\mathcal{L}(W))$ and the last part of the proof of
Lemma~\ref{lemuniqueness} implies that $T$ restricted to $\mathcal
{M}_{2 +
\varepsilon}(\mathcal{L}(W))$ is Lipschitz-continuous with Lipschitz constant
at most $L=2^{-\varepsilon/2}<1$, hence $\mathcal{L}(W)$ is the
unique fixed
point of $T$ in $\mathcal{M}_{2 + \varepsilon}(\mathcal{L}(W))$.

We now show that a more general statement is true, the Wiener measure
is also, up to multiplicative scaling, the unique fixed point of $T$ in
the larger space of probability measures ${\mathcal L}(V)\in\mathcal
{M}(\mathcal{C}[0,1])$
with $V_0=0$. For a related statement, see also Aldous \cite{al94}, page 528.
The subsequent proof is based on the fact that the centered normal
distributions are the only solutions of the fixed-point equation
%
%e64 #&#
%
\begin{equation}
\label{fixnormal} X \stackrel{d} {=} \frac{X + \overline{X}}{\sqrt{2}},
\end{equation}
where $X, \overline{X}$ are independent, identically distributed
real-valued random variables; see
Theorem 7.2.1 in \cite{Lukacs75}.

%th31 #&#
%
\begin{teo} \label{bmchar}
Let $X = (X_t)_{t \in[0,1]}$ be a continuous process with $X_0 =0$.
Then ${\mathcal L}(X)$ is a fixed-point of~(\ref{tspec}) if and only
if either $X = \mathbf{0}$ a.s. or there exists a constant $\sigma>
0$, such that $(\sigma^{-1} X_t)_{t \in[0,1]}$ is a standard Brownian motion.
\end{teo}

\begin{pf} Let ${\mathcal L}(X)$ be a fixed point of~(\ref{tspec}) and
$\overline{X}= (\overline{X}_t)_{t \in[0,1]}$
be independent of $X$ with the same distribution. The fixed point
property implies
\[
X_1 \stackrel{d} {=} \frac{X_1 + \overline{X}_1}{\sqrt{2}},
\]
hence $\mathcal{L}(X_1) = \mathcal{N} (0, \sigma^2)$ for some
$\sigma^2 \geq
0$, where $\mathcal{N}(0, \sigma^2)$ denotes the centered normal
distribution with variance $\sigma^2$. This implies
\[
X_{1/2} \stackrel{d} {=} \frac{X_1}{\sqrt{2}},
\]
hence $\mathcal{L}(X_{1/2})= \mathcal{N}(0, \sigma^2/2)$. Let
$\mathscr{D}= \{
m2^{-n}\dvtx  m,n \in\mathbb N_0, m \leq2^n\}$ by the set of dyadic
numbers in $[0,1]$.
By induction, we obtain $\mathcal{L}(X_t)= \mathcal{N}(0, \sigma^2
t)$ for all
\mbox{$t \in\mathscr{D}$}. For the distribution of the increments, we first obtain
\[
X_1 - X_{1/2} \stackrel{d} {=} \frac{X_1}{\sqrt{2}},
\]
hence $\mathcal{L}(X_1 - X_{1/2})= \mathcal{N}(0, \sigma^2/2)$. Again
inductively, we obtain
$\mathcal{L}(X_1 - X_t)= \mathcal{N}(0,(1-t)\sigma^2)$ for all $t
\in\mathscr{D}$.
Also by induction, it follows
$\mathcal{L}(X_t - X_s)= \mathcal{N}(0,(t-s)\sigma^2)$ for all $s,t
\in\mathscr{D}$
with $s < t$. Finally, continuity of $X$ implies the same property for
all $s,t \in[0,1]$.
It remains to prove independence of increments. Denoting by $X^{(1)},
X^{(2)}, \ldots$ independent distributional copies of $X$, we obtain
from iterating the fixed-point property
\begin{eqnarray*}
&& (X_t)_{t \in[0,1]} 
\\
&&\qquad \stackrel{d} {=} \Biggl( 2^{-n/2}
\sum_{m=1}^{2^n} {\mathbf1}_{\{(m-1)2^{-n} < t \leq m2^{-n} \}}
X_{2^n t - m +1}^{(m)} + {\mathbf1}_{\{ m 2^{-n} < t \}}
X_1^{(m)} \Biggr)_{t \in[0,1]}
\end{eqnarray*}
for all $n \in\mathbb N$. Hence, for any dyadic points $0 \leq t_1 <
t_2 <
\cdots< t_k \leq1$, choosing $n$ large enough, each
$X_{t_{i+1}}-X_{t_i}$ can be
expressed as a function of a subset of $X^{(1)}, \ldots, X^{
(2^n )}$ these subsets being pairwise disjoint for $i=0, \ldots,
n-1$. Since, $\mathscr{D}$ is
dense in $[0,1]$, this shows that $X$ has independent increments. For
$\sigma= 0$, we have $X = \mathbf{0}$ a.s., otherwise
$\sigma^{-1} X$ is a standard Brownian motion.\looseness=-1

The converse direction of the theorem is trivial.
\end{pf}

%re32 #&#
%
\begin{rem}
Note that we cannot cancel the assumption on continuity of $X$ without
replacement, for example, the process
\[
Y_t = \cases{ W_t, &\quad$t \notin\mathscr{D}$,
\cr
0, &\quad$t \in\mathscr{D}$}
\]
also solves~(\ref{eq4}) and is not a multiple of Brownian motion.
However, it would be sufficient to require c\`{a}dl\`{a}g paths, so
$\mathcal{C}[0,1]
$ could be replaced by $\mathcal{D}[0,1]$ in our statement.
\end{rem}

%re33 #&#
%
\begin{rem}
Our decomposition of Brownian motion in~(\ref{eq4}) is in\break time. However,
equation~(\ref{fixnormal}) suggests to also investigate a
decomposition in space
%
%e65 #&#
%
\begin{equation}
\label{raeuml} ( X_t )_{t \in[0,1]}\stackrel{d} {=} \biggl(
\frac{X_t +
\overline{X}_t}{\sqrt{2}} \biggr)_{t \in[0,1]},
\end{equation}
where $(X_t)_{t \in[0,1]}$ and $(\overline{X}_t)_{t \in[0,1]}$ are
independent and identically distributed. Again, equation~(\ref{raeuml})
induces a map on $\mathcal{M}(\mathcal{C}[0,1])$ that is a
contraction in $\zeta_{2+
\varepsilon}$ on the subspace $\mathcal{M}_{2 + \varepsilon}(
\mathcal{L}(W))$, so the
Wiener measure is the only solution in $\mathcal{M}_{2 + \varepsilon
}( \mathcal{L}
(W))$. In this case, we cannot remove the moment assumption as in
Theorem~\ref{bmchar} since
any centered, continuous Gaussian process solves equation (\ref
{raeuml}). Using~(\ref{fixnormal}), it is not hard
to see that there are no further solutions of~(\ref{raeuml}).
\end{rem}

%s4.2 #&#
\subsection{Partial match queries in quad trees} \label{secpm}\label{apppm}
In this section, we outline recurrences coming up in the probabilistic
analysis of the performance of data structures and discuss in detail
the use and verification of our conditions (C1)--(C5) and
Theorem~\ref{teomain}. In this example, the full generality of our
setup is needed.

For preprocessing and supporting search queries in multidimensional
data various types of search trees are in use, most prominently quad
trees and $k$-d trees. Among various other fundamental search operations
in multivariate data so-called partial match queries are of particular
importance. For a partial match query, one specifies some of the
components of the data and asks to report all data in the given set
that match the specified components and are arbitrary in the remaining
components. We will subsequently not need to introduce these data
structures and the partial match queries since there is a geometric
reformulation that is discussed and used below. For details about the
computer science background and precise definition of the structures
and queries, see \cite{BrNeSu11}.

Consider a sequence $(U_i,V_i)_{i\ge1}$ of independent and identically
distributed random vectors all with the uniform distribution on the
unit square $[0,1]^2$. We iteratively construct a decomposition of
$[0,1]^2$ as follows. The first point $(U_1,V_1)$ \mbox{decomposes} the square
into four rectangles by drawing the two lines through $(U_1,V_1)$ in
$[0,1]^2$ that are perpendicular to its sides. We call these line
segments the horizontal and vertical lines. The second point
$(U_2,V_2)$ almost surely falls into the interior of one of the four
rectangles. We recursively draw the horizontal and vertical lines
through $(U_2,V_2)$ within the rectangle. Hence, we then have a
decomposition of the original square $[0,1]^2$ into
seven rectangles. Now we iterate this process. After $n-1$ steps, we
have $3(n-1)+1$ rectangles and the $n$th point is used to decompose the
rectangle it falls in into four new rectangle by the horizontal and
vertical lines through it; see Figure~\ref{figconst}. We identify this
decomposition of the unit square with all the line segments drawn and
call it the \emph{decomposition after $n$ steps}.

%
%f1 #&#
%
\begin{figure}[t]

\includegraphics{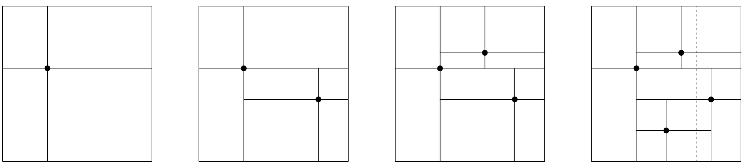}

\caption{The construction of a quad tree at times $n=1, 2, 3, 4$. The
dashed line in the right most square indicates the query line
$x_1=t$.}\label{figconst}
\end{figure}

Now fix $t\in[0,1]$ and
denote the number of horizontal lines in the decomposition after $n$
steps that are cut by the vertical line $x_1=t$ by $C_n(t)$; see
Figures~\ref{figconst}~and~\ref{figct}.
%
%f2 #&#
%
\begin{figure}%[h]

\includegraphics{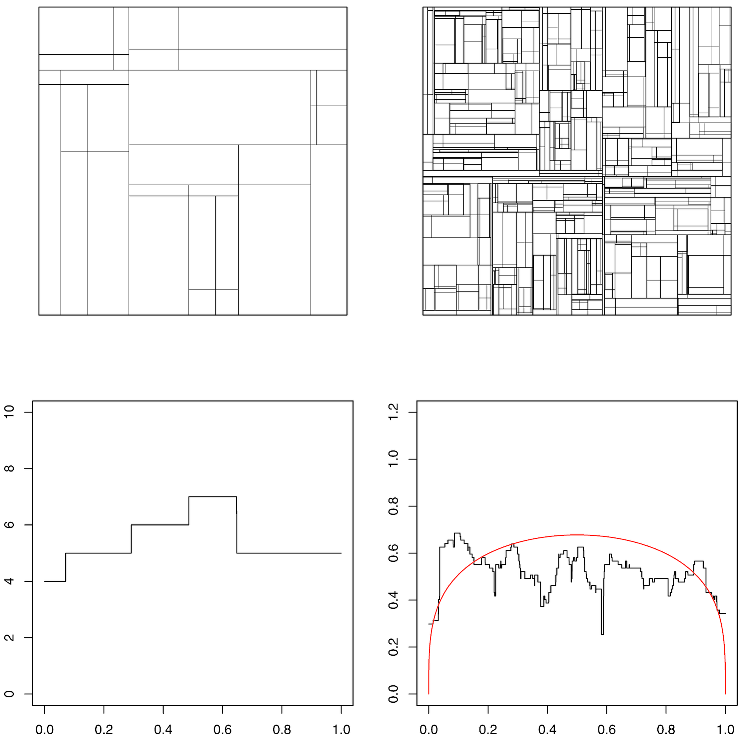}

\caption{Left column: A realization of the decomposition of the unit
square for a quad tree of size $n=10$ and its process $(C_n(t))_{t\in
[0,1]}$. Right column: A realization of the decomposition of the unit
square for a quad tree of size $n=300$ and the process $(X_n(t))_{t\in
[0,1]}$. The smooth curve indicates the function $t \mapsto
(t(1-t))^{\beta/2}$.}
\label{figct}
\end{figure}

In the computer science setting, this is the measure for the complexity
of a partial match query in a random (point) quad tree where the first
component is specified as $t$, the second component is arbitrary and
$n$ data are inserted in the uniform model; see \cite{BrNeSu11}. We
have $C_0(t)=0$ and $C_1(t)=1$ for all $t\in[0,1]$. We consider the
process $(C_n(t))_{t\in[0,1]}$ as a process in $({\mathcal D}[0,1],d_{\mathrm{sk}})$.

For a recursive decomposition of this process, we denote the numbers of
points among the first $n$ points which fall into each
of the\vspace*{1pt} four rectangles generated by the first point $(U_1,V_1)=:(U,V)$
by $I^{(n)}=(I^{(n)}_1,I^{(n)}_2,I^{(n)}_3,I^{(n)}_4)$.
Hence, \mbox{conditionally} on $(U,V)$, the vector $I^{(n)}$ has the
multinomial distribution $M(n-1;UV,U(1-V),(1-U)V,(1-U)(1-V))$,
where a numbering of the four quadrants is used. Moreover,
conditionally on $(U,V)$ and $I^{(n)}$ we have that each point set
within a rectangle is a set of independent and identically distributed
points each with the uniform distribution on the particular rectangle
and that the four
point sets are also independent. Hence, for processes
$(C^{(r)}_j(t))_{t\in[0,1]}$ which are independent and independent of
$(U,V,I^{(n)})$, and $(C^{(r)}_j(t))_{t\in[0,1]}$ distributed as
$(C_j(t))_{t\in[0,1]}$ for $r=1,\ldots,4$ and $j\in\mathbb N_0$ we
obtain the
recurrence
%
%e66 #&#
%
\begin{eqnarray}
\label{Sulzbachrec1} \bigl(C_n(t) \bigr)_{t\in[0,1]} &\stackrel{d}
{=}& \biggl( 1 + \mathbf{1}_{ \{ t<U
\}} \biggl[C^{(1)}_{I_1^{(n)}}
\biggl(\frac{t} U \biggr) + C^{(2)}_{I_2^{(n)}} \biggl(
\frac{t} U \biggr) \biggr]
\nonumber\\[-9pt]\\[-9pt]
&&\hspace*{5pt}{} + \mathbf{1}_{ \{ t \geq U \} } \biggl[C^{(3)}_{I_3^{(n)}}
\biggl(\frac{t-U}{1-U} \biggr) + C^{(4)}_{I_4^{(n)}} \biggl(
\frac{t-U}{1-U} \biggr) \biggr] \biggr)_{t\in[0,1]}.\nonumber\vadjust{\goodbreak}
\end{eqnarray}
The arguments $t/U$ and $(1-t)/(1-U)$ adjust that a vertical line
$x_1=t$ within the whole square $[0,1]^2$, after scaling, corresponds
to the line $x_1=t/U$ in the left rectangles (if $t<U$) and to the line
$x_1=(1-t)/(1-U)$ in the right rectangles (if $t\ge U$). Note that
equation~(\ref{Sulzbachrec1}) has exactly the form~(\ref{rec1}), where
the indicators and rescalings in time in~(\ref{Sulzbachrec1}) give the
random linear maps $A_r(n)$ for $r=1,\ldots,4$, and we have $b(n)=1$.

The first asymptotic analysis of this process was done by Flajolet et
al. \cite{SulzbachFGPR1993}, where the one-dimensional averaged
complexity $C_n(\xi)$ was considered with $\xi$ uniformly distributed
on $[0,1]$ and independent of the sequence $(U_i,V_i)_{i\in\mathbb
N}$. In
\cite{SulzbachFGPR1993}, is shown that, as $n\to\infty$,
\[
\mathbb{E} \bigl[C_n(\xi) \bigr]\sim\kappa n^{\beta}\qquad
\mbox{with } \kappa= \frac{\Gamma(2\beta+2)}{2(\Gamma(\beta+1))^3},  \beta=\frac{\sqrt{17}-3}2,
\]
where $\Gamma$ denotes the gamma function; see also Chern and Hwang
\cite{SulzbachChHw2003} for more refined analysis of this expectation.
Recently, Curien and Joseph \cite{SulzbachCuJo2010} showed
%
%e67 #&#
%
\begin{equation}
\label{SulzbachconstCJ} \qquad\mathbb{E} \bigl[C_n(t) \bigr]\sim\chi \bigl(t(1-t)
\bigr)^{\beta/2} n^{\beta}\qquad\mbox{with } \chi= \frac{\kappa}{B ( ({\beta}/{2}) +1, ({\beta}/{2}) +1 )},
\end{equation}
where $B( \cdot, \cdot)$ denotes the beta function (Euler
integral). The analysis beyond expectations, in particular of variances
and limit laws either for the
process $(C_n(t))_{t\in[0,1]}$ itself or its marginals or the averaged
complexity $C_n(\xi)$ or the worst case complexity $\sup_{t\in[0,1]}
C_n(t)$ remained open.

We now discuss how our general framework from Section~\ref{nnnCM} can
be applied to a proper normalization of $(C_n(t))_{t\in[0,1]}$ and
highlight the use and verification of conditions (C1)--(C5),
which can be shown to hold with the choice $s=2$. The details are
worked out in \cite{BrNeSu11}. The resulting functional limit law
allows to settle the open questions raised in the previous paragraph.

Let us first use the normalization $X_0(t):=0$ and
%
%e68 #&#
%
\begin{equation}
\label{13norm} X_n(t):= \frac{C_n(t)}{\chi n^\beta}, \qquad n\ge1, t
\in[0,1]
\end{equation}
and write $X_n:=(X_n(t))_{t\in[0,1]}$. See Figure~\ref{figct} for a
simulation of $X_n$.
For $X_n$, we obtain the recurrence
\begin{eqnarray*}
X_n &\stackrel{d} {=}& \biggl( \frac{1}{\chi n^\beta} +
\mathbf{1}_{
\{ t<U
\} } \biggl[ \biggl(\frac{I_1^{(n)}}{n} \biggr)^\beta
X^{(1)}_{I_1^{(n)}} \biggl(\frac{t} U \biggr) + \biggl(
\frac
{I_2^{(n)}}{n} \biggr)^\beta X^{(2)}_{I_2^{(n)}}
\biggl(\frac{t} U \biggr) \biggr]
\\[-1pt]
&&\hspace*{5pt}{} + \mathbf{1}_{ \{ t \geq U \} } \biggl[ \biggl(\frac
{I_3^{(n)}}{n}
\biggr)^\beta X^{(3)}_{I_3^{(n)}} \biggl(\frac{t-U}{1-U}
\biggr) + \biggl(\frac{I_4^{(n)}}{n} \biggr)^\beta X^{(4)}_{I_4^{(n)}}
\biggl(\frac{t-U}{1-U} \biggr) \biggr] \biggr)_{t\in[0,1]}
\end{eqnarray*}
with assumptions on independence and identical distributions as in
(\ref{Sulzbachrec1}). This suggests that a
limit process $X=(X(t))_{t \in[0,1]}$ satisfies
%
%e69 #&#
%
\begin{eqnarray}
\label{sulzbachlimit} X &\stackrel{d} {=} & \biggl( \mathbf{1}_{ \{ t<U
\} }
\biggl[(UV)^\beta X^{(1)} \biggl(\frac{t} U \biggr)+
\bigl(U(1-V) \bigr)^\beta X^{(2)} \biggl(\frac{t} U
\biggr) \biggr]\nonumber
\\
&&\hspace*{5pt}{}+ \mathbf{1}_{ \{ t \geq U \} } \biggl[ \bigl((1-U)V \bigr)^\beta
X^{(3)} \biggl(\frac{t-U}{1-U} \biggr)
\\
&&\hspace*{72pt}{} + \bigl((1-U) (1-V)
\bigr)^\beta X^{(4)} \biggl(\frac
{t-U}{1-U} \biggr) \biggr]
\biggr)_{t \in[0,1]},\nonumber
\end{eqnarray}
where $U$ and $V$ are independent $[0,1]$-uniform random variables and\break 
$(X^{(r)}(t))_{t\in[0,1]}$, for $r=1,\ldots, 4$, are independent copies
of the process $X$, also
independent of $(U,V)$. Note that~(\ref{sulzbachlimit}) is a
fixed-point equation of type~(\ref{fix2}).

This heuristic derivation of equation~(\ref{sulzbachlimit}) can be
turned into a rigorous approach as follows. First, note that the
operators $A^{(n)}_1$ and $A_1$ on ${\mathcal D}[0,1]$ are given as
follows: for $f\in{\mathcal D}[0,1]$, the
random functions $A^{(n)}_1(f)$ and $A_1(f)$ are
%
%e70 #&#
%
\begin{equation}
t \mapsto\mathbf{1}_{ \{ t<U \} } \biggl(\frac{I_1^{(n)}}{n}
\biggr)^\beta f \biggl(\frac{t} U \biggr)\quad\mbox{and}\quad t
\mapsto\mathbf{1}_{ \{ t<U \} } (UV)^\beta f \biggl(\frac{t} U
\biggr)
\end{equation}
and direct integration shows that condition~(C2) is satisfied for
the choice $s=2$.

For condition~(C3) first an appropriate process $X=(X(t))_{t \in
[0,1]}$ which solves~(\ref{sulzbachlimit}) has to be constructed. Since
we do not know the completeness of $\zeta_2$ on an appropriate subspace
of ${\mathcal M}_2({\mathcal D}[0,1])$ and also are not able to guess
$X$ as a well-known process (as in the example in Section~\ref{donskerproof}) such a process $X$ has to be constructed individually.
In view of~(\ref{SulzbachconstCJ}), the normalization~(\ref{13norm}),
the choice $s=2$ and Lemma~\ref{matchmoments} we additionally need to
have $\mathbf{E}[\|X\|_\infty^2]<\infty$ and $\mathbf{E}[X(t)]=
(t(1-t))^{\beta/2}$ for $t\in[0,1]$.
In \cite{BrNeSu11}, a sequence of random
continuous functions is constructed from a discrete recurrence
approximating~(\ref{sulzbachlimit}) which converges uniformly. The
construction uses concentration inequalities and tail bounds for the
saturation level of random quad trees. Its limit $X$ is the stochastic
process as needed. Moreover, it can also be shown that it has
continuous paths almost surely.

Our normalization does not imply that $\mathcal{L} ( X_n )
\in
\mathcal{M}_2(\mathcal{L}(X))$ for all $n \geq1$, since the
normalization in
(\ref{13norm}) does violate condition~(\ref{zeil2}). Thus, the
processes $X_n$ cannot be compared with $X$ using the $\zeta_2$
distance. To overcome this technical issue, one can instead consider
the normalization
%
%e71 #&#
%
\begin{equation}
\label{030413} \frac{C_n(t)- \mathbf{E} [C_n(t) ]}{\chi n^\beta}, \qquad t\in[0,1],   n\ge1
\end{equation}
and the shifted limit $(X(t) - (t(1-t))^{\beta/2})_{t\in[0,1]}$. Then
condition~(C3) is satisfied. This also shows the necessity to allow
the perturbation $h_n$ in Corollary~\ref{Cor4} and condition~(C5)
in our general setup.
The centering of the sequence $X_n$ and the solution $X$ only affects
the additive term $b^{(n)}$ and the toll term $b$. In particular,
condition~(C2) remains valid in the centered setting and we have
$\| A^{(n)}_1 - A_1 \|_s \to0$ for any $s>0$. Similarly, $\| A^{(n)}_r
- A_r \|_s \to0$ for $r=2,3,4$. Convergence of the additive term
$b^{(n)}$ is equivalent to uniformity of the expansion in (\ref
{SulzbachconstCJ}). This is shown in \cite{BrNeSu11}. It is also easily
seen that~(\ref{inicond}) holds, hence condition~(C1) is true.

For condition~(C4), an appropriated rate of convergence of the
coefficients in~(\ref{eqrate1}) is needed. Note that such a rate can
only be derived if a rate in the asymptotic expansion of the means in
(\ref{SulzbachconstCJ}) is available. Hence, as a technical step in
\cite{BrNeSu11} a polynomial additive error term of the order $O(n^{\beta-\varepsilon})$ for some $\varepsilon>0$ is shown to hold
valid uniformly in $t\in[0,1]$. This implies that the convergence rates
$\gamma(n)$ in~(\ref{eqrate1}) satisfy
$\gamma(n)=O(n^{-\varepsilon})$ as $n\to\infty$. Hence, for
the sequence $(R(n))_{n\ge1}$ in condition~(C4) we can
choose $R(n)=n^{-\varepsilon'}$ with $0<\varepsilon'\le\varepsilon$
sufficiently small such that we obtain $L^\ast<1$ in (C4).

Finally, note that the jumps of your piecewise constant processes
$X_n$ occur at the random times $U_1,\ldots,U_n$ so that interval
lengths between consecutive jumps may become arbitrarily small.
Condition (C5) allows to cover such instances of processes if the
probability for close jumps can be controlled. In our example, it is
easy to see that the smallest interval between jumps is of length at
least $n^{-3}$ with probability of order $O(1/n)$. Hence,
condition~(C5) is satisfied with the choice $r_n=n^{-3}$ there.
Moreover, the sequences $(r_n)_{n\in\mathbb N}$ and $(R(n))_{n\in
\mathbb N}$ are
chosen such that condition~(\ref{condrate}) is fulfilled.
Hence,
our main result Theorem~\ref{teomain} applies and we first obtain
distributional convergence of the centered normalized sequence in
(\ref{030413}) which also implies
\[
X_n \stackrel{d} {\longrightarrow} X
\]
in $({\mathcal D}[0,1], d_{\mathrm{sk}})$. Here, we may also apply
Theorem~\ref{teosup} to infer convergence of moments of $\|X_n \|$ toward the
moments of $\|X\|$.

The use of some other search trees to support partial match queries
leads to distributional recurrences related to~(\ref{Sulzbachrec1}),
for example, the $2$-d-trees. For the application of our framework in
this case, see \cite{BrNeSu11}.

\section*{Acknowledgements}
We thank Svante Janson for his invaluable support during various stages
of our research on a functional version of the contraction method, in
particular, for bringing the paper \cite{Barbour90} to our attention,
for his help clarifying the conditions in Lemma~\ref{matchmoments} and
many detailed comments on a first version of this paper. We also thank
Alfredas Ra{\v{c}}kauskas for commenting on the counterexamples
discussed in \cite{BeRa84} and three anonymous referees for
constructive comments.

%\begin{appendix}
%\section{}
%\end{appendix}

%\section*{Acknowledgments}

%\begin{supplement}[id=suppA]
%\slink[doi]{10.1214/00-AOPXXXXSUPP} %[doi,text={...}] - jei reikia
%suskaldyti doi
%\sfilename{aopXXXX\_supp.pdf}
%\end{supplement}

% imsref loaded by linak, 2014-03-07 11:58:17
%

\printaddresses


\begin{thebibliography}{42}
% pybtex-1.00. Style name=ims, version=2.7, label_style=nolabel,
%sorting_style=complex, cfg=None, language=None.

%b1 ###
%b1 #&#
\bibitem{al94}
%
\begin{barticle}[mr]
\bauthor{\bsnm{Aldous},~\bfnm{David}\binits{D.}}
(\byear{1994}).
\btitle{Recursive self-similarity for random trees, random
triangulations and {B}rownian excursion}.
\bjournal{Ann. Probab.}
\bvolume{22}
\bpages{527--545}.
\bid{issn={0091-1798}, mr={1288122}}
\end{barticle}
%
\bptok{imsref}%
% NOT OUTPUTED:
% issn = 0091-1798
% url =
%%%%http://links.jstor.org/sici?sici=0091-1798(199404)22:2<527:RSFRTR>2.0.CO;2-A&origin=MSN
% number = 2
% coden = APBYAE
% fjournal = The Annals of Probability
\endbibitem

%b2 ###
%b2 #&#
\bibitem{Barbour90}
%
\begin{barticle}[mr]
\bauthor{\bsnm{Barbour},~\bfnm{A.~D.}\binits{A.~D.}}
(\byear{1990}).
\btitle{Stein's method for diffusion approximations}.
\bjournal{Probab. Theory Related Fields}
\bvolume{84}
\bpages{297--322}.
\bid{doi={10.1007/BF01197887}, issn={0178-8051}, mr={1035659}}
\end{barticle}
%
\bptok{imsref}%
% NOT OUTPUTED:
% issn = 0178-8051
% url = http://dx.doi.org/10.1007/BF01197887
% number = 3
% coden = PTRFEU
% fjournal = Probability Theory and Related Fields
\endbibitem

%b3 ###
%b3 #&#
\bibitem{BaJa09}
%
\begin{barticle}[mr]
\bauthor{\bsnm{Barbour},~\bfnm{A.~D.}\binits{A.~D.}} \AND
\bauthor{\bsnm{Janson},~\bfnm{S.}\binits{S.}}
(\byear{2009}).
\btitle{A functional combinatorial central limit theorem}.
\bjournal{Electron. J. Probab.}
\bvolume{14}
\bpages{2352--2370}.
\bid{doi={10.1214/EJP.v14-709}, issn={1083-6489}, mr={2556014}}
\end{barticle}
%
\bptok{imsref}%
% NOT OUTPUTED:
% issn = 1083-6489
% url = http://dx.doi.org/10.1214/EJP.v14-709
% fjournal = Electronic Journal of Probability
\endbibitem

%b4 ###
%b4 #&#
\bibitem{BeRa84}
%
\begin{barticle}[mr]
\bauthor{\bsnm{Bentkus},~\bfnm{V.~Yu.}\binits{V.~Yu.}} \AND
\bauthor{\bsnm{Rachkauskas},~\bfnm{A.}\binits{A.}}
(\byear{1984}).
\btitle{Estimates for the distance between sums of independent random
elements in {B}anach spaces}.
\bjournal{Teor. Veroyatn. Primen.}
\bvolume{29}
\bpages{49--64}.
\bid{issn={0040-361X}, mr={0739500}}
\end{barticle}
%
\bptok{imsref}%
% NOT OUTPUTED:
% issn = 0040-361X
% number = 1
% fjournal = Akademiya Nauk SSSR. Teoriya Veroyatnoste\u\i i ee
%Primeneniya
\endbibitem

%b5 ###
%b5 #&#
\bibitem{Billingsley1999}
%
\begin{bbook}[mr]
\bauthor{\bsnm{Billingsley},~\bfnm{Patrick}\binits{P.}}
(\byear{1999}).
\btitle{Convergence of Probability Measures},
\bedition{2nd} ed.
%\bseries{Wiley Series in Probability and Statistics: Probability and
%Statistics}.
\bpublisher{Wiley},
\blocation{New York}.
%\bnote{A Wiley-Interscience Publication}.
\bid{doi={10.1002/9780470316962}, mr={1700749}}
\end{bbook}
%
\bptok{imsref}%
% NOT OUTPUTED:
% isbn = 0-471-19745-9
% url = http://dx.doi.org/10.1002/9780470316962
% fpage = x+277
\endbibitem

%b6 ###
%b6 #&#
\bibitem{BrNeSu11}
%
\begin{barticle}[mr]
\bauthor{\bsnm{Broutin},~\bfnm{Nicolas}\binits{N.}},
\bauthor{\bsnm{Neininger},~\bfnm{Ralph}\binits{R.}} \AND
\bauthor{\bsnm{Sulzbach},~\bfnm{Henning}\binits{H.}}
(\byear{2013}).
\btitle{A limit process for partial match queries in random quadtrees
and {$2$}-d trees}.
\bjournal{Ann. Appl. Probab.}
\bvolume{23}
\bpages{2560--2603}.
\bid{doi={10.1214/12-AAP912}, issn={1050-5164}, mr={3127945}}
\end{barticle}
%
\bptok{imsref}%
% NOT OUTPUTED:
% issn = 1050-5164
% url = http://dx.doi.org/10.1214/12-AAP912
% number = 6
% fjournal = The Annals of Applied Probability
\endbibitem

%b7 ###
%b7 #&#
\bibitem{cartan1971}
%
\begin{bbook}[mr]
\bauthor{\bsnm{Cartan},~\bfnm{Henri}\binits{H.}}
(\byear{1971}).
\btitle{Differential Calculus}.
\bpublisher{Hermann},
\blocation{Paris}.
%\bnote{Exercises by C. Buttin, F. Rideau and J. L. Verley, Translated
%from the French}.
\bid{mr={0344032}}
\end{bbook}
%
\bptok{imsref}%
% NOT OUTPUTED:
% fpage = 160
\endbibitem

%b8 ###
%b8 #&#
\bibitem{SulzbachChHw2003}
%
\begin{barticle}[mr]
\bauthor{\bsnm{Chern},~\bfnm{Hua-Huai}\binits{H.-H.}} \AND
\bauthor{\bsnm{Hwang},~\bfnm{Hsien-Kuei}\binits{H.-K.}}
(\byear{2003}).
\btitle{Partial match queries in random quadtrees}.
\bjournal{SIAM J. Comput.}
\bvolume{32}
\bpages{904--915 (electronic)}.
\bid{doi={10.1137/S0097539702412131}, issn={0097-5397}, mr={2001889}}
\end{barticle}
%
\bptok{imsref}%
% NOT OUTPUTED:
% issn = 0097-5397
% url = http://dx.doi.org/10.1137/S0097539702412131
% number = 4
% fjournal = SIAM Journal on Computing
\endbibitem

%b9 ###
%b9 #&#
\bibitem{SulzbachCuJo2010}
%
\begin{barticle}[mr]
\bauthor{\bsnm{Curien},~\bfnm{Nicolas}\binits{N.}} \AND
\bauthor{\bsnm{Joseph},~\bfnm{Adrien}\binits{A.}}
(\byear{2011}).
\btitle{Partial match queries in two-dimensional quadtrees: A~probabilistic approach}.
\bjournal{Adv. in Appl. Probab.}
\bvolume{43}
\bpages{178--194}.
\bid{doi={10.1239/aap/1300198518}, issn={0001-8678}, mr={2761153}}
\end{barticle}
%
\bptok{imsref}%
% NOT OUTPUTED:
% issn = 0001-8678
% url = http://dx.doi.org/10.1239/aap/1300198518
% number = 1
% coden = AAPBBD
% fjournal = Advances in Applied Probability
\endbibitem

%b10 ###
%b10 #&#
\bibitem{dieudonne}
%
\begin{bbook}[mr]
\bauthor{\bsnm{Dieudonn{\'e}},~\bfnm{J.}\binits{J.}}
(\byear{1960}).
\btitle{Foundations of Modern Analysis}.
\bseries{Pure and Applied Mathematics}
\bvolume{10}.
\bpublisher{Academic Press},
\blocation{New York}.
\bid{mr={0120319}}
\end{bbook}
%
\bptok{imsref}%
% NOT OUTPUTED:
% fpage = xiv+361
\endbibitem

%b11 ###
%b11 #&#
\bibitem{do51}
%
\begin{barticle}[mr]
\bauthor{\bsnm{Donsker},~\bfnm{Monroe~D.}\binits{M.~D.}}
(\byear{1951}).
\btitle{An invariance principle for certain probability limit theorems}.
\bjournal{Mem. Amer. Math. Soc.}
\bvolume{6}
\bpages{12}.
\bid{issn={0065-9266}, mr={0040613}}
\end{barticle}
%
\bptok{imsref}%
% NOT OUTPUTED:
% issn = 0065-9266
% number = 6
% fjournal = Memoirs of the American Mathematical Society
\endbibitem

%b12 ###
%b12 #&#
\bibitem{DrJaNe08}
%
\begin{barticle}[mr]
\bauthor{\bsnm{Drmota},~\bfnm{Michael}\binits{M.}},
\bauthor{\bsnm{Janson},~\bfnm{Svante}\binits{S.}} \AND
\bauthor{\bsnm{Neininger},~\bfnm{Ralph}\binits{R.}}
(\byear{2008}).
\btitle{A functional limit theorem for the profile of search trees}.
\bjournal{Ann. Appl. Probab.}
\bvolume{18}
\bpages{288--333}.
\bid{doi={10.1214/07-AAP457}, issn={1050-5164}, mr={2380900}}
\end{barticle}
%
\bptok{imsref}%
% NOT OUTPUTED:
% issn = 1050-5164
% url = http://dx.doi.org/10.1214/07-AAP457
% number = 1
% fjournal = The Annals of Applied Probability
\endbibitem

%b13 ###
%b13 #&#
\bibitem{EiRu07}
%
\begin{barticle}[mr]
\bauthor{\bsnm{Eickmeyer},~\bfnm{Kord}\binits{K.}} \AND
\bauthor{\bsnm{R{\"u}schendorf},~\bfnm{Ludger}\binits{L.}}
(\byear{2007}).
\btitle{A limit theorem for recursively defined processes in {$L^ p$}}.
\bjournal{Statist. Decisions}
\bvolume{25}
\bpages{217--235}.
\bid{doi={10.1524/stnd.2007.0901}, issn={0721-2631}, mr={2412071}}
\end{barticle}
%
\bptok{imsref}%
% NOT OUTPUTED:
% issn = 0721-2631
% url = http://dx.doi.org/10.1524/stnd.2007.0901
% number = 3
% fjournal = Statistics \& Decisions. International Journal for
%Statistical Theory and Related Fields
\endbibitem

%%b14 ###
%%b14 #&#
%\bibitem{SulzbachFiBe1974}
%%
%\begin{barticle}[auto:STB|2014/02/12|14:17:21]
%\bauthor{\bsnm{Finkel},~\bfnm{R.~A.}\binits{R.~A.}} \AND
%\bauthor{\bsnm{Bentley},~\bfnm{J.~L.}\binits{J.~L.}}
%(\byear{1974}).
%\btitle{Quad trees, a data structure for retrieval on composite keys}.
%\bjournal{Acta Inform.}
%\bvolume{4}
%\bpages{1--19}.
%\end{barticle}
%%
%\bptok{imsref}%
%\endbibitem

%b15 ###
%b15 #&#
\bibitem{SulzbachFGPR1993}
%
\begin{barticle}[mr]
\bauthor{\bsnm{Flajolet},~\bfnm{Philippe}\binits{P.}},
\bauthor{\bsnm{Gonnet},~\bfnm{Gaston}\binits{G.}},
\bauthor{\bsnm{Puech},~\bfnm{Claude}\binits{C.}} \AND
\bauthor{\bsnm{Robson},~\bfnm{J.~M.}\binits{J.~M.}}
(\byear{1993}).
\btitle{Analytic variations on quadtrees}.
\bjournal{Algorithmica}
\bvolume{10}
\bpages{473--500}.
\bid{doi={10.1007/BF01891833}, issn={0178-4617}, mr={1244619}}
\end{barticle}
%
\bptok{imsref}%
% NOT OUTPUTED:
% issn = 0178-4617
% url = http://dx.doi.org/10.1007/BF01891833
% number = 6
% coden = ALGOEJ
% fjournal = Algorithmica. An International Journal in Computer Science
\endbibitem

%%b16 ###
%%b16 #&#
%\bibitem{SulzbachFlSe2009}
%%
%\begin{bbook}[mr]
%\bauthor{\bsnm{Flajolet},~\bfnm{Philippe}\binits{P.}} \AND
%\bauthor{\bsnm{Sedgewick},~\bfnm{Robert}\binits{R.}}
%(\byear{2009}).
%\btitle{Analytic Combinatorics}.
%\bpublisher{Cambridge Univ. Press},
%\blocation{Cambridge}.
%\bid{doi={10.1017/CBO9780511801655}, mr={2483235}}
%\end{bbook}
%%
%\bptok{imsref}%
%% NOT OUTPUTED:
%% isbn = 978-0-521-89806-5
%% url = http://dx.doi.org/10.1017/CBO9780511801655
%% fpage = xiv+810
%\endbibitem

%b17 ###
%b17 #&#
\bibitem{GiLe1980}
%
\begin{barticle}[mr]
\bauthor{\bsnm{Gin{\'e}},~\bfnm{Evarist}\binits{E.}} \AND
\bauthor{\bsnm{Le{\'o}n},~\bfnm{Jos{\'e}~R.}\binits{J.~R.}}
(\byear{1980}).
\btitle{On the central limit theorem in {H}ilbert space}.
\bjournal{Stochastica}
\bvolume{4}
\bpages{43--71}.
\bid{issn={0210-7821}, mr={0573725}}
\end{barticle}
%
\bptok{imsref}%
% NOT OUTPUTED:
% issn = 0210-7821
% number = 1
% fjournal = Stochastica
\endbibitem

%b18 ###
%b18 #&#
\bibitem{jankai12}
%
\begin{bmisc}[auto:STB|2014/02/12|14:17:21]
\bauthor{\bsnm{Janson},~\bfnm{S.}\binits{S.}} \AND
\bauthor{\bsnm{Kaijser},~\bfnm{S.}\binits{S.}}
(\byear{2014}).
\bhowpublished{Higher moments of Banach space valued random
variables. \textit{Mem. Amer. Math. Soc.} To appear.}
\end{bmisc}
%
\bptok{imsref}%
% NOT OUTPUTED:
% sortkey = Janson(2012
\endbibitem

%b19 ###
%b19 #&#
\bibitem{JaNe08}
%
\begin{barticle}[mr]
\bauthor{\bsnm{Janson},~\bfnm{Svante}\binits{S.}} \AND
\bauthor{\bsnm{Neininger},~\bfnm{Ralph}\binits{R.}}
(\byear{2008}).
\btitle{The size of random fragmentation trees}.
\bjournal{Probab. Theory Related Fields}
\bvolume{142}
\bpages{399--442}.
\bid{doi={10.1007/s00440-007-0110-1}, issn={0178-8051}, mr={2438697}}
\end{barticle}
%
\bptok{imsref}%
% NOT OUTPUTED:
% issn = 0178-8051
% url = http://dx.doi.org/10.1007/s00440-007-0110-1
% number = 3-4
% coden = PTRFEU
% fjournal = Probability Theory and Related Fields
\endbibitem

%%b20 ###
%%b20 #&#
%\bibitem{SulzbachKnuth1998}
%%
%\begin{bbook}[mr]
%\bauthor{\bsnm{Knuth},~\bfnm{Donald~E.}\binits{D.~E.}}
%(\byear{1998}).
%\btitle{The Art of Computer Programming: Sorting and Searching}.
%\bpublisher{Addison-Wesley},
%\blocation{Reading, MA}.
%\bid{mr={3077154}}
%\end{bbook}
%%
%\bptok{imsref}%
%% NOT OUTPUTED:
%% isbn = 0-201-89685-0
%% fpage = xiv+780
%\endbibitem

%b21 ###
%b21 #&#
\bibitem{LeTa}
%
\begin{bbook}[mr]
\bauthor{\bsnm{Ledoux},~\bfnm{Michel}\binits{M.}} \AND
\bauthor{\bsnm{Talagrand},~\bfnm{Michel}\binits{M.}}
(\byear{1991}).
\btitle{Probability in {B}anach Spaces: Isoperimetry and Processes}.
\bseries{Ergebnisse der Mathematik und Ihrer Grenzgebiete (3) [Results
in Mathematics and Related Areas (3)]}
\bvolume{23}.
\bpublisher{Springer},
\blocation{Berlin}.
\bid{mr={1102015}}
\end{bbook}
%
\bptok{imsref}%
% NOT OUTPUTED:
% isbn = 3-540-52013-9
% fpage = xii+480
\endbibitem

%b22 ###
%b22 #&#
\bibitem{Lukacs75}
%
\begin{bbook}[mr]
\bauthor{\bsnm{Lukacs},~\bfnm{Eugene}\binits{E.}}
(\byear{1975}).
\btitle{Stochastic Convergence},
\bedition{2nd} ed.
\bseries{Probability and Mathematical Statistics}
\bvolume{30}.
\bpublisher{Academic Press},
\blocation{New York}.
\bid{mr={0375405}}
\end{bbook}
%
\bptok{imsref}%
% NOT OUTPUTED:
% fpage = xi+200
\endbibitem

%%b23 ###
%%b23 #&#
%\bibitem{SulzbachMahmoud1992a}
%%
%\begin{bbook}[mr]
%\bauthor{\bsnm{Mahmoud},~\bfnm{Hosam~M.}\binits{H.~M.}}
%(\byear{1992}).
%\btitle{Evolution of Random Search Trees}.
%%\bseries{Wiley-Interscience Series in Discrete Mathematics and
%%Optimization}.
%\bpublisher{Wiley},
%\blocation{New York}.
%%\bnote{A Wiley-Interscience Publication}.
%\bid{mr={1140708}}
%\end{bbook}
%%
%\bptok{imsref}%
%% NOT OUTPUTED:
%% isbn = 0-471-53228-2
%% fpage = xii+324
%\endbibitem

%b24 ###
%b24 #&#
\bibitem{Ne01}
%
\begin{barticle}[mr]
\bauthor{\bsnm{Neininger},~\bfnm{Ralph}\binits{R.}}
(\byear{2001}).
\btitle{On a multivariate contraction method for random recursive
structures with applications to {Q}uicksort}.
\bjournal{Random Structures Algorithms}
\bvolume{19}
\bpages{498--524}.
\bnote{Analysis of algorithms (Krynica Morska, 2000)}.
\bid{doi={10.1002/rsa.10010}, issn={1042-9832}, mr={1871564}}
\end{barticle}
%
\bptok{imsref}%
% NOT OUTPUTED:
% issn = 1042-9832
% url = http://dx.doi.org/10.1002/rsa.10010
% number = 3-4
% fjournal = Random Structures \& Algorithms
\endbibitem

%b25 ###
%b25 #&#
\bibitem{Ne04}
%
\begin{bbook}[auto:STB|2014/02/12|14:17:21]
\bauthor{\bsnm{Neininger},~\bfnm{R.}\binits{R.}}
(\byear{2004}).
\btitle{Stochastische Analyse von Algorithmen, Fixpunktgleichungen und
Ideale Metriken}.
\bpublisher{Univ. Frankfurt},
\blocation{Habilitation}.
\bnote{\url{http://www.math.uni-frankfurt.de/\textasciitilde neiningr/habil.pdf}.}
\end{bbook}
%
\bptok{imsref}%
\endbibitem

%b26 ###
%b26 #&#
\bibitem{NeRu04}
%
\begin{barticle}[mr]
\bauthor{\bsnm{Neininger},~\bfnm{Ralph}\binits{R.}} \AND
\bauthor{\bsnm{R{\"u}schendorf},~\bfnm{Ludger}\binits{L.}}
(\byear{2004}).
\btitle{A general limit theorem for recursive algorithms and
combinatorial structures}.
\bjournal{Ann. Appl. Probab.}
\bvolume{14}
\bpages{378--418}.
\bid{doi={10.1214/aoap/1075828056}, issn={1050-5164}, mr={2023025}}
\end{barticle}
%
\bptok{imsref}%
% NOT OUTPUTED:
% issn = 1050-5164
% url = http://dx.doi.org/10.1214/aoap/1075828056
% number = 1
% fjournal = The Annals of Applied Probability
\endbibitem

%b27 ###
%b27 #&#
\bibitem{NeRu04b}
%
\begin{barticle}[mr]
\bauthor{\bsnm{Neininger},~\bfnm{Ralph}\binits{R.}} \AND
\bauthor{\bsnm{R{\"u}schendorf},~\bfnm{Ludger}\binits{L.}}
(\byear{2004}).
\btitle{On the contraction method with degenerate limit equation}.
\bjournal{Ann. Probab.}
\bvolume{32}
\bpages{2838--2856}.
\bid{doi={10.1214/009117904000000171}, issn={0091-1798}, mr={2078559}}
\end{barticle}
%
\bptok{imsref}%
% NOT OUTPUTED:
% issn = 0091-1798
% url = http://dx.doi.org/10.1214/009117904000000171
% number = 3B
% coden = APBYAE
% fjournal = The Annals of Probability
\endbibitem

%b28 ###
%b28 #&#
\bibitem{Pestman94}
%
\begin{barticle}[mr]
\bauthor{\bsnm{Pestman},~\bfnm{Wiebe~R.}\binits{W.~R.}}
(\byear{1995}).
\btitle{Measurability of linear operators in the {S}korokhod topology}.
\bjournal{Bull. Belg. Math. Soc. Simon Stevin}
\bvolume{2}
\bpages{381--388}.
\bid{issn={1370-1444}, mr={1355827}}
\end{barticle}
%
\bptok{imsref}%
% NOT OUTPUTED:
% issn = 1370-1444
% url = http://projecteuclid.org/euclid.bbms/1103408695
% number = 4
% fjournal = Bulletin of the Belgian Mathematical Society. Simon Stevin
\endbibitem

%b29 ###
%b29 #&#
\bibitem{RaRu95}
%
\begin{barticle}[mr]
\bauthor{\bsnm{Rachev},~\bfnm{S.~T.}\binits{S.~T.}} \AND
\bauthor{\bsnm{R{\"u}schendorf},~\bfnm{L.}\binits{L.}}
(\byear{1995}).
\btitle{Probability metrics and recursive algorithms}.
\bjournal{Adv. in Appl. Probab.}
\bvolume{27}
\bpages{770--799}.
\bid{doi={10.2307/1428133}, issn={0001-8678}, mr={1341885}}
\end{barticle}
%
\bptok{imsref}%
% NOT OUTPUTED:
% issn = 0001-8678
% url = http://dx.doi.org/10.2307/1428133
% number = 3
% coden = AAPBBD
% fjournal = Advances in Applied Probability
\endbibitem

%b30 ###
%b30 #&#
\bibitem{Ro91}
%
\begin{barticle}[mr]
\bauthor{\bsnm{R{\"o}sler},~\bfnm{Uwe}\binits{U.}}
(\byear{1991}).
\btitle{A limit theorem for ``{Q}uicksort''}.
\bjournal{RAIRO Inform. Th\'eor. Appl.}
\bvolume{25}
\bpages{85--100}.
\bid{issn={0988-3754}, mr={1104413}}
\end{barticle}
%
\bptok{imsref}%
% NOT OUTPUTED:
% issn = 0988-3754
% number = 1
% fjournal = RAIRO Informatique Th\'eorique et Applications.
%Theoretical Informatics and Applications
\endbibitem

%b31 ###
%b31 #&#
\bibitem{Ro92}
%
\begin{barticle}[mr]
\bauthor{\bsnm{R{\"o}sler},~\bfnm{Uwe}\binits{U.}}
(\byear{1992}).
\btitle{A fixed point theorem for distributions}.
\bjournal{Stochastic Process. Appl.}
\bvolume{42}
\bpages{195--214}.
\bid{doi={10.1016/0304-4149(92)90035-O}, issn={0304-4149}, mr={1176497}}
\end{barticle}
%
\bptok{imsref}%
% NOT OUTPUTED:
% issn = 0304-4149
% url = http://dx.doi.org/10.1016/0304-4149(92)90035-O
% number = 2
% coden = STOPB7
% fjournal = Stochastic Processes and their Applications
\endbibitem

%b32 ###
%b32 #&#
\bibitem{Ro99}
%
\begin{barticle}[mr]
\bauthor{\bsnm{R{\"o}sler},~\bfnm{U.}\binits{U.}}
(\byear{2001}).
\btitle{On the analysis of stochastic divide and conquer algorithms}.
\bjournal{Algorithmica}
\bvolume{29}
\bpages{238--261}.
\bnote{Average-case analysis of algorithms (Princeton, NJ, 1998)}.
\bid{doi={10.1007/BF02679621}, issn={0178-4617}, mr={1887306}}
\bptnote{check year}%
\end{barticle}
%
\bptok{imsref}%
% NOT OUTPUTED:
% issn = 0178-4617
% url = http://dx.doi.org/10.1007/BF02679621
% number = 1-2
% coden = ALGOEJ
% fjournal = Algorithmica. An International Journal in Computer Science
\endbibitem

%b33 ###
%b33 #&#
\bibitem{RoRu01}
%
\begin{barticle}[mr]
\bauthor{\bsnm{R{\"o}sler},~\bfnm{U.}\binits{U.}} \AND
\bauthor{\bsnm{R{\"u}schendorf},~\bfnm{L.}\binits{L.}}
(\byear{2001}).
\btitle{The contraction method for recursive algorithms}.
\bjournal{Algorithmica}
\bvolume{29}
\bpages{3--33}.
%\bnote{Average-case analysis of algorithms (Princeton, NJ, 1998)}.
\bid{doi={10.1007/BF02679611}, issn={0178-4617}, mr={1887296}}
\end{barticle}
%
\bptok{imsref}%
% NOT OUTPUTED:
% issn = 0178-4617
% url = http://dx.doi.org/10.1007/BF02679611
% number = 1-2
% coden = ALGOEJ
% fjournal = Algorithmica. An International Journal in Computer Science
\endbibitem

%%b34 ###
%%b34 #&#
%\bibitem{SulzbachSamet1990}
%%
%\begin{bbook}[auto:STB|2014/02/12|14:17:21]
%\bauthor{\bsnm{Samet},~\bfnm{H.}\binits{H.}}
%(\byear{1990}).
%\btitle{The Design and Analysis of Spatial Data Structures}.
%\bpublisher{Addison-Wesley},
%\blocation{Reading}.
%\end{bbook}
%%
%\bptok{imsref}%
%% NOT OUTPUTED:
%% sortkey = Samet(1990
%\endbibitem
%
%%b35 ###
%%b35 #&#
%\bibitem{SulzbachSamet1990a}
%%
%\begin{bbook}[auto:STB|2014/02/12|14:17:21]
%\bauthor{\bsnm{Samet},~\bfnm{H.}\binits{H.}}
%(\byear{1990}).
%\btitle{Applications of Spatial Data Structures: Computer Graphics,
%Image Processing, and GIS}.
%\bpublisher{Addison-Wesley},
%\blocation{Reading}.
%\end{bbook}
%%
%\bptok{imsref}%
%% NOT OUTPUTED:
%% sortkey = Samet(1990
%\endbibitem
%
%%b36 ###
%%b36 #&#
%\bibitem{SulzbachSamet2006}
%%
%\begin{bbook}[auto:STB|2014/02/12|14:17:21]
%\bauthor{\bsnm{Samet},~\bfnm{H.}\binits{H.}}
%(\byear{2006}).
%\btitle{Foundations of Multidimensional and Metric Data Structures}.
%\bpublisher{Morgan Kaufmann},
%\blocation{San Mateo}.
%\end{bbook}
%%
%\bptok{imsref}%
%% NOT OUTPUTED:
%% sortkey = Samet(2006
%\endbibitem

%b37 ###
%b37 #&#
\bibitem{DuSt}
%
\begin{bincollection}[mr]
\bauthor{\bsnm{Strassen},~\bfnm{Volker}\binits{V.}} \AND
\bauthor{\bsnm{Dudley},~\bfnm{R.~M.}\binits{R.~M.}}
(\byear{1969}).
\btitle{The central limit theorem and {$\varepsilon$}-entropy}.
In \bbooktitle{Probability and {I}nformation {T}heory ({P}roc.
{I}nternat. {S}ympos., {M}c{M}aster {U}niv., {H}amilton, {O}nt., 1968)}
\bpages{224--231}.
\bpublisher{Springer},
\blocation{Berlin}.
\bid{mr={0279872}}
\end{bincollection}
%
\bptok{imsref}%
\endbibitem

%b38 ###
%b38 #&#
\bibitem{SuDiss}
%
\begin{bmisc}[auto:STB|2014/02/12|14:17:21]
\bauthor{\bsnm{Sulzbach},~\bfnm{H.}\binits{H.}}
(\byear{2012}).
\bhowpublished{On a functional contraction method. Dissertation, Univ.
Frankfurt.
Available at \surl{http://publikationen.ub.uni-frankfurt.de/frontdoor/\\index/index/docId/24858}.}
%\href{http://urn:nbn:de:hebis:30:3-248587}{urn:nbn:de:hebis:30:3-248587}.}
\end{bmisc}
%
\bptok{imsref}%
\endbibitem

%b39 ###
%b39 #&#
\bibitem{Zolo76}
%
\begin{barticle}[mr]
\bauthor{\bsnm{Zolotarev},~\bfnm{V.~M.}\binits{V.~M.}}
(\byear{1976}).
\btitle{Approximation of the distributions of sums of independent
random variables with values in infinite-dimensional spaces}.
\bjournal{Teor. Veroyatn. Primen.}
\bvolume{21}
\bpages{741--758}.
\bid{issn={0040-361X}, mr={0517338}}
\end{barticle}
%
\bptok{imsref}%
% NOT OUTPUTED:
% issn = 0040-361x
% number = 4
% fjournal = Akademija Nauk SSSR. Teorija Verojatnoste\u\i i ee
%Primenenija
\endbibitem

%b40 ###
%b40 #&#
\bibitem{Zolo78}
%
\begin{barticle}[mr]
\bauthor{\bsnm{Zolotarev},~\bfnm{V.~M.}\binits{V.~M.}}
(\byear{1976}).
\btitle{Metric distances in spaces of random variables and of their
distributions}.
\bjournal{Mat. Sb.}
\bvolume{101(143)}
\bpages{416--454, 456}.
\bid{mr={0467869}}
\bptnote{check year}%
\end{barticle}
%
\bptok{imsref}%
% NOT OUTPUTED:
% number = 3
\endbibitem

%b41 ###
%b41 #&#
\bibitem{Zolo77}
%
\begin{barticle}[mr]
\bauthor{\bsnm{Zolotarev},~\bfnm{V.~M.}\binits{V.~M.}}
(\byear{1977}).
\btitle{Ideal metrics in the problem of approximating the distributions
of sums of independent random variables}.
\bjournal{Teor. Veroyatn. Primen.}
\bvolume{22}
\bpages{449--465}.
\bid{issn={0040-361X}, mr={0455066}}
\end{barticle}
%
\bptok{imsref}%
% NOT OUTPUTED:
% issn = 0040-361x
% number = 3
% fjournal = Akademija Nauk SSSR. Teorija Verojatnoste\u\i i ee
%Primenenija
\endbibitem

%b42 ###
%b42 #&#
\bibitem{Zolo79}
%
\begin{barticle}[mr]
\bauthor{\bsnm{Zolotarev},~\bfnm{V.~M.}\binits{V.~M.}}
(\byear{1979}).
\btitle{Ideal metrics in the problems of probability theory and
mathematical statistics}.
\bjournal{Austral. J. Statist.}
\bvolume{21}
\bpages{193--208}.
\bid{issn={0004-9581}, mr={0561947}}
\end{barticle}
%
\bptok{imsref}%
% NOT OUTPUTED:
% issn = 0004-9581
% number = 3
% coden = AUJSA3
% fjournal = The Australian Journal of Statistics
\endbibitem

\end{thebibliography}
\end{document}